\documentclass[a4paper,twoside,11pt,english,intlimits]{article}

\usepackage[utf8]{inputenc}
\usepackage[T2A,T1]{fontenc}

\usepackage[english]{babel}

\usepackage{amsmath,amsthm,amssymb,stmaryrd}
\allowdisplaybreaks
\usepackage{mathtools}
\usepackage[ttscale=.875]{libertine}
\usepackage[libertine]{newtxmath}
 \usepackage{mathrsfs}%
\usepackage{bbm}

\usepackage{float}

\usepackage{setspace}
\usepackage{enumitem}
\usepackage{fancyhdr,titlesec,url}
\usepackage[all,knot,poly]{xy}
\usepackage{array}
\usepackage{hyperref}
\usepackage{graphicx}
\usepackage[numbers,sort&compress]{natbib}
\usepackage{multirow}
\usepackage{ifthen}
\usepackage{time}
\usepackage{algorithm2e}
\usepackage{caption}
\usepackage{subcaption}

\usepackage{etoolbox}

\usepackage{cancel}
\usepackage{tikz}
\usetikzlibrary{automata, positioning}

\usepackage[capitalize,noabbrev]{cleveref}

\DeclareSymbolFont{cyrillic}{T2A}{cmr}{m}{n}
\DeclareMathSymbol{\Sha}{\mathalpha}{cyrillic}{216}

\DeclareFontFamily{U}{mathb}{\hyphenchar\font45}
\DeclareFontShape{U}{mathb}{m}{n}{
<-6> mathb5 <6-7> mathb6 <7-8> mathb7
<8-9> mathb8 <9-10> mathb9
<10-12> mathb10 <12-> mathb12
}{}

\DeclareFontFamily{U}{BOONDOX-calo}{\skewchar\font=45 }
\DeclareFontShape{U}{BOONDOX-calo}{m}{n}{
  <-> s*[1.05] BOONDOX-r-calo}{}
\DeclareFontShape{U}{BOONDOX-calo}{b}{n}{
  <-> s*[1.05] BOONDOX-b-calo}{}
\DeclareMathAlphabet{\mathcalb}{U}{BOONDOX-calo}{m}{n}
\SetMathAlphabet{\mathcalb}{bold}{U}{BOONDOX-calo}{b}{n}

\newcommand{\ltwocell}[3][0.5]{\ar@{}[#2] \ar@{=>}?(#1)+/r 0.175cm/;?(#1)+/l 0.175cm/^{#3}}
\newcommand{\ltwocello}[3][0.5]{\ar@{}[#2] \ar@{=>}?(#1)+/r 0.175cm/;?(#1)+/l 0.175cm/_{#3}}
\newcommand{\dtwocell}[3][0.5]{\ar@{}[#2] \ar@{=>}?(#1)+/u  0.175cm/;?(#1)+/d 0.175cm/^{#3}}

\newcounter{Assume}
\newcommand{\twocong}[2][0.5]{\ar@{}[#2] \save ?(#1)*{\cong}\restore}
\newcommand{\twoeq}[2][0.5]{\ar@{}[#2] \save ?(#1)*{=}\restore}
\newcommand{\threeeq}[2][0.5]{\ar@{}[#2] \save ?(#1)*{\equiv}\restore}

\newcommand{\horaru}[3][0.5]{\ar@{->}[#3]^(#1){#2}\ar@{}[#3]|(#1){\scriptstyle{\bullet}} }
\newcommand{\horard}[3][0.5]{\ar@{->}[#3]_(#1){#2}\ar@{}[#3]|(#1){\scriptstyle{\bullet}} }
\newcommand{\horarm}[3][0.5]{\ar@{->}[#3]|(#1){#2}\ar@{}[#3]|(0.3){\scriptstyle{\bullet}} }

\newcommand{\circaruh}[2]{\ar@{->}[#2]^{#1}\ar@{}[#2]|-{\scriptstyle{\ominus}} }
\newcommand{\circardh}[2]{\ar@{->}[#2]_{#1}\ar@{}[#2]|-{\scriptstyle{\ominus}} }
\newcommand{\circarmh}[2]{\ar@{->}[#2]|-{#1}\ar@{}[#2]|-{\scriptstyle{\ominus}} }

\newcommand{\horto}[2]{\SelectTips{eu}{10}
\xymatrix@C=.2in{#1\horaru{}{r} & #2}}
\newcommand{\circto}[2]{\SelectTips{eu}{10}
\xymatrix@C=.2in{#1\circaruh{}{r} & #2}}

\makeatletter
\DeclareMathSymbol{\b@r}{\mathord}{largesymbols}{"0E}
\newcommand{\newc@ncel}[2]{%
  \ooalign{%
    \hfil$\vcenter{\moved@wn{#1}\hbox{\scalebox{1}[0.5]{$#1\b@r$}}}$\hfil\cr % the bar
    $#1#2$\cr % the symbol
  }%
}
\newcommand{\moved@wn}[1]{%
  \sbox\z@{$#1\mkern-3mu\nonscript\mkern3mu$}%
  \vskip\wd\z@
}
\newcommand{\nprecnapprox}{%
  \mathrel{\m@th\mathpalette\newc@ncel\precapprox}%
}
\newcommand{\nsuccnapprox}{%
  \mathrel{\m@th\mathpalette\newc@ncel\succapprox}%
}
\makeatother

\newcommand{\periodafter}[1]{\ifstrempty{#1}{}{#1.}}
\titleformat{\section}[block]{\scshape\filcenter\LARGE\boldmath}{\thesection.}{.5em}{}
\titleformat{\subsection}[block]{\bfseries\filcenter\large\boldmath}{\thesubsection.}{.5em}{\medskip}
\titleformat{\subsubsection}[runin]{\bfseries\boldmath}{\thesubsubsection.}{.5em}{\periodafter}%{}[.]
\titlespacing{\subsubsection}{0pt}{\topsep}{.5em}

\newtheoremstyle{ntheorem}%
	{\topsep}{\topsep}{\itshape}{0pt}{\bfseries}{.}{.5em}%
	{\thmnumber{#2.\hspace{.5em}}\thmname{#1}\thmnote{ (#3)}}
	
\newtheoremstyle{ndefinition}%
	{\topsep}{\topsep}{\normalfont}{0pt}{\bfseries}{.}{.5em}%
	{\thmnumber{#2.\hspace{.5em}}\thmname{#1}\thmnote{ (#3)}}
	
\newtheoremstyle{nremark}%
	{\topsep}{\topsep}{\normalfont}{0pt}{\itshape}{.}{.5em}%
	{\thmnumber{}\thmname{#1}\thmnote{ (#3)}}
	
\newtheoremstyle{nfigure}%
	{\topsep}{\topsep}{\itshape}{0pt}{\bfseries}{.}{.5em}%
	{\thmnumber{#2.\hspace{.5em}}\thmname{#1}.\thmnote{ #3}}
	
\theoremstyle{ntheorem}
  	\newtheorem{theorem}[subsubsection]{Theorem}
  	\newtheorem{proposition}[subsubsection]{Proposition}
	\newtheorem{lemma}[subsubsection]{Lemma}
  	\newtheorem{corollary}[subsubsection]{Corollary}

\theoremstyle{nfigure}

\theoremstyle{ndefinition}

	\newtheorem{example}[subsubsection]{Example}
	\newtheorem{remark}[subsubsection]{Remark}

\makeatletter
\def\@equationname{equation}
\newenvironment{eqn}[1]{%
    \def\mymathenvironmenttouse{#1}%
    \ifx\mymathenvironmenttouse\@equationname%
        \refstepcounter{subsubsection}%
    \else
        \patchcmd{\@arrayparboxrestore}{equation}{subsubsection}{}{}%          doesn't change output?
        \patchcmd{\print@eqnum}{equation}{subsubsection}{}{}%
        \patchcmd{\incr@eqnum}{equation}{subsubsection}{}{}%
    \fi
    \csname\mymathenvironmenttouse\endcsname%
}{%
    \ifx\mymathenvironmenttouse\@equationname%
        \tag{\thesubsubsection}%
    \fi
    \csname end\mymathenvironmenttouse\endcsname%
}
\makeatother

\makeatletter
\newenvironment{eq*}{%
 \incr@eqnum
  $
}{%
  $
  \ignorespacesafterend
}
\makeatother

\newcounter{DefEq}
\fancypagestyle{plain}{
\fancyhf{}\fancyfoot[LC,RC]{}
\fancyfoot[LE,RO]{$\thepage$}

}

\UseTips
\SelectTips{eu}{11}

\newdir{ >}{{}*!/-10pt/@{>}}
\newdir{ -}{{}*!/-10pt/@{}}
\newdir{> }{{}*!/+10pt/@{>}}

\makeatletter

\xyletcsnamecsname@{dir4{}}{dir{}}
\xydefcsname@{dir4{-}}{\line@ \quadruple@\xydashh@}
\xydefcsname@{dir4{.}}{\point@ \quadruple@\xydashh@}
\xydefcsname@{dir4{~}}{\squiggle@ \quadruple@\xybsqlh@}
\xydefcsname@{dir4{>}}{\Tttip@}
\xydefcsname@{dir4{<}}{\reverseDirection@\Tttip@}

\xydef@\quadruple@#1{%
	\edef\Drop@@{%
		\dimen@=#1\relax
		\dimen@=.5\dimen@
		\A@=-\sinDirection\dimen@
		\B@=\cosDirection\dimen@
		\setboxz@h{%
			\setbox2=\hbox{\kern3\A@\raise3\B@\copy\z@}%
			\dp2=\z@ \ht2=\z@ \wd2=\z@ \box2
			\setbox2=\hbox{\kern\A@\raise\B@\copy\z@}%
			\dp2=\z@ \ht2=\z@ \wd2=\z@ \box2
			\setbox2=\hbox{\kern-\A@\raise-\B@\copy\z@}%
			\dp2=\z@ \ht2=\z@ \wd2=\z@ \box2
			\setbox2=\hbox{\kern-3\A@\raise-3\B@ \noexpand\boxz@}%
			\dp2=\z@ \ht2=\z@ \wd2=\z@ \box2
		}%
		\ht\z@=\z@ \dp\z@=\z@ \wd\z@=\z@ \noexpand\styledboxz@
	}%
}

\xydef@\Tttip@{\kern2pt \vrule height2pt depth2pt width\z@
	\Tttip@@ \kern2pt \egroup
	\U@c=0pt \D@c=0pt \L@c=0pt \R@c=0pt \Edge@c={\circleEdge}%
	\def\Leftness@{.5}\def\Upness@{.5}%
	\def\Drop@@{\styledboxz@}\def\Connect@@{\straight@{\dottedSpread@\jot}}}
	
\xydef@\Tttip@@{%
	\dimen@=.25\dimen@
 	\B@=\cosDirection\dimen@
	\setboxz@h\bgroup\reverseDirection@\line@ \wdz@=\z@ \ht\z@=\z@ \dp\z@=\z@
	{\vDirection@(1,-1)\xydashl@ \xyatipfont\char\DirectionChar}%
	{\vDirection@(1,+1)\xydashl@ \xybtipfont\char\DirectionChar}%
}

\xydef@\ar@form{
	\ifx \space@\next \expandafter\DN@\space{\xyFN@\ar@form}%
	\else\ifx ^\next \DN@ ^{\xyFN@\ar@style}\edef\arvariant@@{\string^}%
	\else\ifx _\next \DN@ _{\xyFN@\ar@style}\edef\arvariant@@{\string_}%
	\else\ifx 0\next \DN@ 0{\xyFN@\ar@style}\def\arvariant@@{0}%
	\else\ifx 1\next \DN@ 1{\xyFN@\ar@style}\def\arvariant@@{1}%
	\else\ifx 2\next \DN@ 2{\xyFN@\ar@style}\def\arvariant@@{2}%
	\else\ifx 3\next \DN@ 3{\xyFN@\ar@style}\def\arvariant@@{3}%
	\else\ifx 4\next \DN@ 4{\xyFN@\ar@style}\def\arvariant@@{4}%
	\else\ifx \bgroup\next \let\next@=\ar@style
	\else\ifx [\next \DN@[##1]{\ar@modifiers{[##1]}}%]
	\else\ifx *\next \DN@ *{\ar@modifiers}%
	\else\addLT@\ifx\next \let\next@=\ar@slide
	\else\ifx /\next \let\next@=\ar@curveslash
	\else\ifx (\next \let\next@=\ar@curveinout %)
	\else\addRQ@\ifx\next \addRQ@\DN@{\ar@curve@}%
	\else\addLQ@\ifx\next \addLQ@\DN@{\xyFN@\ar@curve}%
	\else\addDASH@\ifx\next \addDASH@\DN@{\defarstem@-\xyFN@\ar@}%
	\else\addEQ@\ifx\next \addEQ@\DN@{\def\arvariant@@{2}\defarstem@-\xyFN@\ar@}%
	\else\addDOT@\ifx\next \addDOT@\DN@{\defarstem@.\xyFN@\ar@}%
	\else\ifx :\next \DN@:{\def\arvariant@@{2}\defarstem@.\xyFN@\ar@}%
	\else\ifx ~\next \DN@~{\defarstem@~\xyFN@\ar@}%
	\else\ifx !\next \DN@!{\dasharstem@\xyFN@\ar@}%
	\else\ifx ?\next \DN@?{\ar@upsidedown\xyFN@\ar@}%
	\else \let\next@=\ar@error
	\fi\fi\fi\fi\fi\fi\fi\fi\fi\fi\fi\fi\fi\fi\fi\fi\fi\fi\fi\fi\fi\fi\fi \next@}

\makeatother

\newcommand{\fl}{\rightarrow}
\newcommand{\fll}{\longrightarrow}
\newcommand{\ofl}[1]{\overset{\displaystyle #1}{\fll}}

\newcommand{\pfl}{\twoheadrightarrow}

\newcommand{\dfl}{\Rightarrow}

\newcommand{\tfl}{\Rrightarrow}
\newcommand{\qfl}{\xymatrix@1@C=10pt{\ar@4 [r] &}}

\newcommand{\equivmap}{\xymatrix@1@C=10pt{\ar @{|-|} [r] &}}

\newcommand{\cl}[1]{\overline{#1}}
\newcommand{\rep}[1]{\widehat{#1}}

\DeclareMathOperator{\id}{Id}

\renewcommand{\epsilon}{\varepsilon}

\newcommand{\Ar}{\mathbb{A}}

\newcommand{\Cr}{\mathcal{C}}

\newcommand{\Tr}{\mathcal{T}}

\def\catego#1{\mathsf{#1}}

\newcommand{\Cat}{\catego{Cat}}
\newcommand{\Pol}{\catego{Pol}}

\newcommand{\Set}{\catego{Set}}

\newcommand{\Vect}{\catego{Vect}}

\newcommand{\Alg}{\catego{Alg}}

\newcommand{\OSet}{\operatorname{\Omega\text{-}\catego{Set}}}
\newcommand{\OVect}{\operatorname{\Omega\text{-}\catego{Vect}}}

\newcommand{\OAlg}[1]{\Omega\text{-}\catego{Alg}_{#1}}
\newcommand{\Glob}{\catego{Glob}}

\newcommand{\Bimod}{\catego{Bimod}}

\newcommand{\Mon}{\catego{Mon}}

\newcommand{\OMon}{\operatorname{\Omega\text{-}\catego{Mon}}}

\newcommand{\nc}{\newcommand}

\newcommand{\RBA}[2]{\mathcal{R}\mathcal{B}_{#1}(#2)}
\newcommand{\DRBA}[2]{\mathcal{D}\mathcal{R}\mathcal{B}_{#1}(#2)}
\newcommand{\DA}[2]{\mathcal{D}_{#1}(#2)}

\DeclareMathOperator{\Nf}{Nf}

\DeclareMathOperator{\supp}{supp}

\newcommand{\lin}[1]{\mathscr{A}_\Omega(#1)}
\newtheorem{prop-def}{Proposition-Definition}[section]
\theoremstyle{definition}

\newcommand{\tfll}{\xymatrix@1@C=20pt{\ar@3 [r] &}}
\newcommand{\otfl}[1]{\xymatrix@1@C=20pt{\ar@3 [r] ^-*+{#1} &}}

\renewcommand{\theenumi}{\roman{enumi}}

\renewcommand{\theenumii}{\alph{enumii}}

\nc{\delete}[1]{{}}
\nc{\mmargin}[1]{}
\nc{\mlabel}[1]{\label{#1}}  % Use this to suppress names
\nc{\mcite}[1]{\cite{#1}}  % Use this to suppress names
\nc{\mref}[1]{\ref{#1}}  % Use this to suppress names
\nc{\mbibitem}[1]{\bibitem{#1}} % Use this to show number
\delete{
	\nc{\mlabel}[1]{\label{#1}  % Use the next two lines to show names
		{\hfill \hspace{1cm}{\bf{{\ }\hfill(#1)}}}}
	\nc{\mcite}[1]{\cite{#1}{{\bf{{\ }(#1)}}}}  % Use this lines to show names
	\nc{\mref}[1]{\ref{#1}{{\bf{{\ }(#1)}}}}  % Use this lines to show names
	\nc{\mbibitem}[1]{\bibitem[\bf #1]{#1}} % Use this to show name
}

 \font\cyrs=wncyr7

\newcommand{\bk}{\Bbbk}
\nc{\vep}{\varepsilon}
\nc{\bin}[2]{ (_{\stackrel{\scs{#1}}{\scs{#2}}})}  %binomial coeff
\nc{\binc}[2]{(\!\! \begin{array}{c} \scs{#1}\\
		\scs{#2} \end{array}\!\!)}  %binomial coeff
\nc{\bincc}[2]{  ( {\scs{#1} \atop
		\vspace{-1cm}\scs{#2}} )}  %binomial coeff
\nc{\oline}[1]{\overline{#1}}
\nc{\mapm}[1]{\lfloor\!|{#1}|\!\rfloor}
\nc{\bs}{\bar{S}}
\nc{\la}{\longrightarrow}
\nc{\rar}{\rightarrow}
\nc{\lon }{\,\rightarrow\,}
\nc{\dar}{\downarrow}
\nc{\dap}[1]{\downarrow \rlap{$\scriptstyle{#1}$}}
\nc{\defeq}{\stackrel{\rm def}{=}}
\nc{\dis}[1]{\displaystyle{#1}}
\nc{\dotcup}{\ \displaystyle{\bigcup^\bullet}\ }
\nc{\hcm}{\ \hat{,}\ }
\nc{\hts}{\hat{\otimes}}
\nc{\hcirc}{\hat{\circ}}
\nc{\lleft}{[}
\nc{\lright}{]}
\nc{\curlyl}{\left \{ \begin{array}{c} {} \\ {} \end{array}
	\right .  \!\!\!\!\!\!\!}
\nc{\curlyr}{ \!\!\!\!\!\!\!
	\left . \begin{array}{c} {} \\ {} \end{array}
	\right \} }
\nc{\longmid}{\left | \begin{array}{c} {} \\ {} \end{array}
	\right . \!\!\!\!\!\!\!}
\nc{\ora}[1]{\stackrel{#1}{\rar}}
\nc{\ola}[1]{\stackrel{#1}{\la}}
\nc{\scs}[1]{\scriptstyle{#1}} \nc{\mrm}[1]{{\rm #1}}
\nc{\dirlim}{\displaystyle{\lim_{\longrightarrow}}\,}
\nc{\invlim}{\displaystyle{\lim_{\longleftarrow}}\,}
\nc{\dislim}[1]{\displaystyle{\lim_{#1}}} \nc{\colim}{\mrm{colim}}
\nc{\mvp}{\vspace{0.3cm}} \nc{\tk}{^{(k)}} \nc{\tp}{^\prime}
\nc{\ttp}{^{\prime\prime}} \nc{\svp}{\vspace{2cm}}
\nc{\vp}{\vspace{8cm}}
\nc{\modg}[1]{\!<\!\!{#1}\!\!>}
\nc{\intg}[1]{F_C(#1)}
\nc{\lmodg}{\!<\!\!}
\nc{\rmodg}{\!\!>\!}
\nc{\cpi}{\widehat{\Pi}}
\nc{\ssha}{{\mbox{\cyrs X}}} %sha as product
\nc{\tsha}{{\mbox{\cyrt X}}}
\nc{\shpr}{\diamond}    %Shuffle product
\nc{\labs}{\mid\!}
\nc{\rabs}{\!\mid}
 \nc{\zhx}{\text{-}}
\nc{\ad}{\mrm{ad}}
\nc{\ann}{\mrm{ann}}
\nc{\Av}{\mrm{Av}}
\nc{\bim}{\mbox{-}\mathsf{Bimod}}
\nc{\br}{\mrm{bre}}
\nc{\can}{\mrm{can}}
\nc{\Cont}{\mrm{Cont}}
\nc{\rchar}{\mrm{char}}
\nc{\cok}{\mrm{coker}}
\nc{\db}{\mrm{db}}
\nc{\de}{\mrm{dep}}
\nc{\dgg}{\mrm{dgg}}
\nc{\dgp}{\mrm{dgp}}
\nc{\dgx}{\mrm{dgx}}
\nc{\Dif}{\mrm{Diff}}
\nc{\dtf}{{R-{\rm tf}}}
\nc{\dtor}{{R-{\rm tor}}}

\nc{\Div}{{\mrm Div}}
\nc{\Diff}{\mrm{DA}}
\nc{\Diffl}{\mathsf{DA}_\lambda}
\nc{\diffo}{{\mathsf{DO}_\lambda}}
\nc{\dl}{{\mathrm{PD}}}
\nc{\dRB}{{\mathrm{\Phi}_\mathsf{DRB}}}
\nc{\udRB}{{\mathrm{\Phi}_\mathsf{uDRB}}}
\nc{\OdRB}{{\mathrm{\Phi}_\mathsf{DRB}^0}}
 \nc{\OudRB}{{\mathrm{\Phi}_\mathsf{uDRB}^0}}
\nc{\inte}{{\mathrm{\Phi}_\mathsf{ID}}}
\nc{\uinte}{{\mathrm{\Phi}_\mathsf{uID}}}
\nc{\Ointe}{{\mathrm{\Phi}_\mathsf{ID}^0}}
 \nc{\Ouinte}{{\mathrm{\Phi}_\mathsf{uID}^0}}

\nc{\alg}{\mathsf{Alg}}
\nc{\Fil}{\mrm{Fil}}
\nc{\Frob}{\mrm{Frob}}
\nc{\Gal}{\mrm{Gal}}
\nc{\GL}{\mrm{GL}}
\nc{\Hoch}{\mrm{Hoch}}
\nc{\hsr}{\mrm{H}}
\nc{\hpol}{\mrm{HP}}
\nc{\im}{\mrm{im}}
\nc{\Id}{\mrm{Id}}
\nc{\ID}{\mrm{ID}}
\nc{\Irr}{\mrm{Irr}}
\nc{\incl}{\mrm{incl}}
\nc{\length}{\mrm{length}}
\nc{\NLSW}{\mrm{NLSW}}
\nc{\Nij}{\mrm{Nij}}
\nc{\mchar}{\rm char}
\nc{\mpart}{\mrm{part}}
\nc{\ql}{{\QQ_\ell}}
\nc{\qp}{{\QQ_p}}
\nc{\rank}{\mrm{rank}}
\nc{\rcot}{\mrm{cot}}
\nc{\rdef}{\mrm{def}}
\nc{\rdiv}{{\rm div}}
\nc{\Rey}{\mrm{Rey}}
\nc{\rtf}{{\rm tf}}
\nc{\rtor}{{\rm tor}}
\nc{\res}{\mrm{res}}
\nc{\SL}{\mrm{SL}}
\nc{\Spec}{\mrm{Spec}}
\nc{\tor}{\mrm{tor}}
\nc{\tr}{\mrm{tr}}
\nc{\wt}{\mrm{wt}}

\nc{\udl}{{\mathrm{udl}}}

\nc{\bfk}{{\bf k}}
\nc{\bfone}{{\bf 1}}
\nc{\bfzero}{{\bf 0}}
\nc{\detail}{\marginpar{\bf More detail}
	\noindent{\bf Need more detail!}
	\svp}
\nc{\gap}{\marginpar{\bf Incomplete}\noindent{\bf Incomplete!!}
	\svp}
\nc{\FMod}{\mathbf{FMod}}
\nc{\Int}{\mathbf{Int}}
\nc{\remarks}{\noindent{\bf Remarks: }}

\nc{\BA}{{\mathbb A}}   \nc{\CC}{{\mathbb C}}
\nc{\DD}{{\mathbb D}}   \nc{\EE}{{\mathbb E}}
\nc{\FF}{{\mathbb F}}  
\nc{\HH}{{\mathbb H}}   \nc{\LL}{{\mathbb L}}
\nc{\NN}{{\mathbb N}}   \nc{\PP}{{\mathbb P}}
\nc{\QQ}{{\mathbb Q}}   \nc{\RR}{{\mathbb R}}
\nc{\TT}{{\mathbb T}}   \nc{\VV}{{\mathbb V}}
\nc{\ZZ}{{\mathbb Z}}   \nc{\TP}{\widetilde{P}}

\nc{\cala}{{\mathcal A}}    \nc{\calc}{{\mathcal C}}
\nc{\cald}{\mathcal{D}}     \nc{\cale}{{\mathcal E}}
\nc{\calf}{{\mathcal F}}    \nc{\calg}{{\mathcal G}}
\nc{\calh}{{\mathcal H}}    \nc{\cali}{{\mathcal I}}
\nc{\call}{{\mathcal L}}    \nc{\calm}{{\mathcal M}}
\nc{\caln}{{\mathcal N}}    \nc{\calo}{{\mathcal O}}
\nc{\calp}{{\mathcal P}}    \nc{\calr}{{\mathcal R}}
\nc{\cals}{{\mathcal S}}    \nc{\calt}{{\Omega}}
\nc{\calv}{{\mathcal V}}    \nc{\calw}{{\mathcal W}}
\nc{\calx}{{\mathcal X}}    \nc{\calu}{{\mathcal U}}
\nc{\caly}{{\mathcal Y}}

\nc{\uOpAlg}{{\mathfrak{uOpAlg}}}

\nc{\OpAlg}{{\mathfrak{OpAlg}}}

\nc{\ComOpAlg}{{\mathfrak{ComOpAlg}}}

\nc{\OpVect}{{\mathfrak{OpVect}}}

\nc{\OpSet}{{\mathfrak{OpSet}}}

\nc{\OpMon}{{\mathfrak{OpMon}}}

\nc{\ComOpMon}{{\mathfrak{ComOpMon}}}

\nc{\OpSem}{{\mathfrak{OpSem}}}

\nc{\ComOpSem}{{\mathfrak{ComOpSem}}}

\nc{\uAlg}{{\mathfrak{uAlg}}}

\nc{\ComAlg}{{\mathfrak{ComAlg}}}

\nc{\ComMon}{{\mathfrak{ComMon}}}

\nc{\mtOpSet}{{\Omega\zhx\mathfrak{Set}}}

\nc{\mtOpSem}{{\Omega\zhx\mathfrak{Sem}}}

\nc{\mtOpMon}{{\Omega\zhx\mathfrak{Mon}}}

\nc{\mtOpVect}{{\Omega\zhx\mathfrak{Vect}}}

\nc{\mtOpAlg}{{\Omega\zhx\mathfrak{Alg}}}

\nc{\mtuOpAlg}{{\Omega\zhx\mathfrak{uAlg}}}

\nc{\ComSem}{\mathfrak{ComSem}}
\nc{\fraka}{{\mathfrak a}}
\nc{\frakb}{\mathfrak{b}}
\nc{\frakg}{{\frak g}}
\nc{\frakl}{{\frak l}}
\nc{\fraks}{{\frak s}}
\nc{\frakB}{{\frak B}}
\nc{\frakm}{{\frak m}}
\nc{\frakM}{{\frak M}}
\nc{\frakp}{{\frak p}}
\nc{\frakW}{{\frak W}}
\nc{\frakX}{{\frak X}}
\nc{\frakS}{{\frak{S}}}
\nc{\frakA}{{\frak A}}
\nc{\frakx}{{\frakx}}

\nc{\frakMstar}{{\mathfrak{M}_\Omega^\star}}
\nc{\frakSstar}{{\mathfrak{S}_\Omega^\star}}

\nc{\lir}[1]{\textcolor{red}{\underline{Li:}#1 }}

\newcommand{\brca}[1]{\left\lfloor #1 \right\rfloor}%%operaors
\newcommand{\brck}[2]{\left\lfloor #1 \right\rfloor_{#2}}%%operaors

\newcommand{\LO}{$\Omega$}
\DeclareMathOperator{\lm}{lm}

\renewcommand{\theenumi}{\roman{enumi}}%(

\xyoption{rotate,color}

\makeatletter
\DeclareRobustCommand
  \myvdots{\vbox{\baselineskip4\p@ \lineskiplimit\z@
    \hbox{.}\hbox{.}\hbox{.}}}
\makeatother

\newcommand{\sm}{\scriptstyle}

\newcommand{\pullbackcorner}[1][dr]{\save*!/#1-1.6pc/#1:(-1,1)@^{|-}\restore}

\newcount\hh
\newcount\mm
\mm=\time
\hh=\time
\divide\hh by 60
\divide\mm by 60
\multiply\mm by 60
\mm=-\mm
\advance\mm by \time
\def\hhmm{\number\hh:\ifnum\mm<10{}0\fi\number\mm}
\newcommand{\auteur}[3]{
\noindent
\begin{minipage}[t]{.45\textwidth}
\begin{flushright}
\textsc{#1} \\
{\footnotesize\textsf{#2}}
\end{flushright} 
\end{minipage}
\qquad
\begin{minipage}[t]{.45\textwidth}
#3
\end{minipage}
}

\begin{document}
\thispagestyle{empty}

\begin{center}

% Title
\begin{doublespace}
\begin{huge}
{\scshape Polygraphic resolutions for operated algebras}
\end{huge}

\vskip+2pt

\bigskip
\hrule height 1.5pt 
\bigskip

\vskip+5pt

% Auteurs
\begin{Large}
{\scshape Zuan Liu - Philippe Malbos}
\end{Large}
\end{doublespace}

%%%
%\vfill

%%%
\vskip+25pt

\begin{small}\begin{minipage}{14cm}
\noindent\textbf{Abstract --}
This paper introduces the structure of operated polygraphs as a categorical model for rewriting in operated algebras, generalizing Gröbner--Shirshov bases with non-monomial termination orders. 
We show how to construct polygraphic resolutions of free operated algebras from confluent and terminating presentations. We provide a combinatorial description of critical branchings of operated polygraphs using the structure of polyautomata introduced in this paper. Polyautomata are linear polygraphs equipped with an operator structure formalized by a pushdown automaton. Finally, we apply our constructions to several families of operated algebras, including Rota--Baxter algebras, differential algebras, and differential Rota--Baxter algebras.
\medskip

\smallskip\noindent\textbf{Keywords --} 
Operated algebras, higher-dimensional rewriting,  Gröbner--Shirshov bases, automata, polygraphic resolutions.

\smallskip\noindent\textbf{M.S.C. 2020 --} 
12H05, 16Z10, 18N30, 68Q42, 18G10.
\end{minipage}
\end{small}
\end{center}

\section{Introduction}
An \emph{operated algebra}, also called an \emph{$\Omega$-algebra}, is an associative algebra with linear operators. 
For example, \emph{differential algebras}, introduced by Ritt~\cite{Ritt1950,Ritt1932} in the theory of differential equations, are operated algebras, with an operator satisfying the Leibniz rule.
Differential algebras have been developed in many areas, such as differential Galois theory~\cite{Magid1994,MariusMichael} and differential algebraic group theory~\cite{Kolchin1973}.
\emph{Rota--Baxter algebras} form another important class of operated algebras, introduced by Baxter \cite{Baxter1960}; see also~\cite{Rota1969, Cartier1972, Atkinson1963}. 
They generalize the algebra of continuous functions through integral operators. 
These algebras have found applications across various fields of mathematics and physics, including renormalization in quantum field theory~\cite{Guo10}, the analysis of Volterra integral equations~\cite{GGL22}, Hopf algebras~\cite{ConnesKreimer2000}, Yang--Baxter equations~\cite{Bai2007}, and Rota--Baxter Lie algebras~\cite{GLS21}.
\emph{Differential Rota--Baxter algebras} are equipped with both differential and integral operators. 
Introduced by Guo and Keigher~\cite{GKeigher08}, they describe the relationship between such operators satisfying the first fundamental theorem of calculus. 
The homological study of operated algebras has been the subject of several works.
In particular, homological properties such as Koszul duality, minimal models, deformations, and cohomology theory have been explored in \cite{SongWangZhang,WangZhou24,ChenGuoWangZhou} for Rota--Baxter and differential algebras.
These algebras have also been studied from a computational perspective, using the theory of Gröbner--Shirshov bases~\cite{BokutChenQiu,GuoSitZhang,LQQZ,GaoGuoSit,LiGuo}.

\bigskip

This paper introduces a new rewriting method to compute resolutions of operated algebras. 
Rewriting is a computational model widely used in algebra, logic and computer science. It consists of calculating in an equational theory by generating the set of equalities by means of a system of \emph{rewriting rules}, and sequentially applying the rules that replace the subterms of a formula by other terms.
In algebra, rewriting theory enables the calculation of linear bases~\cite{Buchberger65,Ufnarovski95} and resolutions of algebraic structures in abelian and categorical settings~\cite{Anick86,Kobayashi90,Kobayashi05,Groves90,Brown92,GuiraudMalbos12advances}. For operated algebras, \emph{Gröbner--Shirshov bases}, GS bases for short, introduced in \cite{Shirshov,Buchberger06},  are frequently used alongside rewriting methods for the standardization of ideal representations in free operated algebras, yielding linear bases and other related constructions \cite{BokutChenQiu,GGR15,GuoSitZhang,QQWZ21}. GS methods require a \emph{monomial order} compatible with the rewriting rules, ensuring the termination of calculations. Nevertheless, finding a monomial order for an operated algebra is more complicated  than in the associative case due to the nesting of operators in monomials~\cite{GGR15,LQQZ,BokutChenQiu}. 

The aim of this paper is twofold: first, it introduces a new termination method in operated algebras without relying on monomial orders, and second, it introduces the notion of \emph{polygraphic resolutions} for operated algebras. 
The linear rewriting strategies introduced in \cite{GuiraudHoffbeckMalbos19} allow one to compute resolutions of an associative algebra by extending its rewriting presentation into a polygraphic resolution generated by confluence diagrams of the higher overlapping of rules.
We extend this approach to operated algebras by introducing the structure of \emph{$\Omega$-polygraphs} as a rewriting setting for operated algebras. It can be used to derive their linear bases and resolutions.
This construction is part of a line of research on constructing polygraphic resolutions of algebraic structures from their convergent presentations, e.g. for monoids and categories \cite{GuiraudMalbos18,GuiraudMalbos12advances}, associative algebras \cite{GuiraudHoffbeckMalbos19} and operads \cite{MalbosRen23}. 

\bigskip

We now present the main results of this paper. 
The first part consists of setting up the polygraphic framework for $\Omega$-algebras.
In Section~\ref{S:HigherOperatedAlgebras}, we  introduce the structure of \emph{higher $\Omega$-algebras}, hereafter called \emph{$\omega$-algebras}, as internal $\omega$-categories in the category of $\Omega$-algebras.
Using the underlying linear and operator structures, Proposition~\ref{P:IsomorphismOmegaAlgebrasBimodules} characterizes $\omega$-algebras in terms of globular bimodules over $\Omega$-algebras.
Following this characterization, Subsection~\ref{SS:Polygraphs for operated algebras} introduces the notion of \emph{$\Omega$-polygraphs} as systems of generators and relations for presentations of $\omega$-algebras. This construction is an extension of the structure of linear polygraphs introduced for associative algebras in \cite{GuiraudHoffbeckMalbos19}.

In Section~\ref{S:OperatedRewriting}, we develop rewriting systems for $\Omega$-algebras based on $\Omega$-$1$-polygraphs. 
We define rewriting properties such as \emph{termination} and \emph{confluence} on $\Omega$-$1$-polygraphs. We show that this polygraphic approach allows us to define more general termination orders than those used in the theory of GS bases; see Remark~\ref{E:NoMonomialOrder}.
Proposition~\ref{P:DerivationTerminating} proves termination of an $\Omega$-$1$-polygraph using the method of \emph{derivations} introduced in \cite{Guiraud06jpaa,GuiraudMalbos09}. Proposition~\ref{P:LinearBasisConvergentPolygraphs} states that for \emph{convergent} $\Omega$-$1$-polygraphs, which are both confluent and terminating, the set of \emph{normal forms} is a linear basis for the presented algebra. 
Finally, we classify the \emph{local branchings} of $\Omega$-$1$-polygraphs, and state the coherent critical branching lemma in Theorem~\ref{T:CriticalBranchingLemma}. In Subsection~\ref{SS:GSConvergence}, we make explicit the relationship between GS bases of $\Omega$-algebras and convergent \LO-1-polygraphs. 

Section~\ref{S:PolygraphicResolutions} presents the main result of this paper. 
Starting with an $\Omega$-algebra presented by a reduced convergent left-monomial \LO-1-polygraph~$X$, Theorem~\ref{T:Main conclusion} constructs its polygraphic resolution \emph{à la Squier} $\mathrm{Sq}(X)$.
The $0$-generators and $1$-generators of the $\omega$-polygraph $\mathrm{Sq}(X)$ consist of the variables and the rewriting rules defining the algebra, respectively. For $n\geq 2$, the $n$-generators are the sources of the critical $n$-branchings of the polygraph~$X$.  An associative algebra being an $\Omega$-algebra with $\Omega$ empty, our polygraphic resolution generalizes the corresponding construction for associative algebras~\cite{GuiraudHoffbeckMalbos19}.

Due to possible nesting of operators, the critical branchings of \LO-$1$-polygraphs cannot be described in terms of string overlaps; see Remark~\ref{R:DifferencePolygraphAnd OmegaPolygraph}. 
In order to interpret these critical branchings as critical branchings of linear $1$-polygraphs, in Section~\ref{S:Polyautomata}, we describe the operator structures using \emph{pushdown automata (PDAs)}. 
Explicitly,  we construct a PDA $\Ar_\Omega$ accepting all monomials of the $\Omega$-algebras defined with respect to an indexed set $\Omega$. Given a linear 1-polygraph $\Sigma$ and a PDA $\Ar$, we define a \emph{$1$-polyautomaton} as a pair $(\Sigma, \Ar)$, where both $s(\alpha)$ and $t(\alpha)$ are linear combinations of monomials accepted by the PDA $\Ar$, for every $\alpha \in \Sigma_1$.  Theorem~\ref{T:OmegaCBshape} interprets critical branchings in \LO-$1$-polygraphs in terms of critical branchings of linear $1$-polygraphs.
It thus provides a simpler description of the higher critical branchings involved in polygraphic resolutions, and yields an algorithmic method for constructing such resolutions.

In Section~\ref{S:ExampleOfResolutionOfOmegaAlgebras}, we apply Theorem~\ref{T:Main conclusion} to construct polygraphic resolutions for Rota--Baxter algebras, differential algebras, and differential Rota--Baxter algebras. To this end, we construct reduced and convergent \LO-1-polygraphs for these $\Omega$-algebras by defining  derivations to prove termination and applying the critical branchings theorem to prove confluence.

\subsubsection*{Organization of the paper}
Section~\ref{S:HigherOperatedAlgebras} recalls the main constructions of $\Omega$-algebras used in this work. 
It also introduces the structure of higher $\Omega$-algebras and \LO-polygraphs.
Section~\ref{S:OperatedRewriting} deals with the rewriting properties of \LO-1-polygraphs, including the construction of derivations for termination proofs, the operated version of the coherent critical branchings lemma, and comparisons with Gröbner--Shirshov bases.
Section~\ref{S:PolygraphicResolutions} introduces polygraphic resolutions of $\Omega$-algebras. 
Given an $\Omega$-$\omega$-polygraph, we characterize its property of being a resolution by the existence of a contraction defined on its generators. We then show how to construct such a contraction starting with a reduced convergent $\Omega$-$1$-polygraph.
Section~\ref{S:Polyautomata} introduces the structure of polyautomata, encoding both the associative linear structure and the operator structure to establish a connection between \LO-$1$-polygraphs and linear $1$-polygraphs.
This correspondence allows us to interpret the critical branchings of $\Omega$-$1$-polygraphs in terms of those of polygraphs of associative algebras.
Finally, Section~\ref{S:ExampleOfResolutionOfOmegaAlgebras} presents constructions of polygraphic resolutions for some classical $\Omega$-algebras.

\subsubsection*{Conventions and notations}
Throughout this paper, we fix a field $\bk$ of characteristic zero and an element $\lambda$ in $\bk$.
Unless stated otherwise, all algebras in this paper are assumed to be associative and unital.

\section{Higher operated algebras and operated polygraphs}
\label{S:HigherOperatedAlgebras}
This section introduces the notion of higher $\Omega$-algebras and the structure of \LO-polygraphs as systems of generators and relations for  higher $\Omega$-algebras. 
Operated algebraic structures were defined in \cite{QQWZ21}. 
We first recall the notion of $\Omega$-object in a category and the construction of free $\Omega$-algebras from \cite{BokutChenQiu, Guo09}.

\subsection{Operated algebras}

We denote by
$\Set, \Mon, \Vect$ and $\Alg{}{}$
the categories of sets, monoids, vector spaces, and unital associative algebras, respectively.
\subsubsection{Operated objects}
\label{SSS:OmegaObjects}

Let $\Cr$ be a \emph{concrete category} equipped with a faithful forgetful functor
$
\calu:\Cr\to\Set.
$
An \emph{$\Omega$-object} in $\Cr$ is an object $X$ of $\Cr$ equipped with a family $\Tr=(\Tr_\tau)_{\tau\in\Omega}$ of maps $\Tr_\tau:\calu(X)\to \calu(X)$, called \emph{operators}.
When $\Cr=\Vect$ or $\Cr=\Alg{}$, we additionally require the operators $\Tr_\tau$ to be $\bk$-linear.
A \emph{morphism of $\Omega$-objects} from $(X,\Tr)$ to $(X',\Tr')$ is a morphism $f:X\fl X'$ in $\Cr$ such that 
$
\calu(f)\circ\Tr_\tau
=
\Tr'_\tau\circ \calu(f),
$
for every $\tau\in\Omega$.
The $\Omega$-objects and their morphisms form a category, denoted by $\Omega\text{-}\Cr$.

In this paper, we will consider $\OSet$, $\OMon$, $\OVect$, and $\OAlg{}{}$ as the categories of $\Omega$-objects in the categories $\Set$, $\Mon$, $\Vect$, and $\Alg{}$, respectively.
Note that, when $\Omega$ is empty,  the categories $\Cr$ and $\Omega\text{-}\Cr$ are isomorphic. 
Following \cite{QQWZ21}, we similarly have the following diagram of free functors:
\[
\scalebox{0.9}{$
\xymatrix@R=1pc @C=1pc{
& \OAlg{}{} \ar@<.5ex>@{<-}[rr]
& & \OVect \\
\OMon \ar@<.5ex>@{->}[ur]
       \ar@<.5ex>@{<-}[rr]
& & \OSet \ar@<.5ex>@{->}[ur]
\\
& \Alg \ar@<.5ex>@{-->}[uu]
        \ar@<.5ex>@{<--}[rr]
& & \Vect \ar@<.5ex>@{->}[uu]
\\
\Mon \ar@<.5ex>@{->}[uu]
     \ar@<.5ex>@{-->}[ur]
     \ar@<.5ex>@{<-}[rr]
& & \Set \ar@<.5ex>@{->}[uu]
          \ar@<.5ex>@{->}[ur]
}
$}
\]
These functors are left adjoint to the corresponding forgetful functors.
We now make explicit the following free $\Omega$-algebra functor:
\[
\Set \longrightarrow \Mon \longrightarrow \OMon \longrightarrow \OAlg.
\]

\subsubsection{Free $\Omega$-algebras}
\label{SSS:FreeOmegaObjects}
The \emph{free $\Omega$-set} on a set $Z$ is given by $\Omega^\ast\times Z$,
where $\Omega^\ast$ is the free monoid on~$\Omega$.
Its operators are defined, for every $\tau\in\Omega$, by
$
\Tr_\tau(w,z)\coloneq (\tau w,z),
$
for all $w\in\Omega^\ast$ and $z\in Z$.

The \emph{free $\Omega$-monoid} $Z^\Omega$ on $Z$ is given as follows. 
Following~\cite{QQWZ21}, we write $\Omega\times Z$ as
\[
\lfloor Z \rfloor_\Omega \:\coloneq \bigsqcup_{\tau \in \Omega} \lfloor Z \rfloor_\tau,
\]
where $\lfloor Z \rfloor_\tau$ denotes the set of formal symbols $\lfloor z \rfloor_\tau$ with $z\in Z$.
We set $Z^\Omega_0\coloneq Z^\ast$ and $Z^\Omega_1\coloneq (Z\sqcup   \lfloor Z^\Omega_0 \rfloor_\Omega)^\ast$. 
The inclusion $Z \hookrightarrow Z\sqcup \lfloor Z^\Omega_0 \rfloor_\Omega$ induces an injective morphism
$i_{0,1}: Z^\Omega_0 \hookrightarrow Z^\Omega_1$.
For $1\leq k \leq n-1$, suppose $Z^\Omega_{k}$ has been constructed with injective morphisms $i_{k-1,k}: Z^\Omega_{k-1} \hookrightarrow Z^\Omega_{k}$.
We then set
\[
Z^\Omega_{n} \:\coloneq (Z\sqcup \lfloor Z^\Omega_{n-1} \rfloor_\Omega)^\ast.
\]
The inclusion $Z\sqcup \lfloor Z^\Omega_{n-2} \rfloor_\Omega\hookrightarrow Z\sqcup \lfloor Z^\Omega_{n-1} \rfloor_\Omega$  also induces an injective morphism
$i_{n-1,n}: Z^\Omega_{n-1} \hookrightarrow Z^\Omega_n$.
By construction, we have $Z^\Omega_i\subset Z^\Omega_{i+1}$ for $i\geq 0$, and we define $Z^\Omega\coloneq\varinjlim Z^\Omega_{n}$. The maps sending $u\in Z^\Omega_{n}$ to $\lfloor u \rfloor_{\tau} \in Z^\Omega_{n+1}$ induce a family of operators $\brck{~}{}=(\lfloor~\rfloor_\tau)_{\tau\in\Omega}$ on $Z^\Omega$.
We show below that $(Z^{\Omega},\lfloor\,\rfloor)$ is the free $\Omega$-monoid on $Z$.

If $(X,\Tr)$ is an $\Omega$-monoid and
$f:Z^\ast\to X$ is a monoid morphism, we prove that there exists a unique
$\Omega$-monoid morphism $\tilde f:Z^\Omega\to X$ whose restriction to
$Z^\ast$ is $f$.
It suffices to construct a compatible family of monoid
morphisms $\tilde f_i:Z^\Omega_i\to X$, for $i\geq 0$. Set
$\tilde f_0=f$. Assume that $\tilde f_k:Z^\Omega_k\to X$ has been
constructed for some $k\geq 0$. Since $Z^\Omega_{k+1}$ is the free
monoid on $Z\sqcup \lfloor Z^\Omega_k\rfloor_\Omega$, there exists a
unique monoid morphism $\tilde f_{k+1}:Z^\Omega_{k+1}\to X$ such that
$\tilde f_{k+1}(z)=f(z)$ and $\tilde f_{k+1}(\lfloor v\rfloor_\tau)
=
\Tr_\tau(\tilde f_k(v))$,
for all $z\in Z$, $v\in Z^\Omega_k$, and $\tau\in\Omega$.
These morphisms are compatible with the inclusions
$Z^\Omega_k\hookrightarrow Z^\Omega_{k+1}$ and therefore induce a unique
monoid morphism
\[
\tilde f: Z^\Omega = \varinjlim_{k\ge0} Z^\Omega_k \longrightarrow X.
\]
By construction, $\tilde f$ is an $\Omega$-monoid morphism extending
$f$. Therefore, $Z^\Omega$ is the free $\Omega$-monoid on $Z$, and the
construction defines the left adjoint to the forgetful functor
$\OMon\to\Set$.

For $u\in Z^\Omega$, we define $\lfloor u\rfloor_0 \coloneq u$, where $0$ is a new element not included in $\Omega$.
An \emph{$\Omega$-monomial} $u$ is a non-identity  element of $Z^\Omega$, which can be uniquely written as $u = u_1u_2\cdots u_n$, where $u_i\in Z \sqcup \brck{Z^\Omega}{\Omega}$ and $n$ is the \emph{breadth} of $u$, denoted by $\operatorname{bre}(u) = n$.
The \emph{depth} of $u$ is defined by
\[
\operatorname{dep}(u) \:\coloneq \min \left\{n \mid u \in Z^\Omega_n \right\}.
\]

We denote by $(\lin{Z}, \brck{~}{})$ the \emph{free $\Omega$-algebra} on $Z$, defined as the $\bk$-linear span of $\Omega$-monomials in $Z^\Omega$, and whose each operator $\lfloor ~ \rfloor_\tau$ is induced by linearity.
This defines a functor $ \lin{-}:\Set \fl \OAlg{} $, left adjoint to the forgetful functor $\OAlg{} \fl \catego{Set} $.

\begin{remark}
Following \cite{BokutChenQiu}, for every word
$
u=u_1u_2\cdots u_n\in Z^\Omega,
$
where $u_i\in Z\sqcup \lfloor Z^\Omega\rfloor_\Omega$, the monoid product can be represented by the disjoint union
\[
\bullet_{u_1}\sqcup \bullet_{u_2}\sqcup \cdots \sqcup \bullet_{u_n},
\]
where each $\bullet_{u_i}$ is a planar binary tree, called a \emph{forest}.
In such forests,  the action of $\tau\in \Omega$ on $v\in Z^\Omega$,
denoted by $\lfloor v\rfloor_\tau$, can be represented by
\[
\begin{tikzpicture}[baseline=-0.5ex, scale=0.75, every node/.style={inner sep=1pt}]
\node (top) at (0,0.7) {$\bullet$};
\node (bot) at (0,0) {$\bullet$};
\draw (bot) -- (top);
\node[right=1pt] at (top) {$v$};
\node[right=1pt] at (bot) {$\tau$};
\end{tikzpicture}
\]
For instance, the $\Omega$-monomial
$
\lfloor yz\rfloor_{\tau_1}\lfloor x\lfloor yz\rfloor_{\tau_1}\rfloor_{\tau_2}
$
is represented by the following forest
\[
\begin{tikzpicture}[baseline=-0.5ex, scale=0.55, every node/.style={inner sep=1pt}]
% First tree: \lfloor yz\rfloor_{\tau_1}
\node (tau1a) at (0,0) {$\bullet$};
\node (ya) at (-0.9,1.2) {$\bullet$};
\node (za) at (0.9,1.2) {$\bullet$};

\draw (tau1a) -- (ya);
\draw (tau1a) -- (za);

\node[right=1pt] at (tau1a) {$\tau_1$};
\node[right=1pt] at (ya) {$y$};
\node[right=1pt] at (za) {$z$};

% Second tree: \lfloor x\lfloor yz\rfloor_{\tau_1}\rfloor_{\tau_2}
\node (tau2b) at (4,0) {$\bullet$};
\node (xb) at (2.8,2.1) {$\bullet$};
\node (tau1b) at (5.2,1.2) {$\bullet$};
\node (yb) at (4.6,2.1) {$\bullet$};
\node (zb) at (6,2.1) {$\bullet$};

\draw (tau2b) -- (xb);
\draw (tau2b) -- (tau1b);
\draw (tau1b) -- (yb);
\draw (tau1b) -- (zb);

\node[right=1pt] at (tau2b) {$\tau_2$};
\node[right=1pt] at (xb) {$x$};
\node[right=1pt] at (tau1b) {$\tau_1$};
\node[right=1pt] at (yb) {$y$};
\node[right=1pt] at (zb) {$z$};
\end{tikzpicture}\raisebox{0.2cm}{\; ,}
\]
where $x,y,z\in Z$ and $\tau_1,\tau_2\in\Omega$.
\end{remark}

\subsubsection{Free operated bimodules}
\label{SSS:FreeOmegaBimodule}
Let $A$ be an algebra in $\OAlg{}{}$. 
The following construction depends only on the underlying algebra structure of $A$, not on its operators.
We denote by $\Bimod(A)$ the category of $A$-bimodules, and define $\Omega\text{-}\Bimod(A)$ to be the category of $\Omega$-objects in $\Bimod(A)$. 
More explicitly, an \emph{$\Omega$-$A$-bimodule} is an $\Omega$-vector space $(M,\Tr)$ such that $M$ is an $A$-bimodule. 
A \emph{morphism of $\Omega$-$A$-bimodules} $(M,\Tr) \fl (M',\Tr')$ is a morphism of $\Omega$-vector spaces $f:M\to M'$ that is also a morphism of $A$-bimodules.

As in the construction of free $\Omega$-algebras, we construct the \emph{free $\Omega$-$A$-bimodule} on a set $Z$ as follows.
 We set
$\mathcal{M}_0(Z) \coloneq A \otimes_\bk \bk Z \otimes_\bk A$,
the free $A$-bimodule with basis $Z$, where $\bk Z$ denotes the free $\bk$-vector space with basis $Z$, and for each $n \geq 0$,
\[
\mathcal{M}_{n+1}(Z) \: \coloneq \: A\otimes_\bk \bk\brck{\mathcal{M}_{n}(Z)}{\Omega \sqcup \{0\}}\otimes_\bk A,
\]
where $\brck{u}{0}\coloneq u$ for $u\in\mathcal{M}_n(Z)$.
By construction, we have $\mathcal{M}_{i}(Z) \subset \mathcal{M}_{i+1}(Z)$ for $i \geq 0$, and define $\mathcal{M}_{\omega}(Z) \coloneq \varinjlim \mathcal{M}_{n}(Z)$. The maps sending $m \in \mathcal{M}_{n}(Z)$ to $\lfloor m \rfloor_{\tau} \in \mathcal{M}_{n+1}(Z)$ induce a family of operators $\lfloor~\rfloor_\tau$ on $\mathcal{M}_{\omega}(Z)$ indexed by $\tau \in \Omega$.

\subsubsection{Operated contexts}
\label{SSS:OneHole}
Let $\square$ be a symbol not in $Z$.
The set of \emph{(one-hole) $\Omega$-contexts} is the subset of $(Z\sqcup\square)^\Omega$, denoted by $Z^\Omega[\square]$, consisting of $\Omega$-monomials with $\square$ occurring only once. 

For $q\in  Z^\Omega[\square]$ and $u\in Z^\Omega$, we define $q|_{u}\in Z^\Omega$ as the element obtained by replacing the symbol~$\square$ in $q$ with $u$. 
The \emph{composition} of contexts $p$ and $q$ in $Z^\Omega[\square]$ is defined by $q\circ p\coloneq p|_q$, where, similarly, $p|_q$ is obtained by replacing the symbol~$\square$ in $p$ with $q$.
With this composition and identity element $\square$, $Z^\Omega[\square]$ forms an $\Omega$-monoid.
For $q\in Z^\Omega[\square]$ and $a=\sum_i\lambda_iu_i\in \lin{Z}$, with~$\lambda_i\in\bk$ and $u_i\in Z^\Omega$, we define
\[
q|_a\:\coloneq\sum_i\lambda_iq|_{u_i}.
\]

Similarly, we denote by $\mathcal{M}_{\omega}(Z)[\square]$ the set of
\emph{(one-hole) $\Omega$-contexts}, consisting of the basis monomials of
$\mathcal{M}_{\omega}(Z\sqcup\{\square\})$ in which $\square$ occurs exactly once.
Every element of $\mathcal{M}_{\omega}(Z)$ can be written as a linear combination of terms of the form
\[
D|_x \: = \: (C_1\circ C_2 \circ \cdots \circ C_n)|_x,
\]
where $x\in Z$ and $C_k=\left\lfloor a_k\otimes\square \otimes b_k \right\rfloor_{\tau_k}$, with $a_k,b_k\in A$, $\tau_k\in\Omega\sqcup\{0\}$ and  $1\leq k \leq n$. 

\subsubsection{Operated ideals}
An \emph{$\Omega$-ideal} of an $\Omega$-algebra $(A,\Tr)$ is an ideal of the algebra $A$ closed under the action of its operators.
We denote by $I_{\Omega}(S)$ the $\Omega$-ideal of $\lin{Z}$ generated by a subset $S$ of $\lin{Z}$, and by $\lin{Z}/I_{\Omega}(S)$ the quotient $\Omega$-algebra.

If $\Omega$ is empty, $\lin{Z}$ is the free associative algebra, and 
$I_{\Omega}(S)$ consists of all the linear combinations of monomials $\left.p\right|_{s}$,
where $s \in S$ and $p=u\square v \in Z^\Omega[\square]$ with $u, v \in Z^\Omega$.  
If $\Omega$ is not empty, an element in~$I_{\Omega}(S)$ is a linear combination of terms of the form
\[
q|_s \: = \: (p_1\circ p_2 \circ \cdots \circ p_n)|_s,
\]
where $p_k = \left\lfloor u_k \square  v_k\right\rfloor_{\tau_k}$ with  $u_k, v_k \in Z^\Omega$ and $\tau_k \in \Omega\sqcup\{0\}$, for $1\leq k \leq n$.

\subsubsection{Examples}
\label{SSS:ExamplesFreeOmegaAlgebras}
We present examples of classical $\Omega$-algebras, which will be further studied in Section~\ref{S:ExampleOfResolutionOfOmegaAlgebras}.
\begin{enumerate}
\item The \emph{free differential algebra of weight $\lambda$} on a set $Z$ is the free $\Omega$-algebra, denoted by $\DA{\lambda}{Z}$, equipped with an operator $D$ such that
\begin{eqn}{equation}
\label{E:DifferentialRelation}
D(ab) = D(a)b + aD(b) + \lambda D(a)(b) \quad \text{and}\quad  D(1)=0, \quad \text{ for all } a, b \in \DA{\lambda}{Z}.
\end{eqn}
\item The \emph{free Rota--Baxter algebra of weight $\lambda$} on a set $Z$ is the free $\Omega$-algebra, denoted by $\RBA{\lambda}{Z}$, equipped with an operator $P$ such that 
\begin{eqn}{equation}
\label{E:RotaBaxterRelation}
P(a)P(b)= P(aP(b)) + P(P(a)b) + \lambda P(ab), \quad \text{ for all $a,b\in \RBA{\lambda}{Z}$.}
\end{eqn}
\item
The \emph{free differential Rota--Baxter algebra of weight $\lambda$} on a set $	Z$ is the free $\Omega$-algebra, denoted by  $\DRBA{\lambda}{Z}$,  equipped with two  operators  $D$ and $P$  satisfying \eqref{E:DifferentialRelation}, \eqref{E:RotaBaxterRelation} and the relation
\[
D(P(a))=a, \quad \text{ for all } a \in \DRBA{\lambda}{Z}.
\]
\end{enumerate}
We will use $D$ and $P$ to denote the \emph{differential operator} and \emph{Rota--Baxter operator}, respectively, instead of $\brck{~}{D}$ and $\brck{~}{P}$, following the conventions in \cite{QQWZ21, GGR15, BokutChenQiu}.

\subsection{Higher operated algebras}
\label{SS:HigherOperatedAlgebras}
This subsection introduces  higher $\Omega$-algebras as a generalization of higher associative algebras introduced in \cite{GuiraudHoffbeckMalbos19}. 
We also make explicit their structure in terms of operated bimodules.

\subsubsection{Globular objects}
An \emph{(internal) globular object of a category~$\Cr$} is a sequence $X\coloneq(X_k)_{k\geq 0}$ of objects in $\mathcal{C}$, equipped with the following families of morphisms in~$\Cr$ 
\[
(s_k:X_{k+1}\fl X_k)_{k\geq 0},
\qquad
(t_k:X_{k+1}\fl X_k)_{k\geq 0},
\quad\text{and}\quad
(i_k:X_{k-1}\fl X_k)_{k\geq 1},
\]
satisfying the following globular and identity relations
\begin{eqn}{equation}
\label{E:GlobularRelations}
s_ks_{k+1} = s_kt_{k+1}, 
\qquad 
t_ks_{k+1} = t_kt_{k+1}
\qquad\text{and}\qquad
 s_ki_{k+1} = t_ki_{k+1} = \id_{X_k},
\end{eqn}
for every $k\geq 0$.
The elements of~$X_k$ are \emph{$k$-cells of~$X$}.
For $x\in X_k$ with $k\geq 1$, the $(k-1)$-cells~$s_{k-1}(x)$ and~$t_{k-1}(x)$ are the \emph{source} and \emph{target} of $x$, also denoted by $s(x)$ and $t(x)$.
A \emph{morphism of globular objects}~$f : X \fl Y$ is a family $(f_k : X_k \fl Y_k)_{k\geq 0}$ of morphisms in $\Cr$ that commutes with the morphisms $s_k$,~$t_k$ and $i_k$.
We denote by~$\Glob(\Cr)$ the category of globular objects  of~$\Cr$ and their morphisms. 
For~$n\geq 0$, we denote by $\Glob_n(\Cr)$ the full subcategory of $\Glob(\Cr)$ consisting of globular objects indexed up to $n$, and called \emph{$n$-globular objects}. 

In the following, unlike in \eqref{E:GlobularRelations}, iterated compositions are written in diagrammatic order, from left to right.
For a globular object $X$ and $0\leq m\leq k \leq n$, the \emph{$k$-source}, \emph{$k$-target}, and \emph{$k$-identity maps} are defined by iterated composition
\[
s_{n-1}\ldots s_k:X_n \fl X_k,
\quad 
t_{n-1}\ldots t_k:X_n \fl X_k
\quad \text{ and } \quad
i_{m+1}\ldots i_k:X_m \fl X_k,
\]
also denoted respectively by $s_k$, $t_k$ and $i_k$ for short.
As in {\cite[§1.1.3]{GuiraudHoffbeckMalbos19}}, when no confusion can arise, we write $i_k(x)$ or $1_x$, or even just $x$, for any iterated identity of $x$.
For~$k \geq 0$, we denote by $X \star_k X$ the pullback of the morphisms $s_k,t_k : X \fl X_k$, consisting of pairs $(a,b)$ satisfying $t_k(a)=s_k(b)$.
For $n\geq 1$, two $n$-cells $a$ and $b$ are  \emph{parallel} if $s(a)=s(b)$ and $t(a)=t(b)$.   
An \emph{$n$-sphere} of $X$ is a pair of parallel $n$-cells.

\subsubsection{Higher categories}
Let $\Cr$ be a category with fibred products.
An \emph{(internal) $\omega$-category} in $\Cr$ consists of a globular object $X$ of $\Cr$
and 
a \emph{$k$-composition} map $X_l \star_k X_l \to X_l$ for all $0\leq k<l$, along with $k$-source and $k$-target maps
\[
\xymatrix @!C @C=3em @W=2em {
X_k
& X_l,
\ar@<-0.8ex> [l] _-{s_k}
\ar@<0.8ex> [l] ^-{t_k}
}
\]
such that the $2$-globular object
\[
\xymatrix @!C @C=3em @W=2em {
X_j
& X_k
\ar@<-0.8ex> [l] _-{s_j}
\ar@<0.8ex> [l] ^-{t_j}
& X_l
\ar@<-0.8ex> [l] _-{s_k}
\ar@<0.8ex> [l] ^-{t_k}
}
\]
is a 2-category in $\Cr$ for all $j < k < l$.

We denote by~$\Cat_\omega(\Cr)$ the category of $\omega$-categories in~$\Cr$ and, for each $n\geq 0$, by~$\Cat_n(\Cr)$ the category of $n$-categories in~$\Cr$.

\subsubsection{Higher $\Omega$-algebras}
For $n \in \mathbb{N} \sqcup \{\omega\}$, we denote by $\OAlg{n}$ the category $\Cat_n(\OAlg{})$, whose objects are called \emph{$\Omega$-$n$-algebras}, or \emph{$n$-algebras} for short. 

Given an $n$-algebra $(A,\Tr)$, there is a family of operators $\Tr_{k}=(\Tr_{\tau,k})_{\tau\in\Omega}$ on $A_k$.
Since the source, target, and identity maps are morphisms of \LO-algebras, the following relations hold:
\begin{eqn}{equation}
\label{E:OperatorsAndsti}
s(\Tr_{\tau,k}(a)) =\Tr_{\tau,k-1}(s(a)),
\quad
t(\Tr_{\tau,k}(a)) =\Tr_{\tau,k-1}(t(a))
\quad \text{ and } \quad
i(\Tr_{\tau,k}(a)) =\Tr_{\tau,k+1}(i(a)),
\end{eqn}
for all $a\in A_k$ and $\tau\in\Omega$. The third relation can also be written as $1_{\Tr_{\tau,k}(a)} = \Tr_{\tau,k+1}(1_a)$.
From this structure, we deduce the following $\Omega$-algebraic properties.

\begin{proposition}
\label{L:PropertiesOmegaAlgebras}
For an $\omega$-algebra $(A, \Tr)$, the following conditions hold:
\begin{enumerate}
\item\label{1} \label{I:nVect-1} 
For all $0\leq k <n$ and $(a,b)$ in $A\star_k A$, we have $a \star_k b = a - t_k (a) + b$.
\item For all $n \geq 1$, every $n$-cell $a$ of $A$ is invertible with inverse $a^- = s(a) - a + t(a)$.
\item For all $0\leq k <n$ and $(a,b)$ in $A\star_k A$, we have  $\Tr_{\tau,n}(a^{-})=\Tr_{\tau,n}(a)^{-}$ and $\Tr_{\tau,n}(a \star_k b )=\Tr_{\tau,n}(a) \star_k \Tr_{\tau,n}(b)$ for every $\tau\in\Omega$.
\end{enumerate}
\end{proposition}
\begin{proof}
The proofs of {\bf i)} and {\bf ii)} follow from
{\cite[Prop.~1.2.3]{GuiraudHoffbeckMalbos19}}; in particular, they use the
linearity of the $k$-compositions and the neutrality axioms, which are also
satisfied in the operated setting.
Property {\bf iii)} follows directly from \eqref{E:OperatorsAndsti}.
\end{proof}

\subsubsection{Globular operated bimodules}
We denote by $\Glob(\Omega\text{-}\Bimod)$ the category of globular operated bimodules whose objects are pairs $\big(A,M\big)$, where $A$ is an $\Omega$-algebra and $M$ is a globular  $\Omega$-$A$-bimodule.
Let $\Glob_{\text{sub}}(\Omega\text{-}\Bimod)$ denote its \emph{full subcategory} consisting of those pairs $\big(A,M\big)$ satisfying the following two conditions. First, $M_0$ is isomorphic to $A$ as an $\Omega$-$A$-bimodule, where $A$ is endowed with its canonical $\Omega$-$A$-bimodule structure induced by multiplication $
a\cdot x\cdot b \coloneq axb,
$
for all $a,b,x\in A$.
Second, for all $a,b\in M_n$, the following relation holds:
\begin{eqn}{equation}
\label{E:LinearExchangeRelation}
a\cdot s_0(b) + t_0(a)\cdot b - t_0(a)\cdot s_0(b)
\: = \:
s_0(a)\cdot b + a\cdot t_0(b) - s_0(a)\cdot t_0(b).
\end{eqn}

The following result extends the one known for  associative algebras {\cite[Thm.~1.3.3]{GuiraudHoffbeckMalbos19}}.

\begin{proposition}
\label{P:IsomorphismOmegaAlgebrasBimodules}
The categories $\OAlg{\omega}$ and $\Glob_{\text{sub}}(\Omega\text{-}\Bimod)$ are isomorphic.
\end{proposition}

\begin{proof}
The isomorphism is given as follows.
Let $B$ be an $\Omega$-$\omega$-algebra. 
We associate with $B$ the pair $(A,M)$ defined by $A\coloneq B_0$ and $M_n\coloneq B_n$ for $n\geq 0$, where each $M_n$ is regarded as an $\Omega$-$A$-bimodule with its canonical structure induced by the product of $B$.
The source, target and identity maps of $B$ make $M$ into a globular $\Omega$-$A$-bimodule.
Moreover, for all $a,b\in M_n$, the product of the algebra $B_n$ satisfies
\[
ab
=
(a\star_0 t_0(a))(s_0(b)\star_0 b)
=
as_0(b)\star_0 t_0(a)b
=
a\cdot s_0(b)-t_0(a)\cdot s_0(b)+t_0(a)\cdot b,
\]
and symmetrically for the other part of \eqref{E:LinearExchangeRelation}. 
Thus $(A,M)$ is an object of $\Glob_{\text{sub}}(\Omega\text{-}\Bimod)$.

Conversely, let $(A,M)$ be an object of $\Glob_{\text{sub}}(\Omega\text{-}\Bimod)$.
For $a,b\in M_n$, we define the product
\[
ab
\:\coloneq
a\cdot s_0(b)+t_0(a)\cdot b-t_0(a)\cdot s_0(b)
=
s_0(a)\cdot b+a\cdot t_0(b)-s_0(a)\cdot t_0(b),
\]
where the equality follows from \eqref{E:LinearExchangeRelation}.
As in \cite[Thm.~1.3.3]{GuiraudHoffbeckMalbos19}, these products induce an $\omega$-algebra structure on $M$.
Since the source, target and identity maps in $M$ are $\Omega$-linear and compatible with the $\Omega$-operators, this is an $\Omega$-$\omega$-algebra.
\end{proof}

\subsection{Polygraphs for operated algebras}
\label{SS:Polygraphs for operated algebras}
The structure of polygraphs was introduced by Street~\cite{Street76} and Burroni~\cite{Burroni93} as systems of generators and relations for strict $\omega$-categories. We refer to \cite{Polybook2025} for a comprehensive presentation of the structure of polygraphs in rewriting theory and higher category theory. 
More recently, polygraphs have been introduced for presentations of higher associative algebras~\cite{GuiraudHoffbeckMalbos19} and shuffle operads~\cite{MalbosRen23}.

In this subsection, we introduce the structure of polygraphs for presentations of higher $\Omega$-algebras. As in the cases of higher associative algebras and shuffle operads, our polygraphic construction fits into Batanin’s general notion of computad associated with a finitary monad on globular sets~\cite{Batanin98}. In the present setting, this monad is induced by the free-forgetful adjunction associated with the forgetful functor from $\OAlg{n}$ to $\Glob_n(\Set)$, which forgets the $\Omega$-algebraic structure.

\subsubsection{Extended higher $\Omega$-algebras}
For $n \geq 0$, the category $\OAlg{n}^+$ of \emph{extended $n$-algebras} is defined as the pullback of forgetful functors
\[
\xymatrix @=1.75em{
\OAlg{n}^+
	\pullbackcorner
\ar[r]
\ar[d]
&
\Glob_{n+1}(\Set)
\ar[d]^-{\calu_n}
\\
\OAlg{n}
\ar[r] _-{{\calv_n^\Omega}}
&
\Glob_{n}(\Set)
}
\]
where the functor $\calu_n$ forgets the structure of $(n+1)$-cells and the functor $\calv_n^\Omega$ forgets the $\Omega$-algebraic structure. 
Explicitly, an object in~$\OAlg{n}^+$ is a pair~$\big((A,\Tr),X\big)$, where~$(A,\Tr)$ is an $n$-algebra and~$X$ is a \emph{cellular extension} of~$A_n$, that is, a set $X$ equipped with two maps
\[
\xymatrix @=1.5em {
A_n 
& X
\ar@<-0.5ex> [l] _-{s}
\ar@<0.5ex> [l] ^-{t}
}
\]
such that, for every $x \in X$, the pair $(s(x), t(x))$ is an $n$-sphere of $A$. 
A morphism from $\big((A,\Tr), X\big)$ to $\big((B,\Tr'), Y\big)$ in~$\OAlg{n}^+$ consists of a pair $(f, g)$, where $f: (A,\Tr) \to (B,\Tr')$ is a morphism of $n$-algebras, and $g: X \to Y$ is a map that commutes with the source and target maps.

\subsubsection{Free higher $\Omega$-algebras}
For $n\geq 1$, we define $A[X]$ to be the \emph{free $n$-algebra} on an extended $(n-1)$-algebra $\big((A,\Tr),X\big)$ as follows. First, we construct an $(A_0,\Tr_0)$-bimodule $(M, \Tr')$, by setting
\[
M \: = \: \mathcal{M}_{\omega}(X)\oplus A_{n-1},
\]
where $\mathcal{M}_{\omega}(X)$ is the free $(A_0, \Tr_0)$-bimodule on $X$  as defined in \eqref{SSS:FreeOmegaBimodule}, and
\[
\Tr'_\tau(m+c) \: \coloneq \: \left\lfloor m\right \rfloor_\tau+\Tr_{\tau,n-1}(c),
\]
for all $\tau\in\Omega$, $m\in \mathcal{M}_{\omega}(X)$ and $c\in A_{n-1}$. Following \eqref{SSS:OneHole}, the elements of $(M, \Tr')$ are the linear combinations of  $D|_x$ and $(n-1)$-cells $c$ of $(A, \Tr)$, where  $x \in X$ and $D\in \mathcal{M}_{\omega}(X_0)[\square]$.
The source, target and identity
maps $s, t$ and $i$ in $(M, \Tr')$  are defined by 
\begin{eqn}{equation}
\label{E:mapsBimodules}
s(\left.D \right|_{x})  = \left.D \right|_{s(x)}, \quad t(\left.D \right|_{x})  = \left.D \right|_{t(x)}  \quad 
\text{and} \quad
s(c)=t(c) =i(c) =c.
\end{eqn}
We define the $(A_0, \Tr_0)$-bimodule $A[X]_n$ as the quotient of $(M, \Tr')$ by the $(A_0, \Tr_0)$-bimodule ideal generated by elements
\[
\big(a s_0(b)+t_0(a) b-t_0(a) s_0(b)\big)- \big(s_0(a) b+a t_0(b)-s_0(a) t_0(b)\big),
\]
for all $a, b$ in $\mathcal{M}_{\omega}(X)$. 
The source, target and identity maps defined in \eqref{E:mapsBimodules} are compatible with this quotient. Hence, by Proposition~\ref{P:IsomorphismOmegaAlgebrasBimodules}, the $(A_0, \Tr_0)$-bimodule $A[X]_n$ extends $(A, \Tr)$ uniquely into the free $n$-algebra $A[X]$.

\subsubsection{Operated polygraphs}
Let us define the category $\Pol_n(\OAlg{})$ of \emph{\LO-$n$-polygraphs} by induction on $n$.
For $n = 0$, we define $\Pol_0(\OAlg{})$ as the category of sets. The \emph{free $0$-algebra functor} maps a set $X_0$ to the free $\Omega$-algebra $\lin{X_0}$.
For $n\geq1$,
assuming that the category $\Pol_{n-1}(\OAlg{})$ and the free $(n-1)$-algebra functor $\mathscr{A}_{n-1} : \Pol_{n-1}(\OAlg{}) \fl \OAlg{n-1}$ have been constructed, we define the category $\Pol_n(\OAlg{})$ as the pullback
\begin{eqn}{equation}
\label{E:PolygraphFunctor}
\vcenter{\xymatrix @=2.5em{
\Pol_n(\OAlg{})
	\pullbackcorner
\ar[r]
\ar[d] _-{\mathcal{V}_{n-1}}
&
\OAlg{n-1}^+
\ar[d]^{\calw_{n-1}}
\\
\Pol_{n-1}(\OAlg{})
\ar[r]_(.57){\mathscr{A}_{n-1}}
&
\OAlg{n-1}
}}
\end{eqn}
of the functor $\mathscr{A}_{n-1}$ and the functor ${\calw_{n-1}}$ that forgets the cellular extension.
The free $n$-algebra functor is defined as the composition
\[
\mathscr{A}_n : \Pol_n(\OAlg{}) \longrightarrow \OAlg{n-1}^+ \longrightarrow \OAlg{n}
\]
of the functor $\Pol_n(\OAlg{}) \to \OAlg{n-1}^+$ from \eqref{E:PolygraphFunctor}, followed by the functor mapping $(A, X)$ to $A[X]$.
The category $\Pol_\omega(\OAlg{})$ of \emph{$\Omega$-$\omega$-polygraphs} is defined as the limit of the sequence $(\mathcal{V}_n : \Pol_n(\OAlg{})\fl \Pol_{n-1}(\OAlg{}))_{n>0}$ of forgetful functors $\mathcal{V}_{n-1}$ defined by the pullback \eqref{E:PolygraphFunctor}.

Expanding this definition, an \LO-$n$-polygraph is a sequence $X = (X_0, X_1, \ldots, X_n)$ made of a set $X_0$ of~\emph{$0$-generators}, a cellular extension $X_1$ of $\lin{X_0}$, and, for each $1 \leq k \leq n-1$, a cellular extension $X_{k+1}$ of the free $k$-algebra $\mathscr{A}_k(X_0, X_1, \ldots, X_k)$, whose elements are called \emph{$(k+1)$-generators} of $X$.
We will denote the free $n$-algebra on $X$ by $\lin{X}$.

\subsubsection{Higher $\Omega$-monomials}
\label{SSS: HigherOmegaMonomials}
Let $X$ be an \LO-$\omega$-polygraph.
Every $0$-cell $a$ of $\lin{X}$ can be uniquely written as a linear combination
\[
a = \sum_{i=1}^p \lambda_i u_i
\]
of distinct $\Omega$-monomials $u_i$ with  nonzero scalars $\lambda_i$. 
The \emph{support} of $a$ is the set $\supp(a) \coloneq \{u_1,\dots,u_p\}$. 

For $n\geq1$, an \emph{$\Omega$-$n$-monomial}, or \emph{$n$-monomial} for short,
$\left.q\right|_{\alpha}$ of $\lin{X}$ is an $n$-cell of $\lin{X}$ where~$\alpha$ is an $n$-generator of $X$ and~$q\in X_0^\Omega[\square]$. Following the construction of the free $n$-algebra, every $n$-cell $a$ of $\lin{X}$ can be written as a linear combination
\begin{eqn}{equation}
\label{E:OmegaMonomialDecomposition}
a = \sum_{i=1}^p \lambda_i v_i + 1_c
\end{eqn}
of distinct $\Omega$-$n$-monomials  $v_i$ and an $(n-1)$-cell $c$. This decomposition is unique up to the exchange relation \eqref{E:LinearExchangeRelation}.
The \emph{size of~$a$} is the minimum number of $\Omega$-$n$-monomials required to write~$a$ as in \eqref{E:OmegaMonomialDecomposition}.

\subsubsection{Linear polygraphs}
\label{R:OmegaEmpty}
\LO-polygraphs are natural generalizations of \emph{linear polygraphs} introduced in \cite{GuiraudHoffbeckMalbos19}.
Indeed, when $\Omega$ is empty, an \LO-$\omega$-polygraph is a linear $\omega$-polygraph, presenting an $\omega$-associative algebra.
We will denote by $\Pol_n(\Alg)$ the category of linear $n$-polygraphs and their morphisms. 

\section{Rewriting in operated algebras}
\label{S:OperatedRewriting}
This section presents the main rewriting properties of \LO-1-polygraphs and the coherent critical branching theorem in the operated setting. We compare the shapes of critical branchings generated by \LO-1-polygraphs with those of linear 1-polygraphs, with further exploration in Section~\ref{S:Polyautomata}. Additionally, we relate convergent \LO-1-polygraphs to Gröbner--Shirshov theory for $\Omega$-algebras.

\subsection{Polygraphic presentations of operated algebras}
In this subsection, we present the rewriting properties of \LO-$1$-polygraphs. Inspired by~\cite{GuiraudMalbos09}, we introduce the notion of derivations to prove the termination of \LO-$1$-polygraphs.

\subsubsection{Operated polygraphic rewriting}
\label{SSS:1-polygraph}
 An \LO-$1$-polygraph is a pair $X=(X_0,X_1)$, where $X_0$ is a set and $X_1$ is a cellular extension
\[
\xymatrix @=1.5em {
\lin{X_0} 
& X_1.
\ar@<-0.5ex> [l] _-{s}
\ar@<0.5ex> [l] ^-{t}
}
\]
A $1$-cell $f$ in the free $1$-algebra $\lin{X}$ can be written as
\[
f = \sum_{i=1}^p \lambda_i \left.q_i\right|_{\alpha_i} + 1_c,
\]
where $\alpha_i \in X_1$, $q_i \in X_0^\Omega[\square]$, $c \in \lin{X_0}$ and $\left.q_i\right|_{\alpha_i}$ are $1$-monomials.
An \LO-$1$-polygraph $X$ is  \emph{left-monomial} if, for every $\alpha \in X_1$, the source $s(\alpha)$ is an $\Omega$-monomial in~$\lin{X_0}$ that does not belong to $\operatorname{Supp}\big(t(\alpha)\big)$. 
A \emph{rewriting step} in $X$ is a 1-cell $\lambda f + 1_c$ of size $1$ in $\lin{X}$ that satisfies the conditions $\lambda \neq 0$ and $s(f) \notin \operatorname{Supp}(c)$.
 A 1-cell in $\lin{X}$ is \emph{positive} if it is a (possibly empty) 0-composition of rewriting steps in $X$.

\subsubsection{Presentations of $\Omega$-algebras}
For $\alpha\in X_1$, we set $\partial(\alpha) \coloneq s(\alpha) - t(\alpha)$. 
The \emph{ideal $I_\Omega(X)$ of $X$} is the $\Omega$-ideal of $\lin{X_0}$ generated by the set of $0$-cells
\[
\partial(X_1) \: \coloneq  \: \{ \partial(\alpha) \mid \alpha \in X_1 \}.
\]
It consists of linear combinations of terms of the form $\left.q_i\right|_{\partial(\alpha_i)}$,
where $\alpha_i \in X_1$ and $q_i \in X_0^\Omega[\square]$.
The \emph{$\Omega$-algebra presented by $X$} is the quotient $\Omega$-algebra
\[
\overline{X} \: \coloneq \: \lin{X_0}/I_\Omega(X).
\]
A \emph{presentation} of an $\Omega$-algebra $A$ is an $\Omega$-$1$-polygraph $X$ such that $A$ is isomorphic to $\overline{X}$.
Two \LO-$1$-polygraphs are \emph{Tietze equivalent} if they present isomorphic algebras. 
Note that every \LO-1-polygraph is Tietze equivalent to a left-monomial one; see {\cite[§2.3.2]{GuiraudHoffbeckMalbos19}} for the associative case, whose proof extends to the operated case.
Hence, from now on, we assume without loss of generality that all \LO-$1$-polygraphs are left-monomial.

\subsubsection{Termination and monomial orders}
\label{SSS:Termination}
An \LO-$1$-polygraph $X$ is \emph{terminating} if there is no infinite sequence of $0$-composable rewriting steps.
 A \emph{monomial order} on $X_0^\Omega$ is a well-founded total order $\prec$ on $X_0^\Omega$ stable under products and operators, meaning that
\[
u\prec v \Rightarrow q|_u\prec q|_v\quad\text{for all }u,v\in X_0^\Omega\text{ and } q\in X_0^\Omega[\square].
\]
We say that $\prec$ is \emph{compatible} with $X_1$ if $v \prec s(\alpha)$ holds for all $\alpha\in X_1$ with $v$ in $\operatorname{Supp}(t(\alpha))$.
If there exists a monomial order $\prec$ on $X_0^\Omega$ compatible with $X_1$, then $X$ is terminating.

A \emph{rewriting order} on $X$ is a relation $\preccurlyeq_X$ on $\lin{X_0}$ defined as follows.
For $u,v\in X_0^\Omega$, we set $v\preccurlyeq_X u$ if there exists a positive $1$-cell
$u\fl b$ such that $v\in\operatorname{Supp}(b)$.
For $a,b\in\lin{X_0}$, we set $b\preccurlyeq_X a$ if, for any
$v\in\operatorname{Supp}(b)\setminus\operatorname{Supp}(a)$, there exists
$u\in\operatorname{Supp}(a)\setminus\operatorname{Supp}(b)$ such that
$v\preccurlyeq_X u$.
We define $v\prec_X u$ if $v\preccurlyeq_X u$ but $v\ne u$.
Note that $\prec_X$ is well-founded if $X$ terminates.

\subsubsection{Normal forms}
For an \LO-$1$-polygraph $X$, a $0$-cell $a$ of $\lin{X}$ is a \emph{normal form} if there is no rewriting step with source $a$; otherwise, it is \emph{reducible}. 
A \emph{normal form} of a $0$-cell $a$ is a normal form $b$ such that there exists a positive $1$-cell from $a$ to $b$. When such a normal form is chosen, we denote it by $\Nf(a,X)$ or~$\rep{a}$ for short.
We denote by $\Nf(X)$ the set of normal forms of $\lin{X}$.

\subsubsection{Derivations}
A \emph{derivation} of $\lin{X_0}$ with values in an $\lin{X_0}$-bimodule $N$ is a linear map $d: \lin{X_0} \to N$ satisfying the following conditions:
\begin{eqn}{equation}
\label{E:Derivation}
d(ab)=d(a)  \cdot b+ a \cdot d(b), \quad d(1)=0, \quad\text{ and } \quad d(\brck{a}{\tau})=\brck{d(a)}{\tau},
\end{eqn}
for all $a, b \in \lin{X_0}$.  
A well-founded partial order $\prec$ on $N$ is \emph{monotone} if the following hold
\[
n+m\succ n'+m,\quad
u\cdot n\cdot v\succ u\cdot n'\cdot v
\quad\text{and}\quad
\brck{n}{\tau}\succ\brck{n'}{\tau},
\]
for all $m,n,n'\in N$ and $u,v\in X_0^\Omega$ such that $n\succ n'$.

\begin{proposition}
\label{P:DerivationTerminating}
Let $X$ be an \LO-1-polygraph.
If there exist a derivation $d: \lin{X_0} \fl N$ and a monotone well-founded partial order $\prec$ on $d(X_0^\Omega)$, such that $d(s(\alpha)) \succ d(v)$, for all $\alpha\in X_1$ and $v\in\operatorname{Supp}(t(\alpha))$, then $X$ is terminating.
\end{proposition}
\begin{proof}
Suppose that there exists such a derivation. We show that it is compatible with $\Omega$-contexts.
If $d(v) \succ d(v')$ holds for $v,v'\in X_0^\Omega$, we have
\[
d(uvw)=d(u)\cdot vw + u\cdot d(v)\cdot w+uv \cdot d(w)\succ d(u)\cdot vw + u\cdot d(v')\cdot w+uv \cdot d(w)=d(uv'w),
\]
and
\[
d(\brck{v}{\tau})=\brck{d(v)}{\tau}\succ\brck{d(v')}{\tau}=d(\brck{v'}{\tau}),
\]
for all $u,w\in X_0^\Omega$.
By induction, we deduce that  $d(\left.q\right|_v) \succ d(\left.q\right|_{v'})$, for every $q \in X_0^\Omega[\square]$. 
Now, suppose that there exists no infinite sequence of rewriting steps
\[
a_1\fl a_2 \fl \cdots \fl a_n\fl \cdots,
\]
with $a_i\in \lin{X_0}$. 
We deduce an infinite sequence $u_1\fl u_2\fl \cdots \fl u_n\fl \cdots$,
where $u_i=s(\left.q_i\right|_{\alpha_i})$ and $u_{i+1}\in \operatorname{Supp}(t(\left.q_i\right|_{\alpha_i}))$ for some $\alpha_i\in X_1$ and $q_i \in X_0^\Omega[\square]$. By the compatibility of $\prec$, this implies a strictly decreasing chain
\[
d(u_1)\succ d(u_2)\succ\cdots\succ d(u_n)\succ\cdots, 
\]
which contradicts the well-foundedness of $\prec$.
Hence, the polygraph $X$ terminates.
\end{proof}

\subsection{Confluence and critical branchings of operated polygraphs}
This subsection deals with confluence properties of \LO-1-polygraphs. The classification of local branchings in \LO-1-polygraphs extends the case of linear 1-polygraphs in \cite{GuiraudHoffbeckMalbos19}. We highlight the differences in their critical  branchings, and state the coherent critical branching theorem in the operated setting.
From now on, for simplicity, we omit the subscript $\tau$ from the notation $\brck{~}{\tau}$.

\subsubsection{Branchings and confluence}
A \emph{branching} of an \LO-1-polygraph $X$ is a pair $(f, g)$ of positive 1-cells such that $s(f) = s(g)$. It is \emph{local} if both $f$ and $g$ are rewriting steps of $X$.
Given a cellular extension~$Y$ of $\lin{X}$, a branching~$(f,g)$ is \emph{$Y$-confluent} if there exist positive $1$-cells~$h$ and~$k$ in~$\lin{X}$ and a $2$-cell~$F$ in~$\lin{X}[Y]$ as follows
\[
\vcenter{\xymatrix @R=0.5em{
& b
\ar@/^/ [dr] ^-{h}
\ar@2 []!<0pt,-12.5pt>;[dd]!<0pt,12.5pt> ^-*+{F}
\\
a
\ar@/^/ [ur] ^-{f}
\ar@/_/ [dr] _-{g}
&& d\;.
\\
& c
\ar@/_/ [ur] _-{k}
}}
\]
We say that~$X$ is \emph{$Y$-confluent} (resp.  \emph{locally $Y$-confluent) at a 0-cell $a$} if every branching (resp.\ local branching) of~$X$ with source~$a$ is $Y$-confluent, and that~$X$ is \emph{$Y$-confluent} (resp.\ \emph{locally $Y$-confluent}) if it is so at every $0$-cell of~$\lin{X}$. 
When $Y$ contains all $1$-spheres of $\lin{X}$, $Y$-confluence corresponds to the classical notion of \emph{confluence}. We say that $X$ is \emph{convergent} when it is both terminating and confluent.
Each $0$-cell~$a$ of $\lin{X}$ then has a unique normal form.

\subsubsection{Classification of local branchings}
\label{SSS:ClassificationOfLocalBranchings}
The local branchings of \LO-1-polygraphs fall into the following families:
\begin{enumerate}
    \item \emph{Aspherical branchings}: $(\lambda f + c, \lambda f + c)$, where $f$ is a $1$-monomial of $X$, and $\lambda$ is a nonzero scalar.

    \item \emph{Additive branchings}: $(\lambda f + \mu v + c, \lambda u + \mu g + c)$, where $f: u \to a$ and $g: v \to b$ are 1-monomials in $\lin{X}$, $\lambda, \mu$ are nonzero scalars, and $c$ is a 0-cell in $\lin{X}$, satisfying $u \neq v$ and $u, v \notin \operatorname{Supp}(c)$.

    \item \emph{Peiffer branchings}: $(\lambda q|_{\lfloor fv \rfloor_{\tau}} + c, \lambda q|_{\lfloor ug \rfloor_{\tau}} + c)$,  
where $q\in X_0^\Omega[\square]$, $f: u \to a$ and $g: v \to b$ are 1-monomials in $\lin{X}$, $\lambda$ is a nonzero scalar, and $c$ is a 0-cell in $\lin{X}$, satisfying $q|_{\lfloor uv \rfloor_{\tau}} \notin \operatorname{Supp}(c)$.  
This case corresponds to the Peiffer branching in the associative setting, as defined in {\cite[Def.~3.3.2]{GuiraudHoffbeckMalbos19}}, when $q = \square$ and $\tau = 0$.

    \item \emph{Overlapping branchings}: $(\lambda f + c, \lambda g + c)$,  
where $f: u \to a$ and $g: u \to b$ are 1-monomials in $\lin{X}$ such that the pair $(f, g)$ is neither aspherical, additive, nor Peiffer.  
Here, $\lambda$ is a nonzero scalar, and $c$ is any 0-cell of $\lin{X}$, with $u \notin \operatorname{Supp}(c)$.
\end{enumerate}

\subsubsection{Critical branchings}
The \emph{critical branchings} of an $\Omega$-1-polygraph $X$ are the overlapping branchings in \eqref{SSS:ClassificationOfLocalBranchings} for which $\lambda = 1$ and $c = 0$, and which cannot be factored as $(f, g) = (\left.q\right|_{f'}, \left.q\right|_{g'})$, where~$q \in X_0^\Omega[\square]$ and $q \neq \square$. Specifically, if we define a well-founded order $\prec$ on overlapping branchings by
\[
(\left.q\right|_{f'},\left.q\right|_{g'})\succ (f', g')
\quad \text{ for } q\in X_0^\Omega[\square] \text{ and } q\neq \square,
\]
then critical branchings are minimal with respect to this order. Explicitly, every overlapping branching has a unique decomposition  as $(\left.q\right|_{f}+c, \left.q\right|_{g}+c)$, where $(f, g)$ is a critical branching.
We denote the set of critical branchings of a polygraph $X$ by $\mathrm{CB}(X)$. 
In particular, critical branchings are classified into the following two types,
called \emph{intersection branchings} $(fw,ug)$ and \emph{inclusion branchings}
$(\left.q\right|_h,k)$, respectively,
\begin{eqn}{equation}
\label{E:OperatedCriticalBranchings}
\raisebox{0.6cm}{
\xymatrix @R=0.5em {
& aw
\\
uvw
\ar@/^/ [ur] ^-{fw}
\ar@/_/ [dr] _-{ug}
\\
& ub
}
\qquad\raisebox{-0.65cm}{\text{and}}\qquad
\raisebox{0.25cm}{
\xymatrix @R=0.5em {
& \left.q\right|_{a'}
\\
\left.q\right|_{v'}
\ar@/^/ [ur] ^-{\left.q\right|_{h}}
\ar@/_/ [dr] _-{k}
\\
& b'\quad,
}}}
\end{eqn}
where $u,v,w,v' \in X_0^{\Omega} \setminus \{1\}$, $q\neq \square$ and $f: uv\fl a$, $g: vw \fl b$, $h: v'\fl a'$, $k:\left.q\right|_{v'} \fl b'$ belong to $X_1$.

Let $Y$ be a cellular extension of $\lin{X}$. We say that~$X$ is \emph{critically $Y$-confluent at a 0-cell $a$} if every critical branching of~$X$ with source~$a$ is $Y$-confluent, and that~$X$ is \emph{critically $Y$-confluent} if it is so at every~$0$-cell of~$\lin{X}$. 

\begin{remark}
\label{R:DifferencePolygraphAnd OmegaPolygraph}
When $\Omega = \emptyset$, in \eqref{E:OperatedCriticalBranchings}, $q$ takes the form $u' \square w'$, which corresponds to the inclusion branchings $(u'hw',k)$ in the associative setting. However, operators introduce additional complexity in the structure of critical branchings. For example, the rules $k:\left\lfloor xy \right\rfloor \to yx$ and $h:xy \to z$ give rise to an inclusion branching $(\left\lfloor h \right\rfloor,k)$ that cannot be expressed in the associative setting. In Section~\ref{S:Polyautomata}, we explain how to describe these critical branchings in terms of string overlaps.
\end{remark}

The following lemma is the operated analogue of {\cite[Lem.~3.1.3 and Lem.~4.1.2]{GuiraudHoffbeckMalbos19}}:
\begin{lemma}
\label{SSS:LemmaOfCriticalBranchings}
Let~$X$ be an \LO-$1$-polygraph, and $Y$ be a cellular extension of~$\lin{X}$ such that $X$ is $Y$-confluent at every $0$-cell $b\prec_X a$  for some fixed $0$-cell~$a$ of~$\lin{X}$.
 Let~$f$ be a $1$-cell of~$\lin{X}$ that admits a decomposition
\[
a_0 \:\ofl{f_1}\: a_1 \:\ofl{f_2}\: \cdots \:\ofl{f_p}\: a_p
\]
into $1$-cells of size~$1$. If there exist positive $1$-cells $a\fl a_i$ for every~$0<i<p$, then there exist positive $1$-cells~$g$ and~$h$ in~$\lin{X}$ and a $2$-cell~$F$ in~$\lin{X}[Y]$ satisfying the following
\[
\xymatrix @R=1em {
& a_p
	\ar@/^/ [dr] ^-{h}
	\ar@2 []!<0pt,-12pt>;[d]!<0pt,3pt> ^-*+{F}
\\	
a_0 
	\ar@/^/ [ur] ^-{f} 
	\ar@/_/ [rr] _-{g}
&& a'\;.
}
\]
\end{lemma}

\begin{theorem}
\label{T:CriticalBranchingLemma}
Let~$X$ be a terminating \LO-$1$-polygraph, and~$Y$ be a cellular extension of~$\lin{X}$. If $X$ is critically $Y$-confluent, then $X$ is  $Y$-confluent.
\end{theorem}
\begin{proof}
We prove this result by considering the four cases of local branchings in \eqref{SSS:ClassificationOfLocalBranchings}. 
The proof of the confluence of these branchings uses the method from {\cite[Thm.~4.2.1]{GuiraudHoffbeckMalbos19}} for associative algebras, except for the Peiffer branchings, whose source is given by $\lambda \left.q\right|_{\left\lfloor uv\right\rfloor_{\tau}} + c$, as in~\eqref{SSS:ClassificationOfLocalBranchings}. 
We prove the confluence of such a branching by induction on the rewriting order $\prec_X$ defined in \eqref{SSS:Termination}.
Given a reducible 0-cell $a$ of $\lin{X}$, we assume that $X$ is locally $Y$-confluent at every $0$-cell $b\prec_X a$.
By applying the coherent version of Newman's Lemma, as stated in {\cite[Prop.~4.1.3]{GuiraudHoffbeckMalbos19}}, which also holds in the operated setting, we conclude that $X$ is $Y$-confluent at every $b$.
We then have the following coherently confluent diagram
\[
\xymatrix @R=2.5em @C=1.5em {
& \lambda \left.q\right|_{\left\lfloor av\right\rfloor_{\tau}}+ c
\ar@/^/ [rr] ^-{f'_1}
\ar@{.>} [dr] |-*+{\lambda\left.q\right|_{\left\lfloor ag\right\rfloor_{\tau}}+ c}
&& a'
\ar@/^/ [drr] ^-{f'_2}
\ar@2 []!<2.5pt,-37.5pt>;[dd]!<2.5pt,37.5pt> ^-*+{H}
\\
\lambda \left.q\right|_{\left\lfloor uv\right\rfloor_{\tau}} + c
\ar@/^2ex/ [ur] ^-*+{\lambda\left.q\right|_{\brck{fv}{\tau}}+ c}
\ar@/_2ex/ [dr] _-*+{\lambda \left.q\right|_{\left\lfloor ug\right\rfloor_{\tau}} + c}
\ar@{} [rr] |(0.45){\sm =}
&& \lambda\left.q\right|_{\left\lfloor ab\right\rfloor_{\tau}} + c
\ar [ur] |-*+{h}
\ar [dr] |-*+{k}
\ar@2 [u]!<0pt,-7.5pt>;[]!<0pt,27.5pt> ^-*+{F^-}
\ar@2 []!<0pt,-27.5pt>;[d]!<0pt,7.5pt> ^-*+{G}
&&& d \;.
\\
& \lambda \left.q\right|_{\left\lfloor ub\right\rfloor_{\tau}} + c
\ar@{.>} [ur] |-*+{\lambda \left.q\right|_{\left\lfloor fb\right\rfloor_{\tau}}+ c}
\ar@/_/ [rr] _-{g'_1}
&& b'
\ar@/_/ [urr] _-{g'_2}
}
\]
Note that the dotted 1-cells $\lambda \left.q\right|_{\left\lfloor ag\right\rfloor_{\tau}}+ c$ and $\lambda \left.q\right|_{\left\lfloor fb\right\rfloor_{\tau}}+ c$ may not be positive 1-cells if either $\operatorname{supp}(\left.q\right|_{\left\lfloor av\right\rfloor_{\tau}}) \cap\operatorname{supp}(c)$ or $\operatorname{supp}(\left.q\right|_{\left\lfloor ub\right\rfloor_{\tau}}) \cap \operatorname{supp}(c)$ is not empty. Following Lemma~\ref{SSS:LemmaOfCriticalBranchings}, we derive positive 1-cells $f_1', g_1', h, k$ and 2-cells $F$ and $G$. Since $\lambda \left.q\right|_{\left\lfloor ab\right\rfloor_{\tau}} + c\prec_X\lambda \left.q\right|_{\left\lfloor uv\right\rfloor_{\tau}} + c$, we further obtain positive 1-cells $f_2'$, $g_2'$, and a 2-cell $H$ that ensures $(h,k)$ is $Y$-confluent by hypothesis.
Finally, $X$ is $Y$-confluent by applying Newman's Lemma.
\end{proof}

As in the case of associative algebras {\cite[Thm.~3.4.2]{GuiraudHoffbeckMalbos19}}, convergent polygraphs provide canonical linear bases, as stated by the following result.

\begin{proposition}
\label{P:LinearBasisConvergentPolygraphs}
If $X$ is a convergent \LO-1-polygraph, the set $\Nf(X)$ forms a linear basis of the $\Omega$-algebra~$\overline{X}$.
\end{proposition}

\subsubsection{Reduced convergent presentations}
\label{SSS:ReducedPolygraph}
An \LO-$1$-polygraph $X$ is \emph{left-reduced} if the only rewriting step in $X$ with source $s(\alpha)$ is $\alpha$ itself, for every $1$-generator~$\alpha$.
It is \emph{right-reduced} if, for every $1$-generator~$\alpha$, the 0-cell $t(\alpha)$ is a normal form.
The polygraph $X$ is \emph{reduced} if it is both left-reduced and right-reduced.
Following {\cite[Thm.~2.4]{Squier87}}, every (finite) convergent string rewriting system is Tietze equivalent to a (finite) reduced convergent one. This result also holds for \LO-$1$-polygraphs; see {\cite[§3.2.3]{GuiraudMalbos18}} for the corresponding reduction procedure.

\begin{example}
\label{E:ConvergentOfDifferentialAlgebra}
The free differential algebra $\DA{\lambda}{X_0}$ is presented by the following \LO-$1$-polygraph
\begin{eqn}{equation}
\label{E:FreeDifferentialAlgebra}
\begin{aligned}
X^{D} \coloneq (X_0,X_1^{D}),
\qquad
X_1^{D} \coloneq \big\{ \alpha[u,v]&: D(uv)\rightarrow D(u)v+u D(v)+\lambda D(u) D(v),\\
\varphi&:D(1)\fl0 \text{ } |\text{ }  u,v \in X_0^\Omega\setminus \{1\} \big\}.
\end{aligned}
\end{eqn}
We set $N \coloneq \mathbb{Z}^3$ and define a derivation $d: X_0^\Omega \to N$ by  
\[
d(u) = \bigg(\sum_{D|u} \max\{\operatorname{deg}_\Omega(D) + \operatorname{deg}_{X_0}(D) - 1, 0\}, \operatorname{deg}_\Omega(u), \operatorname{deg}_{X_0}(u) \bigg),
\]
for every $u \in X_0^\Omega$, where $D|u$ denotes each occurrence of the operator $D$ in $u$. Here, $\operatorname{deg}_\Omega(D)$ and $\operatorname{deg}_{X_0}(D)$ count the number of operators and $0$-generators inside the operator $D$, respectively, while $\operatorname{deg}_\Omega(u)$ and $\operatorname{deg}_{X_0}(u)$ count the number of operators and $0$-generators in $u$.
For instance, we have $d(D(1))=(0,1,0)$, $d(D(xy))=(1,1,2)$, and $d(D(x)D(y))=(0,2,2)$ for $x,y\in X_0$.
For $(m_1,m_2,m_3)\in N$ and $a\in \DA{\lambda}{X_0}$, we define
\[
a\cdot(m_1,m_2,m_3)
=(m_1,m_2,m_3)\cdot a
=(m_1,m_2,m_3),
\quad
D\bigl((m_1,m_2,m_3)\bigr)
=(\operatorname{max}\{m_1+m_2+m_3-1,0\},m_2+1,m_3).
\]
By definition, $d$ satisfies the conditions in \eqref{E:Derivation}. Next, we endow $d(X_0^\Omega) \subseteq N$ with a monotone \emph{lexicographic order} $\prec$, comparing tuples $(m,n,l)\in d(X_0^\Omega)$ lexicographically, where $d(1) = (0,0,0)$ is the minimal element. This ensures that $d(D(u)D(v))$, $d(D(u)v)$ and $d(uD(v))$ are all less than $d(D(uv))$ for every $u, v \in X_0^\Omega\setminus \{1\}$, and $d(D(1)) \succ (0,0,0)$. Hence, $X^D$ is terminating.

If we write $\alpha[1,1] \coloneq \varphi $, then the polygraph $X^{D}$ has two families of critical branchings
\[
\xymatrix @R=0.5em@C=4em {
&a_1 \ar@/^/ [dr] 
\\
D(\left.q\right|_{D(uv)}w)
\ar@/^/ [ur] ^-{\alpha[\left.q\right|_{D(uv)},w]}
\ar@/_/ [dr] _-{D(\left.q\right|_{\alpha[u,v]}w)}
&& a_3
\\
& a_2 \ar@/_/ [ur] 
}
\qquad\qquad\qquad
\xymatrix @R=0.5em@C=4em {
&b_1 \ar@/^/ [dr] 
\\
D(w\left.q\right|_{D(uv)})
\ar@/^/ [ur] ^-{\alpha[w,\left.q\right|_{D(uv)}]}
\ar@/_/ [dr] _-{D(w\left.q\right|_{\alpha[u,v]})}
&& b_3
\\
&b_2 \ar@/_/ [ur] 
}
\]
indexed by $u,v,w\in X_0^\Omega$ with $w\neq 1$, both of which are confluent. Here, 
\[
\begin{aligned}
a_1 &= D(\left.q\right|_{D(uv)})w + \left.q\right|_{D(uv)} D(w) + \lambda D(\left.q\right|_{D(uv)}) D(w), \\
a_2 &= D(\left.q\right|_{D(u)v}w) + D(\left.q\right|_{uD(v)}w) + \lambda D(\left.q\right|_{D(u)D(v)}w), \\
a_3 &= D(\left.q\right|_{D(u)v})w + D(\left.q\right|_{uD(v)})w + \lambda D(\left.q\right|_{D(u)D(v)})w \\
&\quad + \left.q\right|_{D(u)v} D(w) + \left.q\right|_{uD(v)} D(w) + \lambda \left.q\right|_{D(u)D(v)} D(w) \\
&\quad + \lambda D(\left.q\right|_{D(u)v}) D(w) + \lambda D(\left.q\right|_{uD(v)}) D(w) + \lambda^2 D(\left.q\right|_{D(u)D(v)}) D(w),
\end{aligned}
\]
and $b_1, b_2$ and $b_3$ can be written similarly. Thus $X^D$ is convergent.
Let $D^\theta(X_0) \coloneq \{ D^i(x) \mid x \in X_0, i \geq 0 \}$, where $D^0(x) \coloneq x$. The set
\[
\Nf(X^D)=\left(D^\theta(X_0)\right)^\ast
\]
forms a linear basis of $\DA{\lambda}{X_0}$.
Note that, since $X^D$ contains inclusion branchings, it is not reduced. We will construct a reduced presentation of $\DA{\lambda}{X_0}$ in Subsection~\ref{SS:ResolutionsOfFreeDifferentialAlgebras}.
\end{example}

\subsection{Gröbner--Shirshov bases and convergence}
\label{SS:GSConvergence}
In this subsection, we establish a relationship between Gröbner--Shirshov bases of operated algebras \cite{BokutChenQiu,GGR15,QQWZ21} and convergent \LO-1-polygraphs.

\subsubsection{Gröbner--Shirshov bases}
Let $X$ be an \LO-1-polygraph and $S$ be a nonzero subset of $\lin{X_0}$.
For the critical branchings in \eqref{E:OperatedCriticalBranchings}, we set $r = uvw$ (resp. $r = \left.q\right|_{v'}$) and define the $0$-cells
\begin{eqn}{equation}
\label{EQ:TwoCompositions} 
(a,b)_r \: \coloneq \: aw - ub \quad \text{\big(resp. $(a',b')_r \: \coloneq \: \left.q\right|_{a'} - b'$\big)}. 
\end{eqn}
Given a monomial order $\prec$ on $X_0^\Omega$ compatible with $X_1$ and a nonzero $0$-cell $c$ in $\lin{X_0}$, we denote by $\lm_\prec(c)$ the maximal $\Omega$-monomial in $\operatorname{Supp}(c)$.
The cell $c$ is \emph{trivial modulo $(S,r)$} 
if there is a decomposition
\[
c=\Sigma_i \lambda_i \left.q_i\right|_{s_i}\quad \text{ with } \left.q_i\right|_{\lm_\prec(s_i)}\prec r,
\]
where $\lambda_i \in \bk$, $q_i \in X_0^\Omega[\square]$, and $s_i \in S$.

A nonzero subset $S$ is a \emph{Gröbner--Shirshov (GS) basis of $\lin{X_0}$ with respect to $\prec$} if, for all critical branchings \eqref{E:OperatedCriticalBranchings}, the $0$-cells in \eqref{EQ:TwoCompositions} are trivial modulo $(S, r)$.

\begin{proposition}
\label{P:ConvergenceGrobnerShirshov}
Let $X$ be an \LO-$1$-polygraph.
If the set $\partial(X_1)$ forms a GS basis~of $\lin{X_0}$ with respect to a monomial order $\prec$ compatible with $X_1$, then the polygraph $X$ is convergent.
\end{proposition}
\begin{proof}
The termination of $X$ follows from the compatibility of the rewriting rules with the monomial order $\prec$. We consider every intersection branching $(fw,ug)$ in \eqref{E:OperatedCriticalBranchings}. Since $\partial(X_1)$ forms a GS basis of $\lin{X_0}$ with respect to $\prec$, there is a decomposition
\begin{eqn}{equation}
\label{E:GSDecomposition}
aw-ub=\Sigma_i \lambda_i \left.q_i\right|_{\partial(\alpha_i)}, 
\end{eqn}
where $\alpha_i\in X_1$ and $\left.q_i\right|_{s(\alpha_i)}\prec uvw$.
Since $\partial(\alpha_i)$ and $0$ have the same image in $\lin{X_0}$, we have $\rep{\partial(\alpha_i)} = 0$.
Applying the normal form to both sides of \eqref{E:GSDecomposition}, we obtain $\rep{aw}=\rep{ub}$. Hence every critical branching 
$(fw,ug)$ is confluent. 
We apply the same reasoning to the inclusion branchings.
Taking the cellular extension $Y$ containing all 1-spheres of $\lin{X}$, the confluence of $X$ is deduced from Theorem~\ref{T:CriticalBranchingLemma}. 
\end{proof}

\begin{remark}
\label{E:NoMonomialOrder}
The converse of Proposition~\ref{P:ConvergenceGrobnerShirshov} does not hold in general. 
Indeed, a
convergent $\Omega$-1-polygraph may not admit a compatible monomial order.
For example, the \LO-1-polygraph
\[
X \: \coloneq \: \big\{x, y, z \mid x\brca{y}z \fl \brca{x}y\brca{z} + \brca{x}\brca{y}z + x\brca{y}\brca{z}\big\}
\]
is terminating since, for every $q \in X_0^\Omega[\square]$, the $\Omega$-monomial $\left.q\right|_{x\brca{y}z}$ contains one more factor $x\brca{y}z$ than $\left.q\right|_{\brca{x}y\brca{z}}$, $\left.q\right|_{\brca{x}\brca{y}z}$, or $\left.q\right|_{x\brca{y}\brca{z}}$.
Moreover, this polygraph is confluent, since it has no critical branchings.
However, there does not exist a monomial order $\prec$ compatible with $X_1$. Such an order $\prec$  would imply $x\brca{y}z \succ \brca{x}\brca{y}z$ and $x\brca{y}z \succ x\brca{y}\brca{z}$, leading to $x \succ \brca{x}$, $z \succ \brca{z}$, and thus to
\[
x \succ \brca{x} \succ \brca{\brca{x}} \succ \ldots  \quad \text{and} \quad z \succ \brca{z} \succ \brca{\brca{z}} \succ \ldots,
\]
contradicting the fact that $\prec$ is a well-founded total order.
\end{remark}

\subsubsection{Completion procedure}
\label{SSS:CompletionProcedure}
When $X$ is a non-confluent terminating \LO-1-polygraph, we can complete the set of $1$-generators of $X$, without changing the presented algebra, in order to reach confluence.
This \emph{completion procedure} is well known in rewriting theory; see 
\cite{Buchberger65} for commutative algebras and \cite{KnuthBendix70} for term rewriting. 
Starting with a monomial order $\prec$ compatible with $X_1$, the procedure examines each non-confluent critical branching $(f, g)$ in $X$, and reduces $t(f)$ and $t(g)$ to some normal forms~$\rep{t(f)}$ and~$\rep{t(g)}$. 
A new $1$-generator $\lm_\prec(a) \fl -\lambda^{-1}a + \lm_\prec(a)$ is then added to the polygraph, where $a= \rep{t(f)}-\rep{t(g)}$ and $\lambda$ is the coefficient of $\lm_\prec(a)$ in $a$.
When it terminates, the procedure produces a terminating \LO-1-polygraph $Y$ such that $\overline{X} \cong \overline{Y}$.

\section{Polygraphic resolutions of operated algebras}
\label{S:PolygraphicResolutions}

In this section, we present the acyclicity properties of an \LO-$\omega$-polygraph using homotopical contractions, as introduced in \cite{GuiraudHoffbeckMalbos19}. The main result of this paper, Theorem~\ref{T:Main conclusion}, constructs polygraphic resolutions for $\Omega$-algebras from convergent and reduced presentations, extending the constructions for associative algebras \cite{GuiraudHoffbeckMalbos19}, categories \cite{GuiraudMalbos12advances}, and operads \cite{MalbosRen23} to $\Omega$-algebras.

\subsection{Polygraphic resolutions and contractions}

A notion of homotopy on associative $\omega$-algebras was introduced in~\cite{GuiraudHoffbeckMalbos19}.
In this subsection, we extend this notion to $\Omega$-$\omega$-algebras. 
To account for the operator structure, we introduce the notion of bracket contraction to characterize acyclic $\Omega$-$\omega$-algebras.

\subsubsection{Polygraphic resolutions}
A cellular extension $Y$ of an $n$-algebra $A$ is \emph{acyclic} if, for every $n$-sphere $(f,g)$ in $A_n$, there exists an $(n+1)$-cell of the free $(n+1)$-algebra $A[Y]$ with source $f$ and target~$g$.
An \LO-$\omega$-polygraph $X$ is a \emph{polygraphic resolution} of an $\Omega$-algebra $A$ if its underlying \LO-1-polygraph $(X_0, X_1)$ is a presentation of~$A$ and all the cellular extensions $X_n$ are acyclic.

\subsubsection{Homotopies}
Let $F,G:A\fl B$ be two morphisms of \LO-$\omega$-algebras. A \emph{homotopy from~$F$ to~$G$} is an \emph{indexed morphism} of \LO-$\omega$-algebras
\[
\eta \: : \: A \ofl{} B
\]
of degree~$1$, namely, a sequence $\eta = (\eta_k : A_k \fl B_{k+1})_{k \geq 0}$ of morphisms of \LO-algebras, satisfying the following conditions, where we write~$\eta_a\coloneq\eta_k(a)$ for every $a\in A_k$:
\begin{enumerate}
\item for every~$0$-cell~$a$ of~$A$, 
\[
s(\eta_a) = F(a)
\qquad\text{and}\qquad
t(\eta_a) = G(a);
\]
\item for $n\geq 1$ and every $n$-cell~$a$ of~$A$,
\begin{eqn}{align}
\label{E:SourceHomotopy}
s(\eta_a) &= F(a) \star_0 \eta_{t_0(a)} \star_1 \cdots \star_{n-1} \eta_{t_{n-1}(a)}, 
\\[0.5ex]
\label{E:TargetHomotopy}
t(\eta_a) &= \eta_{s_{n-1}(a)} \star_{n-1} \cdots \star_1 \eta_{s_0(a)} \star_0 G(a);
\end{eqn}
\item for $n\geq 0$ and every $n$-cell~$a$ of~$A$,
\[
\eta_{1_a} = 1_{\eta_a}.
\]
\end{enumerate}

Following {\cite[Def.~5.1.1]{GuiraudHoffbeckMalbos19}} in the associative setting, \eqref{E:SourceHomotopy} and \eqref{E:TargetHomotopy} are well-defined. 
The globularity of $\eta_a$, for every $n$-cell $a$ of $A$, follows from
\begin{align*}
& ss(\eta_a)
= s(F(a)) \star_0 \eta_{t_0(a)} \star_1 \cdots \star_{n-2} \eta_{t_{n-2}(a)} 
= s(\eta_{s(a)}) 
= st(\eta_a)
\\[0.5ex]
\text{and}\quad 
&
ts(\eta_a) 
= t(\eta_{t(a)})
= \eta_{s_{n-2}(a)} \star_{n-2} \cdots \star_1 \eta_{s_0(a)} \star_0 t(G(a))
= tt(\eta_a).
\end{align*}

\begin{remark}
From this definition, we prove by induction on $n$ that for every $a\in A_n$, the following relation holds 
\[
\eta_{\brca{a}} = \brca{\eta_{a}}.
\]
Indeed, for $a\in A_0$, since $\brca{~}$ commutes with the morphisms $F$ and $G$, we have $\eta_{\brca{a}} = \brca{\eta_{a}}$.
Assume $\eta_{\brca{a}} = \brca{\eta_{a}}$ holds for every $a\in A_k$ with $0 \leq k \leq n$.
Given $a\in A_{n+1}$, we use the inductive hypothesis and the commutativity of $\brca{~}$ with $F$, $s$, and $\star_0$ to obtain 
\[
s(\eta_{\brca{a}}) = \brca{F(a)} \star_0 \brca{\eta_{t_0(a)}} \star_1 \cdots \star_{n-1} \brca{\eta_{t_{n-1}(a)}} = \brca{s(\eta_{a})}.
\]
Similarly, we deduce
$t(\eta_{\brca{a}}) = \brca{t(\eta_{a})}$.
\end{remark}

\begin{example}
In low dimensions, the homotopy $\eta$ maps a 1-cell $f: a \fl a'$ of $A$ to a 2-cell
\[
\xymatrix @R=0.25em @C=1.5em {
& F(a')
\ar@/^/ [dr] ^-{\eta_{a'}}
\ar@2 []!<0pt,-15pt>;[dd]!<0pt,15pt> ^-*+{\eta_f}
\\
F(a)
\ar@/^/ [ur] ^-{F(f)}
\ar@/_/ [dr] _-{\eta_a}
&& G(a')
\\
& G(a)
\ar@/_/ [ur] _-{G(f)}
}
\]
of~$B$, and a $2$-cell $H:f\dfl f':a\fl a'$ of~$A$ to a $3$-cell 
\[
\vcenter{\xymatrix @R=0.25em @C=4em {
& F(a')
\ar@/^/ [dr] ^-{\eta_{a'}}
\ar@2 []!<5pt,-15pt>;[dd]!<5pt,15pt> ^-*+{\eta_{f'}}
\\
F(a)
\ar@/^5ex/ [ur] ^(0.2){F(f)} _{}="s"
\ar@/_/ [ur] |-*+{F(f')} ^{}="t"
\ar@2 "s"!<-6pt,-10pt>;"t"!<-9pt,10pt> ^-{F(H)}
\ar@/_/ [dr] _{\eta_a}
&& {G(a')} 
\\
& G(a)
\ar@/_/ [ur] _-{G(f')}
}}
\quad \tfl^{\eta_H}\quad
\vcenter{\xymatrix @R=0.25em @C=4em {
& F(a')
\ar@/^/ [dr] ^-{\eta_{a'}}
\ar@2 []!<-10pt,-15pt>;[dd]!<-10pt,15pt> ^-*+{\eta_{f}}
\\
F(a)
\ar@/^/ [ur] ^-{F(f)}
\ar@/_/ [dr] _{\eta_a}
&& {G(a')} 
\\
& G(a)
\ar@/^/ [ur] |-*+{G(f)} _-{}="s"
\ar@/_5ex/ [ur] _(0.8){G(f')} ^{}="t"
\ar@2 "s"!<-7pt,-12pt>;"t"!<-10pt,7pt> ^-{G(H)}
}}
\]
of $B$, and a 3-cell $\Gamma: H \tfl H': f \dfl f': a \fl a'$

\[
\xymatrix@!R@!C@C=9em{
a
\ar@/^5ex/ [r] ^-{f}="src"
\ar@/_5ex/ [r] _-{f'}="tgt"
\ar@2 "src"!<-18pt,-15pt>;"tgt"!<-18pt,15pt> _-*+{H} ^-{}="src2"
\ar@2 "src"!<18pt,-15pt>;"tgt"!<18pt,15pt> ^-*+{H'} _-{}="tgt2"
\ar@3 "src2"!<7.5pt,0pt>;"tgt2"!<-7.5pt,0pt> ^-*+{\Gamma}
& a'
}
\]
to a 4-cell of $B$ whose source is 
\[
\xymatrix@!C@C=25em{
F(a)
\ar@/^6ex/ [r] ^-{F(f)\star_0\eta_{a'}}="src"
\ar@/_6ex/ [r] _-{\eta_{a}\star_0 G(f')}="tgt"
\ar@2 "src"!<-57pt,-18pt>;"tgt"!<-57pt,18pt> _-*+{\scriptscriptstyle F(H)\star_0 \eta_{a'} \star_1 \eta_{f'}} ^-{}="src2"
\ar@2 "src"!<0pt,-15pt>;"tgt"!<0pt,15pt> |(0.71){\scriptscriptstyle F(H')\star_0 \eta_{a'}\star_1 \eta_{f'}} _-{}="tgt1"
\ar@2 "src"!<57pt,-18pt>;"tgt"!<57pt,18pt> ^-*+{\scriptscriptstyle \eta_{f}\star_1\eta_{a}\star_0G(H')} _-{}="tgt2"
\ar@3 "src2"!<7.5pt,5pt>;"tgt1"!<-7.5pt,5pt> ^-*+{\scriptscriptstyle F(\Gamma)\star_0 \eta_{a'} \star_1 \eta_{f'}}
\ar@3 "tgt1"!<7.5pt,5pt>;"tgt2"!<-7.5pt,5pt> ^-*+{\scriptscriptstyle\eta_{H'}}
& G(a')
}
\]
and whose target is as follows
\[
\xymatrix@!C@C=25em{
F(a)
\ar@/^6ex/ [r] ^-{F(f)\star_0\eta_{a'}}="src"
\ar@/_6ex/ [r] _-{\eta_{a}\star_0 G(f')}="tgt"
\ar@2 "src"!<-57pt,-18pt>;"tgt"!<-57pt,18pt> _-*+{\scriptscriptstyle F(H)\star_0 \eta_{a'} \star_1 \eta_{f'}} ^-{}="src2"
\ar@2 "src"!<0pt,-15pt>;"tgt"!<0pt,15pt> |(0.71){\scriptscriptstyle \eta_f \star_1 \eta_a \star_0 G(H)} _-{}="tgt1"
\ar@2 "src"!<57pt,-18pt>;"tgt"!<57pt,18pt> ^-*+{\scriptscriptstyle \eta_{f}\star_1\eta_{a}\star_0G(H')} _-{}="tgt2"
\ar@3 "src2"!<7.5pt,5pt>;"tgt1"!<-7.5pt,5pt> ^-*+{\scriptscriptstyle \eta_H}
\ar@3 "tgt1"!<7.5pt,5pt>;"tgt2"!<-7.5pt,5pt> ^-*+{\scriptscriptstyle\eta_f \star_1 \eta_a \star_0 G(\Gamma)}
& G(a') \;.
}
\]
\end{example}

The following result is proved as in the case of linear polygraphs  {\cite[Lem.~5.1.4]{GuiraudHoffbeckMalbos19}}.

\begin{lemma}
\label{L:LemmaOfHomotopy}
Let~$X$ be an \LO-$\omega$-polygraph, $(A, \Tr)$ an $\omega$-algebra, and~$F,G:\lin{X}\fl A$ morphisms of $\omega$-algebras. A homotopy~$\eta$ from~$F$ to~$G$ is uniquely determined by its values on the $n$-monomials of~$\lin{X}$, for~$n\geq 0$, provided the relation
\begin{eqn}{equation}
\label{E:HomotopyExchange}
\eta_{us_0(v)} + \eta_{t_0(u)v} - \eta_{t_0(u)s_0(v)} 
\: = \: \eta_{s_0(u)v} + \eta_{ut_0(v)} - \eta_{s_0(u)t_0(v)} 
\end{eqn}
is satisfied for all $n$-monomials~$u$ and~$v$ of~$\lin{X}$.
\end{lemma}

\subsubsection{Unital sections and contractions}
Let $X$ be an \LO-$\omega$-polygraph.
The presented algebra $\cl{X}$ can be viewed as an $\omega$-algebra in which all $n$-cells are identities for $n \geq 1$.
A \emph{unital section} of~$X$ is a linear map of the underlying vector spaces $\iota:\cl{X}\fl\lin{X}$ that is a section of the canonical projection $\pi:\lin{X}\pfl\cl{X}$ and that satisfies $\iota(1)=1$.
For any $n$-cell  $a$ of~$\lin{X}$, we write~$\rep{a}$ for $\iota\pi(a)$, and have $\rep{a}=1_{\rep{s_0(a)}}$ for~$n\geq 1$.

An \emph{$\iota$-contraction of~$X$} is a homotopy~$\sigma$ from~$\id_{\lin{X}}$ to $\iota\pi$ such that $\sigma_a = 1_a$
for every $n$-cell~$a$ of~$\lin{X}$ that belongs to the image of~$\iota$ or of~$\sigma$. The $\iota$-contraction $\sigma$ is \emph{right and bracketed} if, for every $n \geq 0$, for all $n$-cells $f, g, h$ in $\lin{X}$ with respective $0$-sources $a, b, c$, and for any $\Omega$-monomial $v$ in~$\lin{X_0}$, the following conditions hold
\begin{eqn}{align}
\label{E:RightContraction}
\sigma_{fg} &= a\sigma_g \star_0 \sigma_{f\rep{b}},
\\[0.5ex]
\label{E:BracketedContraction}
\sigma_{\left\lfloor h\right\rfloor v} &= \left\lfloor \sigma_h\right\rfloor v \star_0 \sigma_{\left\lfloor \rep{c}\right\rfloor v}.
\end{eqn}
The right-hand side of \eqref{E:BracketedContraction} is well-defined, as $t_0(\left\lfloor \sigma_h \right\rfloor v)= \left\lfloor \rep{c} \right\rfloor v =s_0(\sigma_{\left\lfloor \rep{c} \right\rfloor v})$
holds.

\begin{example}
Let us explain the right and bracketed $\iota$-contraction $\sigma$ in low dimensions. For 0-cells $a,b \text{ and } c$, we have the following
\[
\xymatrix @R=0.75em {
& {a\rep{b}}
\ar @/^/ [dr] ^-{\sigma_{a\rep{b}}} 
\ar@{} [d] |(0.70){\sm =} 
\\
ab
\ar @/^/ [ur] ^-{a\sigma_{b}} 
\ar @/_/ [rr] _-{\sigma_{ab}} 
&& {\rep{ab}}
}
\qquad\quad\raisebox{-0.8cm}{\text{and}}\qquad\quad
\xymatrix @R=0.75em {
& \left\lfloor \rep{c}\right\rfloor v
\ar@/^/ [dr] ^-{\sigma_{\left\lfloor \rep{c}\right\rfloor v}}
\ar@{} [d] |(0.70){\sm =} 
\\
\left\lfloor c\right\rfloor v
\ar@/^/ [ur] ^-{\left\lfloor \sigma_c\right\rfloor v}
\ar@/_/ [rr] _{\sigma_{\left\lfloor c\right\rfloor v}}
&& {\rep{\left\lfloor c\right\rfloor v} } \;.
}
\]
For 1-cells $f:a\fl b$, $g:c\fl d$ and a 0-cell $v$, we have the following
\[
\vcenter{\xymatrix @R=0.75em {
& bd
\ar@/^/ [dr] ^-{\sigma_{bd}}
\ar@2 []!<0pt,-10pt>;[d]!<0pt,2.5pt> ^-*+{\sigma_{fg}}
\\
ac
\ar@/^/ [ur] ^-{fg}
\ar@/_/ [rr] _-{\sigma_{ac}}
&& {\rep{ac}}
}}
\qquad = \qquad
\vcenter{\xymatrix @R=1.5em @C = 3.5em {
&& bd
\ar @/^5ex/ [ddrr] ^-{\sigma_{bd}} _-{}="B"
\ar [dr] |-*+{b\sigma_{d}}
\ar@{} [dd] |-{\sm =}
\\
& ad
\ar [ur] |-*+{fd}
\ar [dr] |-*+{f\sigma_{d}}
\ar@2 []!<-5pt,-12pt>;[d]!<-5pt,11pt> ^-*+{a\sigma_g}
&& {b\rep{c}}
\ar [dr] |-*+{\sigma_{b\rep{c}}}
\ar@2 []!<-5pt,-12pt>;[d]!<-5pt,11pt> ^-*+{\sigma_{f\rep{c}}}
\ar@{} [];"B" |(0.6){\sm =}
\\
ac 
\ar @/^5ex/ [uurr] ^-{fg} _-{}="A"
\ar@{} [ur];"A" |(0.5){\sm =}
\ar [ur] |-*+{ag}
\ar [rr] |-*+{a\sigma_c}
\ar @/_5ex/ [rrrr] _-{\sigma_{ac}} ^-{}="C"
&& {a\rep{c}}
\ar [ur] |-*+{f\rep{c}}
\ar [rr] |-*+{\sigma_{a\rep{c}}}
\ar@{} [];"C" |-{\sm =}
&& {\rep{ac}} \;,
}}
\]
\[
\vcenter{\xymatrix @R=0.75em {
& \left\lfloor b\right\rfloor v
\ar@/^/ [dr] ^-{\sigma_{\left\lfloor b\right\rfloor v}}
\ar@2 []!<0pt,-10pt>;[d]!<0pt,2.5pt> ^-*+{\sigma_{\left\lfloor f\right\rfloor v}}
\\
\left\lfloor a\right\rfloor v
\ar@/^/ [ur] ^-{\left\lfloor f\right\rfloor v}
\ar@/_/ [rr] _-{\sigma_{\left\lfloor a\right\rfloor v}}
&& {\rep{\left\lfloor a\right\rfloor v}}
}}
= 
\vcenter{\xymatrix @R=3em @C=4em {
& \left\lfloor b\right\rfloor v
\ar [dr] |-*+{\scriptstyle \left\lfloor \sigma_b\right\rfloor v}
\ar@2 []!<-9pt,-12pt>;[d]!<-9pt,11pt> ^-*+{\scriptstyle \left\lfloor \sigma_f\right\rfloor v}
\ar @/^5ex/ [drrr] ^-{\scriptstyle \sigma_{\left\lfloor b\right\rfloor v}} _-{}="B"
& & &
\\
\left\lfloor a\right\rfloor v
\ar@/^/ [ur] |-*+{\scriptstyle \left\lfloor f\right\rfloor v}
\ar [rr] |-*+{\scriptstyle \left\lfloor \sigma_a\right\rfloor v}
\ar @/_5ex/ [rrrr] _-{\scriptstyle \sigma_{\left\lfloor a\right\rfloor v}} ^-{}="C"
& 
& \left\lfloor \rep{a}\right\rfloor v
\ar [rr] |-*+{\scriptstyle \sigma_{\left\lfloor \rep{a}\right\rfloor v}}
\ar@{} [];"C" |-{\sm =}
\ar@{} [ur] |(0.45){\sm =}
&& \rep{\left\lfloor a\right\rfloor v} \;.
}}
\] 
\end{example}

\subsection{Operated polygraphic resolutions from convergence}
This subsection presents the main result of this paper, Theorem~\ref{T:Main conclusion}, which constructs a polygraphic resolution for an $\Omega$-algebra from a convergent presentation.
Such a resolution \emph{à la Squier} is generated in each dimension $n\geq 2$ by the sources of the critical $n$-branchings of the presentation.

\subsubsection{Reduced and essential $\Omega$-monomials}\label{SSS:EssentialMonomial}
Let $\iota$ be a unital section of an \LO-$\omega$-polygraph $X$. An $\Omega$-monomial~$u$ of~$\lin{X_0}$ is \emph{$\iota$-reduced} if $u=\rep{u}$. A non-$\iota$-reduced $\Omega$-monomial~$u$ is \emph{$\iota$-essential} if it satisfies one of the following two conditions:
\begin{enumerate}
\item $u=xv$, where $x$ is a $0$-generator of~$X$ and $v$ is an $\iota$-reduced $\Omega$-monomial of~$\lin{X_0}$,
\item $u=\left\lfloor w\right\rfloor  v$, where $w$ and $v$ are both $\iota$-reduced $\Omega$-monomials of~$\lin{X_0}$.
\end{enumerate}

If~$\sigma$ is an $\iota$-contraction of~$X$, for~$n\geq 1$, an $n$-cell~$a$ of~$\lin{X}$ is \emph{$\sigma$-reduced} if it is an identity or in the image of~$\sigma$. If the $\iota$-contraction $\sigma$ is right and bracketed, for~$n\geq 1$, a non-$\sigma$-reduced $n$-monomial~$a$ of~$\lin{X}$ is \emph{$\sigma$-essential} if 
\begin{enumerate}[resume]
\item $u=\alpha v$, where $\alpha$ is an $n$-generator of~$X$ and $v$ is an $\iota$-reduced $\Omega$-monomial of~$\lin{X_0}$.
\end{enumerate}

\begin{lemma}
\label{L:RightAndBracketed}
Let $X$ be an \LO-$\omega$-polygraph and $\iota$ be a unital section of $X$. A right and bracketed $\iota$-contraction $\sigma$ of $X$ is
uniquely determined by its values on the $\iota$-essential $\Omega$-monomials of~$\lin{X_0}$ and, for $n>1$, on the $\sigma$-essential $\Omega$-monomials of $\lin{X}$. 
\end{lemma}
\begin{proof}
By Lemma~\ref{L:LemmaOfHomotopy}, it suffices to check that the values of $\sigma$ on the $\iota$-essential $\Omega$-monomials and on $\sigma$-essential $\Omega$-monomials determine its values on other $\Omega$-monomials and $n$-monomials, and that equation \eqref{E:HomotopyExchange} holds.
If $u$ is a non-$\iota$-essential $\Omega$-monomial, we consider three cases:
\begin{enumerate}
\item $u=1$. Then $\sigma_1 = 1$ is forced since $1$ is $\iota$-reduced.
\item $u=yv$, where $y$ is either a $0$-generator of $X$ or an $\iota$-reduced monomial $\lfloor w \rfloor$, and $v$ is a non-$\iota$-reduced $\Omega$-monomial of $\lin{X_0}$. Then \eqref{E:RightContraction} imposes $\sigma_{yv} = y\sigma_v \star_0 \sigma_{y\rep{v}}$. We proceed by induction on the length of $v$ to define $\sigma_v$.
\item $u=\left\lfloor w\right\rfloor v$, where $w$ is a non-$\iota$-reduced $\Omega$-monomial of $\lin{X}$ and $v$ is an arbitrary $\Omega$-monomial. Then \eqref{E:BracketedContraction} imposes $\sigma_{\left\lfloor w\right\rfloor v} = \left\lfloor \sigma_w \right\rfloor v \star_0 \sigma_{\left\lfloor \rep{w}\right\rfloor v}$. When $v$ is $\iota$-reduced, $\left\lfloor \rep{w}\right\rfloor v$ is also $\iota$-reduced. We define~$\sigma_w$ by induction on the number of operators in $w$ and its length. When $v$ is not $\iota$-reduced, we define $\sigma_{\left\lfloor \rep{w}\right\rfloor v}$ based on the previous case.
\end{enumerate}
If $w$ is an $n$-monomial, then $w$ can be written as $(p_1\circ p_2 \circ \cdots \circ p_n)|_\alpha$,
where  $\alpha$ is a 1-generator with 0-source $a$ and $p_k\coloneq\left\lfloor u_k \square  v_k\right\rfloor_{\tau_k}\in X_0^\Omega[\square]$ with  $u_k, v_k \in X_0^\Omega$. We distinguish between two cases:
\begin{enumerate}
\item $w = u\alpha v$. Then \eqref{E:RightContraction} imposes $\sigma_{u\alpha v} = u a \sigma_v \star_0 u\sigma_{\alpha \rep{v}} \star_0 \sigma_{u \rep{a v}}$. We proceed by induction on the length of $u$ and $v$ to define $\sigma_{u \rep{a v}}$ and $\sigma_v$.
\item $w = u\left\lfloor \left.q\right|_{\alpha}\right\rfloor v$. Then \eqref{E:RightContraction} imposes $\sigma_{u \left\lfloor \left.q\right|_{\alpha}\right\rfloor v} = u \left\lfloor \left.q\right|_{a}\right\rfloor \sigma_v \star_0 u\sigma_{\left\lfloor \left.q\right|_{\alpha}\right\rfloor \rep{v}} \star_0 \sigma_{u \rep{\left\lfloor \left.q\right|_{a}\right\rfloor v}}$. We proceed by induction on the length of $u$ and $v$ to define $\sigma_v$ and $\sigma_{u \rep{\left\lfloor \left.q\right|_{a}\right\rfloor v}}$.
As for $\sigma_{\left\lfloor \left.q\right|_{\alpha}\right\rfloor \rep{v}}$,  \eqref{E:BracketedContraction} imposes $\sigma_{\left\lfloor \left.q\right|_{\alpha}\right\rfloor \rep{v}} = \left\lfloor \sigma_{\left.q\right|_{\alpha}}\right\rfloor \rep{v} \star_0 \sigma_{\left\lfloor \rep{\left.q\right|_{a}}\right\rfloor \rep{v}}$. We then proceed by induction on the number of operators and the length of $\left.q\right|_{\alpha}$ to define $\sigma_{\left.q\right|_{\alpha}}$.
\end{enumerate}
To verify \eqref{E:HomotopyExchange}, we refer to the case of associative algebras in {\cite[Lem.~5.2.5]{GuiraudHoffbeckMalbos19}} and that of the shuffle operad in {\cite[Lem.~5.1.6]{MalbosRen23}}, as their proofs do not differ significantly.
\end{proof}

\begin{proposition}
\label{T:ResolutionContraction}
Let $X$ be an \LO-$\omega$-polygraph with a fixed unital section $\iota$. Then $X$ is a polygraphic
resolution of $\overline{X}$ if and only if it admits a right and bracketed $\iota$-contraction.
\end{proposition}
\begin{proof}
Assume that $X$ is a polygraphic resolution of $\overline{X}$. We define a right and bracketed $\iota$-contraction using Lemma~\ref{L:RightAndBracketed}. 
For an essential monomial $xv$, both $xv$ and $\rep{xv}$ map to the same element in $\lin{X_0}$, so there exists a 1-cell $\sigma_{xv}: xv \to \rep{xv}$ in $\lin{X}$.
Similarly, for an essential monomial of the form $\left\lfloor w \right\rfloor v$, $\sigma_{\left\lfloor w \right\rfloor v}$ is defined.
Assume that $\sigma$ is defined on the~$n$-cells of $\lin{X}$, for $n\geq 1$, and let $\alpha v$ be a $\sigma$-essential $n$-monomial of $\lin{X}$.
By hypothesis, $s(\sigma_{\alpha v})$ and $t(\sigma_{\alpha v})$ are parallel, ensuring the existence of an $(n+1)$-cell $\sigma_{\alpha v}$.

Conversely, for $n \geq 1$, let $\sigma$ be a right and bracketed $\iota$-contraction, and let $f$ and $g$ be parallel $n$-cells in $\lin{X}$. We show that $t(\sigma_f) = \sigma_{s(f)} = \sigma_{s(g)} = t(\sigma_g)$, ensuring that the $(n+1)$-cell $\sigma_f \star_n \sigma_g^-$ is well-defined.
The fact that $t_k(f) = t_k(g)$ for all $0 \leq k < n$ implies that
\[
\sigma_f \star_n \sigma_g^- \star_{n-1} \sigma_{t_{n-1}(f)}^- \star_{n-2} \cdots \star_0 \sigma_{t_0(f)}^-
\]
is a well-defined $n$-cell of $\lin{X}$ with the source $f$ and target $g$. Hence, $X_{n+1}$ is acyclic.
\end{proof}

In the following, we assume that each 0-generator of $X$ is a normal form; if not, we reduce this polygraph to a smaller one using a collapsing mechanism from {\cite[Sec.~5.3]{GuiraudHoffbeckMalbos19}}.

\subsubsection{Operated polygraphic resolutions}
\label{SSS:SquierPolygraphicResolution}
Let $X$ be a reduced convergent \LO-1-polygraph. 
We define $\mathrm{Sq}(X)$ as the family of generators $\left(\mathrm{Sq}_n(X)\right)_{n \geqslant 0}$, where the following hold.
\begin{enumerate}[label=\bf\roman*), ref=\theenumi)]
\setlist[enumerate,2]{label=\bf\alph*), ref=\theenumi-\theenumii)}
\item $\mathrm{Sq}_0(X) = X_0$.
\item  $\mathrm{Sq}_1(X)$ is the set of tuples $(u_1,u_2)$, written $u_1 | u_2$, satisfying one of the following two cases:
\begin{enumerate}
\item\label{I:Condition1OfSq1} $u_1\in X_0$, or there exists $u_0\in X_0^\Omega$ such that $u_1=\left\lfloor u_0 \right\rfloor \in \Nf(X)$, and $u_2\in \Nf(X)$,
\item \label{I:Condition1OfSq2}$u_1 = \epsilon$ and $u_2 = \left\lfloor u_0 \right\rfloor$, where $u_0\in \Nf(X)$ and $\epsilon$ denotes the empty word,
\end{enumerate}
together with the following condition
\begin{enumerate}[resume]
\item\label{I:Condition1OfSq3} $u_1 u_2$ is reducible, and every proper left-factor of $u_1 u_2$ is a normal form.
\end{enumerate}
\item  For $n \geq 2$, $\mathrm{Sq}_n(X)$ is the set of tuples $(u_1, \ldots, u_{n+1})$, written $u_1 | \cdots | u_{n+1}$, such that the following conditions hold:
\begin{enumerate}
\item $(u_1,u_2) \in \mathrm{Sq}_1(X)$ and satisfies conditions {\bf \ref{I:Condition1OfSq1}} and {\bf \ref{I:Condition1OfSq3}};
\item $u_i\in \Nf(X)$, for every $i>2$;
\item for every $2 \leq i< n+1$, $u_i u_{i+1}$ is reducible, and every proper left-factor of $u_i u_{i+1}$ is a normal form.
\end{enumerate}
\end{enumerate}

\begin{theorem}
\label{T:Main conclusion}
Let $X$ be a reduced convergent \LO-$1$-polygraph. There exists a unique polygraphic structure on $\mathrm{Sq}(X)$ and a unique unital section $\iota$, as well as a right and bracketed $\iota$-contraction $\sigma$ of $\mathrm{Sq}(X)$, such that $\iota\pi(u) = \widehat{u}$ for every $u\in X_0^\Omega$, and
\begin{eqn}{equation}
\label{E:ValueOfSigamOnEssentialMonomial}
\sigma_{\left(u_1|\cdots| u_n\right) u_{n+1}} = \begin{cases}
u_1|\cdots| u_{n+1} & \text{if } u_1|\cdots| u_{n+1} \in \mathrm{Sq}_n(X), \\
1_{\left(u_1|\cdots| u_n\right) u_{n+1}} & \text{if } u_nu_{n+1}\in \Nf(X),
\end{cases}
\end{eqn}
for all $(n-1)$-generators $u_1 | \cdots | u_n$ in $\mathrm{Sq}_{n-1}(X)$, and $u_{n+1} \in \Nf(X)$ with $n \geq 1$. When $n=1$, we also allow the notation $u_1 = \brca{u_0} \in \Nf(X)$. 
We also set
\begin{eqn}{equation}
\label{E:ExtraCases}
\sigma_{\brca{u_0}}=\epsilon|\brca{u_0} \quad \text{and} \quad \sigma_{(\epsilon|\brca{u_0})u_3}=1_{(\epsilon|\brca{u_0})u_3},
\end{eqn}
for every reducible $\Omega$-monomial $\brca{u_0}$ with $u_0\in \Nf(X)$, and  all 1-generators $\epsilon|\brca{u_0}$ with $u_3\in \Nf(X)$.
Then this structure makes $\mathrm{Sq}(X)$ a polygraphic resolution of $\lin{X_0}$.
\end{theorem}
\begin{proof}
When  \eqref{E:ValueOfSigamOnEssentialMonomial} holds, the source and target maps of $\mathrm{Sq}(X)$, except for the 1-generators $\epsilon|\brca{u_0}$, are determined by the first case. Writing $\underline{u} = u_1 | \cdots | u_n$ for short, we obtain
\[
s(u_1|\cdots|u_{n+1})=s(\sigma_{\underline{u}u_{n+1}})=\underline{u}u_{n+1}\star_0 \sigma_{t_0(\underline{u})u_{n+1}}\star_1\cdots\star_{n-2} \sigma_{t_{n-2}(\underline{u})u_{n+1}},
\]
and
\[
t(u_1|\cdots|u_{n+1})=t(\sigma_{\underline{u}u_{n+1}})=
\begin{cases}
\rep{u_1u_2}
&\text{if~$n=1$,}
\\
\sigma_{s(\underline{u})u_{n+1}}
&\text{otherwise.}
\end{cases}
\]
We determine $s(\epsilon|\brca{u_0}) = \brca{u_0}$ and $t(\epsilon|\brca{u_0}) = \rep{\brca{u_0}}$ by \eqref{E:ExtraCases}.
Next, we define the values of $\sigma$ on $n$-cells of $\lin{\mathrm{Sq}(X)}$. According to Lemma~\ref{L:RightAndBracketed}, it suffices to define $\sigma$ on the $\iota$-essential $\Omega$-monomials and $\sigma$-essential $n$-monomials.

Consider the $\iota$-essential $\Omega$-monomial 
$\left\lfloor u_0\right\rfloor u_3$, where 
$u_0,u_3\in\Nf(X)$ and $\left\lfloor u_0\right\rfloor$ is reducible.
If $u_3$ is the identity, we have $\sigma_{\left\lfloor u_0\right\rfloor}=\epsilon|\left\lfloor u_0\right\rfloor$. 
Otherwise, \eqref{E:ExtraCases} reads $\sigma_{(\epsilon|\brca{u_0})u_3}=1_{(\epsilon|\brca{u_0})u_3}$, so the source and target of~$\sigma_{(\epsilon|\brca{u_0})u_3}$ must be equal, giving the value of $\sigma$ on $\left\lfloor u_0\right\rfloor u_3$:
\[
\sigma_{\left\lfloor u_0\right\rfloor u_3}=t(\sigma_{(\epsilon|\brca{u_0})u_3})=s(\sigma_{(\epsilon|\brca{u_0})u_3})=(\epsilon|\brca{u_0})u_3 \star_0 \sigma_{\rep{\left\lfloor u_0\right\rfloor}u_3}.
\]

Consider the $\iota$-essential $\Omega$-monomial 
$u_1u_2$, where $u_1\in X_0$ or $u_1=\brca{u_0}\in\Nf(X)$, $u_2\in\Nf(X)$, and $u_1u_2$ is reducible. If $u_1|u_2\in\mathrm{Sq}_1(X)$, then \eqref{E:ValueOfSigamOnEssentialMonomial} imposes $\sigma_{u_1u_2}=u_1|u_2$. If not, there exists a proper factorization $u_2=v_2w_2$ such that $u_1|v_2\in\mathrm{Sq}_1(X)$. We have $\sigma_{(u_1|v_2)w_2}=1_{(u_1|v_2)w_2}$ by the second case in \eqref{E:ValueOfSigamOnEssentialMonomial}, so
\[
\sigma_{u_1u_2}
= t(\sigma_{(u_1| v_2)w_2})
= s(\sigma_{(u_1| v_2)w_2})
= (u_1| v_2)w_2 \star_0 \sigma_{\rep{u_1v_2}w_2}.
\]

Now, consider $n\geq 2$. For the $\sigma$-essential monomial $\underline{u}u_{n+1}$, where $\underline{u}$ is an $(n-1)$-generator of $\mathrm{Sq}(X)$, and $u_{n+1}\in\Nf(X)$, we distinguish four cases. First, for $n=1$, if $\underline{u}u_3=(\epsilon|\brca{u_0})u_3$, then \eqref{E:ExtraCases} reads $\sigma_{(\epsilon|\brca{u_0})u_3}=1_{(\epsilon|\brca{u_0})u_3}$.
Second, if $\underline{u}| u_{n+1}\in \mathrm{Sq}(X)$, then \eqref{E:ValueOfSigamOnEssentialMonomial} imposes $\sigma_{\underline{u}u_{n+1}}=\underline{u}| u_{n+1}$. Third, if~$u_n u_{n+1}\in\Nf(X)$, then \eqref{E:ValueOfSigamOnEssentialMonomial} imposes $\sigma_{\underline{u}u_{n+1}}=1_{\underline{u}u_{n+1}}$. Otherwise, there exists a proper factorization $u_{n+1}=v_{n+1}w_{n+1}$ such that $\underline{u}| v_{n+1}\in\mathrm{Sq}(X)$. In that case, \eqref{E:ValueOfSigamOnEssentialMonomial} implies that the source and the target of~$\sigma_{(\underline{u}| v_{n+1})w_{n+1}}$ are equal. On the one hand,
\[
s(\sigma_{(\underline{u}| v_{n+1})w_{n+1}})
= (\underline{u}| v_{n+1}) w_{n+1} \star_0 \sigma_{t_0(\underline{u}|
v_{n+1})w_{n+1}} \star_1 \cdots \star_{n-2} \sigma_{t_ {n-2}(\underline{u}| v_{n+1})
w_{n+1}},
\]
and, on the other hand,
\[
t(\sigma_{(\underline{u}| v_{n+1})w_{n+1}})
= \sigma_{s(\underline{u}| v_{n+1}) w_{n+1}} 
= \sigma_{s(\sigma_{\underline{u} v_{n+1}}) w_{n+1}} 
= \sigma_{\underline{u}u_{n+1} \star_0 \sigma_{t_0(\underline{u})v_{n+1}} w_{n+1}
\star_1 \cdots\star_{n-2} \sigma_{t_{n-2}(\underline{u})v_{n+1}} w_{n+1}}.
\]
Since $a \star_k b = a + b - t_k(a)$ and $\sigma$ commutes with $t$, we can rewrite the latter expression, by induction on~$n$, as a linear composition of $n$-cells containing $\sigma_{\underline{u}u_{n+1}}$, $\sigma_{\sigma_{t_{n-2}(\underline{u})v_{n+1}} w_{n+1}}$, and other lower-dimensional cells. Thus, $\sigma_{\underline{u}u_{n+1}}$ can be determined from the relation $s(\sigma_{(\underline{u}| v_{n+1})w_{n+1}}) = t(\sigma_{(\underline{u}| v_{n+1})w_{n+1}})$.

Finally,  by Proposition~\ref{T:ResolutionContraction}, we conclude that $\mathrm{Sq}(X)$ is a polygraphic resolution of $\overline{X}$.
\end{proof}

\subsubsection{Example}
\label{E:InvolutiveAlgebra1}
Consider the \LO-1-polygraph $X$ with $X_1\coloneq \{\beta_u:\brca{\brca{u}}\fl u \text{ }|\text{ } u\in X_0^\Omega \}$.
The convergence of $X$ follows from Example~\ref{E:InvolutiveAlgebra}.
In order to construct $\Nf(X)$, we consider
\[
\Phi = \bigcup_{n \geq 0} \Phi_n,
\quad\text{with}\quad
\Phi_0 = \left(X_0\cup \brca{X_0^\ast} \right)^\ast
\quad \text{and} \quad 
\Phi_n = \left(X_0\cup\left\lfloor\left( \Phi_{n-1}\right)_{\geq 2}\right\rfloor\right)^\ast
\;\text{for $n \geq 1$,}
\]
where $\left( \Phi_{n-1}\right)_{\geq 2}$ denotes the set of all $\Omega$-monomials $u$ in $\Phi_{n-1}$ satisfying $\operatorname{bre}(u) \geq 2$.  
Note that the construction of $\Phi$ excludes all $\Omega$-monomials of the form $\left.q\right|_{\brca{\brca{u}}}$ in $X_0^\Omega$, so we have $\Nf(X) = \Phi$.
We then consider the reduced polygraph $\widetilde{X} = (X_0, \widetilde{X}_1)$, which is Tietze equivalent to $X$, where
\[
\widetilde{X}_1\coloneq\{\beta_u:\brca{\brca{u}}\fl u \mid u\in(\Phi)_{n \geq 2} \text{ or } u\in X_0\}.
\]
Thus, the $\Omega$-algebra presented by $X$ has the polygraphic resolution $\mathrm{Sq}(\widetilde{X})$, where
\begin{enumerate}
\item $\mathrm{Sq}_0(\widetilde{X})=X_0$,
\item $\mathrm{Sq}_1(\widetilde{X})$ has the 1-generators $\epsilon |\brca{\brca{u}}:\brca{\brca{u}}\fl u, \text{ for all } u\in  (\Phi)_{n \geq 2} \text{ or } u\in X_0$,
\item for every $n>1$, $\mathrm{Sq}_n(\widetilde{X})$ is empty.
\end{enumerate}

\section{Polyautomata and operated polygraphs}
\label{S:Polyautomata}
In this section, we introduce the structure of polyautomata as a way to encode the operator structure of \LO-1-polygraphs. This construction enables us to interpret the critical branchings of $\Omega$-$1$-polygraphs as critical branchings of linear $1$-polygraphs. It thus provides a simpler description of the higher critical branchings involved in polygraphic resolutions, and yields an effective procedure for constructing such resolutions, as illustrated in \eqref{SSS:InterpretationForRBResolution}.

The structure of polyautomata also allows us to compute normal forms of \LO-1-polygraphs, making explicit the normal forms appearing, for instance, in Example~\ref{D:DiffAutomaton}.
Moreover, it can be used to reduce presentations prior to the computation of polygraphic resolutions: this is illustrated in Subsection~\ref{SS:RotaBaxterResolutions} for free Rota--Baxter algebras, and in Subsection~\ref{SS:PolygraphicResolutionsDifferentialRotaBaxterAlgebras} for free differential Rota--Baxter algebras.

\subsection{Pushdown automata}
Automata are models of computation based on the notion of state.
Among them, \emph{pushdown automata (PDAs)} use a last-in-first-out stack to process \emph{context-free languages}. Their state transitions depend on both the input symbol and the top of the stack, enabling them to handle nested structures such as operator brackets.
In our setting, the stack records the operator
brackets that have been opened but not yet closed, allowing the PDA to verify
the nesting structure of operators.

\subsubsection{}
\label{SSS:PDA}
A \emph{pushdown automaton (PDA) on $\Sigma_0$} is a tuple $\mathbb{A}=(Q, \Sigma_0, \Gamma, \delta, q_0, F)$, where:
\begin{enumerate}
    \item $Q$ is a finite set of internal states.
    \item $\Sigma_0$ is the input alphabet.
    \item $\Gamma$ is a finite set of symbols called the \emph{stack alphabet}.
    \item $\delta: Q \times \{\Sigma_0\cup \epsilon\} \times \{\Gamma\cup \epsilon\cup \$\} \to Q \times \Gamma^*$ is the state transition function, where $\epsilon$ is the empty word and $\$$ is the initial bottom-of-stack marker with $\$\notin\Gamma$. In a given state, the PDA reads both the input symbol and the top symbol of the stack, then transitions to a new state and updates the top of the stack.
    \item $q_0 \in Q$ is the initial state.
    \item $F \subseteq Q$ is the set of accepting states.
\end{enumerate}

A \emph{monomial accepted by $\Ar$} is a word $ w = a_1 \cdots a_n$ in $\Sigma_0^* $ such that there exists a finite sequence of valid transitions
\[
(q_0, \$)
\xrightarrow[\;a_1\;]{} (q_1, \Gamma_1)
\xrightarrow[\;a_2\;]{} \cdots 
\xrightarrow[\;a_n\;]{} (q_n, \Gamma_n),
\]
with $q_n \in F$.
We denote by \emph{$\Sigma_0^{\mathbb{A}}$} the set of monomials accepted by $\Ar$. Note that~$\Sigma_0^{\mathbb{A}}$ does not form a monoid in general. 
We further denote by $\bk \Sigma_0^\Ar$ the set of \emph{polynomials accepted by $\Ar$}, consisting of linear combinations of monomials in $\Sigma_0^\Ar$.

\subsubsection{Examples}
A PDA is \emph{trivial} if it accepts all monomials in $\Sigma_0^\ast$. 
It can be pictured as follows
\[
\begin{tikzpicture}[shorten >=1pt, node distance=2.8cm, on grid, auto]
   \node[state, initial] (q_0)   {$q_0$}; 
   \node[state] (q_1) [right=of q_0] {$q_1$}; 
   \node[state, accepting] (q_2) [right=of q_1] {$q_2$};

    \path[->] 
    (q_0) edge[below]              node {$\epsilon, \epsilon \fl \$ $} (q_1)
    (q_1) edge [loop above] node {$\Sigma_0, \epsilon \fl \epsilon $} ()
          edge              node[below] {$\epsilon, \$ \fl \epsilon $} (q_2);
\end{tikzpicture}\raisebox{0.52cm}{\;\;,}
\]
where $\Sigma_0, \epsilon \fl \epsilon$ denotes the set of instructions $x, \epsilon \fl \epsilon$ for every $x\in \Sigma_0$. 
The accepting state $q_2$ is represented by a double circle.

As a nontrivial example, consider the PDA $\Ar=(\{q_0,q_1,q_2\},\{a,b\},\{\$,0\},\delta,q_0,q_2)$, with
\[
    \delta(q_0, \epsilon, \epsilon)  = (q_1, \$), \quad  \delta(q_1, a, \epsilon)  = (q_1, 0),  \quad \delta(q_1, b, 0)  = (q_1, \epsilon),  \quad
\delta(q_1,  \epsilon,\$)  = (q_2, \epsilon).
\]
Its transition diagram is the following
\[
\begin{tikzpicture}[shorten >=1pt, node distance=2.7cm, on grid, auto]
   \node[state, initial] (q_0)   {$q_0$}; 
   \node[state] (q_1) [right=of q_0] {$q_1$}; 
   \node[state, accepting] (q_2) [right=of q_1] {$q_2$};

    \path[->] 
    (q_0) edge[below]              node {$\epsilon, \epsilon \fl \$ $} (q_1)
    (q_1) edge [loop above] node {$a, \epsilon \fl 0 $; $b, 0 \fl \epsilon $} ()
          edge              node[below] {$\epsilon, \$ \fl \epsilon $} (q_2);
\end{tikzpicture}\raisebox{0.52cm}{\;\;.}
\]
This PDA accepts monomials of the form $a^n b^n$ for $n \geq 0$. The key mechanism lies in its stack operations: each $a$ pushes a $0$ onto the stack, while each $b$ pops a $0$, ensuring a balanced number of $a$'s and $b$'s. For instance, when processing the word $aabb$, the automaton follows these transitions:
\[
\begin{aligned}
    \text{Input: } & aabb, \quad \text{Stack: } \epsilon\xrightarrow{\epsilon} \$ \xrightarrow{a} \$0 \xrightarrow{a} \$00 \xrightarrow{b} \$0 \xrightarrow{b} \$ \xrightarrow{\epsilon}\epsilon.
\end{aligned}
\]
The PDA reaches the accepting state $q_2$ once the stack is emptied and all transitions are completed.

\subsection{Polyautomatic formulation of rewriting in operated algebras}
\label{SS:PolyautomaticFormulation}

\subsubsection{Polyautomata}
As mentioned in~\eqref{R:OmegaEmpty}, a linear $1$-polygraph $\Sigma=(\Sigma_0,\Sigma_1)$ is an \LO-1-polygraph with an empty set $\Omega$.
A \emph{$1$-polyautomaton} is a pair $(\Sigma,\mathbb{A})$ consisting of a linear $1$-polygraph $\Sigma$ and a PDA $\Ar$ on $\Sigma_0$, such that for every $\alpha\in \Sigma_1$, both the source $s(\alpha)$ and the target~$t(\alpha)$ are polynomials accepted by $\Ar$.
For a fixed PDA $\Ar$ on $\Sigma_0$, we denote by $\Pol_1(\Ar)$ the category
of $1$-polyautomata $(\Sigma,\Ar)$ with a fixed set of $0$-cells $\Sigma_0$.

\subsubsection{Bracket polyautomaton}
\label{SSS:OperatedRewriting}
The \emph{bracket $1$-polyautomaton on a set $X_0$ and an indexed set
$\Omega$} is the pair $(\Sigma,\mathbb{A}_{\Omega})$, where:
\begin{enumerate}
\item  $\Sigma_0\coloneq X_0\sqcup \mathrm{Brck}(\Omega)$, where the  \emph{bracket set}  $\mathrm{Brck}(\Omega)$ is defined as follows
\[
\mathrm{Brck}(\Omega)\:\coloneq\:\bigcup_{\tau_i\in\Omega}\{\mathsf{\ell}_{\tau_i},\mathsf{r}_{\tau_i}\}.
\]
For simplicity, we will write $\mathsf{\ell}_{\tau_i}$ and $\mathsf{r}_{\tau_i}$ as $\mathsf{\ell}_i$ and $\mathsf{r}_i$ when there is no ambiguity.
\item The PDA $\Ar_\Omega$ is illustrated by the following state transition diagram
\begin{eqn}{equation}
\label{EQ:PolyautomataForOmegaMonoids}
\raisebox{-1.5cm}{
\begin{tikzpicture}[shorten >=1pt, node distance=2.8cm, on grid, auto]
   \node[state, initial] (q_0)   {$q_0$}; 
   \node[state] (q_1) [right=of q_0] {$q_1$}; 
   \node[state, accepting] (q_2) [right=of q_1] {$q_2$};

    \path[->] 
    (q_0) edge[below]              node {$\epsilon, \epsilon \fl \$ $} (q_1)
    (q_1) edge [loop above] node {$\mathsf{\ell}_i, \epsilon \fl i; \;\mathsf{r}_j, j \fl \epsilon $} ()
          edge [loop below] node {$X_0, \epsilon \fl \epsilon $} ()
          edge              node[below] {$\epsilon, \$ \fl \epsilon $} (q_2);
\end{tikzpicture}\raisebox{1.51cm}{\;\;.}
}
\end{eqn}
\end{enumerate}

\begin{lemma}
\label{L:OmegaIsomorphic}
The set $\Sigma_0^{\Ar_\Omega}$ with concatenation as its operation forms a monoid, isomorphic  to the free $\Omega$-monoid $X_0^\Omega$. Moreover, this  induces an isomorphism $\psi$ of algebras between $\lin{X_0}$ and $\bk\Sigma_0^{\Ar_{\Omega}}$.
\end{lemma}
\begin{proof}
First, we define a map $\psi: \Sigma_0^{\Ar_\Omega} \rightarrow X_0^\Omega$ as follows. For every $u \in \Sigma_0^{\Ar_\Omega}$ and $\tau_i \in \Omega$, $\psi(u)$ is obtained by replacing each $\mathsf{\ell}_i$ with the left bracket $\lfloor$ and each $\mathsf{r}_i$ with the right bracket $\rfloor_{\tau_i}$ of the bracket $\brck{~}{\tau_i}$. For example,
\[
\psi(\mathsf{\ell}_1 x \mathsf{r}_1 \mathsf{\ell}_2 y \mathsf{r}_2) = \brck{x}{\tau_1}\brck{y}{\tau_2}.
\]
In particular, we set $\psi(\mathsf{\ell}_i \mathsf{r}_i) = \brck{1}{\tau_i}$ and $\psi(\epsilon) = 1$.

The map $\psi$ is surjective. Indeed, as shown in \eqref{EQ:PolyautomataForOmegaMonoids}, state $q_0$ transitions to $q_1$ by reading the empty string $\epsilon$, initializing the stack with $\$$. To reach $q_2$ from $q_1$, the stack must remain unchanged as $\$$.
In particular, a direct transition from $q_1$ to $q_2$ without additional instructions results in an output of~$\epsilon$.
At~$q_1$, by repeatedly reading instructions of the form $x, \epsilon \mapsto \epsilon$ for all $x \in X_0$, $\Ar_\Omega$ can output any monomial in $X_0^\ast$ while keeping the stack unchanged as $\$$, and then transition to $q_2$ to stop the process. Thus, $\psi(X_0^\ast) = (X_0)_0^\Omega$, where $(X_0)_0^\Omega$ is defined in \eqref{SSS:FreeOmegaObjects}.
Alternatively, at $q_1$, $\Ar_\Omega$ may first read the instruction $\mathsf{\ell}_i, \epsilon \mapsto i$, output $\mathsf{\ell}_i$, and push $i$ onto the stack. Since the presence of $i$ in the stack does not interfere with the instructions $x, \epsilon \mapsto \epsilon$, $\Ar_\Omega$ can continue to output any monomial in $X_0^\ast$. Finally, by reading $\mathsf{r}_i, i \mapsto \epsilon$, $\Ar_\Omega$ outputs a monomial of the form $\mathsf{\ell}_i X_0^\ast \mathsf{r}_i$, removes $i$ from the stack, and reaches $q_2$. 
Similarly, $\Ar_\Omega$ can accept monomials of the form
\[
(X_0 \sqcup \mathsf{\ell}_{i_n} \cdots \mathsf{\ell}_{i_1} X_0^\ast \mathsf{r}_{i_1} \cdots \mathsf{r}_{i_n})^\ast
\]
Thus, $\psi\left((X_0 \sqcup \mathsf{\ell}_{i_n} \cdots \mathsf{\ell}_{i_1} X_0^\ast \mathsf{r}_{i_1} \cdots \mathsf{r}_{i_n})^\ast\right) = (X_0)_1^\Omega$. By similar reasoning, for every subset $(X_0)_k^\Omega \subseteq X_0^\Omega$, we can construct its preimage under $\psi$.
Since $X_0^\Omega$ is defined by $X_0^\Omega \coloneq \varinjlim (X_0)_n^\Omega$, we conclude that $\psi$ is surjective. 

The injectivity of $\psi$ and its compatibility with products are straightforward. It follows that $\psi : \Sigma_0^{\Ar_\Omega} \to X_0^\Omega$  
is a monoid isomorphism which we extend by linearity to an isomorphism $\psi : \bk\Sigma_0^{\Ar_\Omega} \to \lin{X_0}$ of algebras.
\end{proof}

\subsubsection{Polyautomaton of an operated rewriting system}
\label{SSS:PolygraphVsPolyautomaton}
Given an \LO-$1$-polygraph $X$, we define the $1$-polyautomaton $\Sigma(X) := (\Sigma_0, \Sigma_1, \Ar_\Omega)$, by setting $\Sigma_0 \coloneq  X_0 \sqcup \mathrm{Brck}(\Omega)$ and
\[
\Sigma_1 \: \coloneq \: \{\psi^{-1}(\alpha): \psi^{-1}(a) \fl\psi^{-1}(b), \text{ for all } \alpha: a \fl b \in X_1 \},
\]
with $\psi$ defined in Lemma~\ref{L:OmegaIsomorphic}. Since $\psi^{-1}(a), \psi^{-1}(b) \in \bk \Sigma_0^{\Ar_\Omega}$, the cellular extension $\Sigma_1$ is well-defined.
We have an isomorphism of algebras
\[
\bk\Sigma_0^{ \Ar_\Omega} / \Sigma_1 \cong \overline{X}.
\]

\subsubsection{Example}
\label{D:DiffAutomaton}
According to \eqref{E:ConvergentOfDifferentialAlgebra} and the fact that $\left(D^\theta(X_0)\right)^\ast$ is a linear basis of the free differential algebra~$\DA{\lambda}{X_0}$, the following PDA, denoted by $\Ar^{\scriptscriptstyle D}$, accepts the normal forms of $\DA{\lambda}{X_0}$
\begin{eqn}{equation}
\label{D:PDAOfDiff}
\raisebox{-1.6cm}{
\begin{tikzpicture}[shorten >=1pt, node distance=2.8cm, on grid, auto]
    \node[state, initial] (q0) {$q_0$};
    \node[state] (q1) [right=of q0] {$q_1$};
    \node[state] (q2) [below left=of q1] {$q_2$};
    \node[state] (q3) [below right=of q1] {$q_3$};
    \node[state, accepting] (q5) [right=of q1] {$q_5$};
    \path[->]  (q0) edge node [above] {$\varepsilon, \varepsilon \to \$$} (q1);
    \path[->] (q1) edge node [above] {$\varepsilon, \$ \to \varepsilon$} (q5);
    \path[->] (q1) edge node [left] {$\varepsilon, \varepsilon \to \varepsilon$} (q2);
    \path[->] (q3) edge node [right] {$\varepsilon, \$ \to \$$} (q1);
    \path[->] (q2) edge [loop left] node {$\mathsf{\ell}_{\scriptscriptstyle D}
, \varepsilon \to D$} (q2);
    \path[->] (q3) edge [loop right] node {$\mathsf{r}_{\scriptscriptstyle D}
, D \to \varepsilon$} (q3);
    \path[->] (q2) edge node [above] {$X_0, \varepsilon \to \varepsilon$} (q3);
\end{tikzpicture}\raisebox{0.51cm}{\;\;.}
}
\end{eqn} 
A differential algebra $A$ can thus be presented by a $1$-polyautomaton $(\Sigma, \Ar^{\scriptscriptstyle D})$, where $\Sigma_0 \coloneq X_0 \sqcup \mathrm{Brck}(\Omega)$ and $\Sigma_1$ is its set of defining relations. 
The algebra $A$ is thus isomorphic to the quotient algebra~$\bk\Sigma_0^{\Ar^{\scriptscriptstyle D}} / \Sigma_1$. 
For instance, for commutative differential algebras as in \cite{ChenChenLi}, we set $\Sigma_1 \coloneq \{uv \to vu \mid u,v \in \Sigma_0^{\Ar^{\scriptscriptstyle D}}\}$.

\begin{theorem}
\label{T:OmegaCBshape}
There exists a one-to-one correspondence between the sets of critical branchings  $\mathrm{CB}(X)$ and $\mathrm{CB}(\Sigma(X))$.
\end{theorem}
\begin{proof}
We consider a critical branching $(f,g) \in \mathrm{CB}(X)$, where $f: a \to b$ and $g: a \to c$. We map $(f,g)$ to $(f', g')$ in $\mathrm{CB}(\Sigma(X))$,  where $f': \psi^{-1}(a) \to \psi^{-1}(b)$ and $g': \psi^{-1}(a) \to \psi^{-1}(c)$.  
By Lemma~\ref{L:OmegaIsomorphic}, the bijective map $\psi^{-1}$ induces a correspondence between $\mathrm{CB}(X)$ and $\mathrm{CB}(\Sigma(X))$.
\end{proof}

\begin{remark}
\label{R:CriticalBranchingForLinearPolyautomaton}
Following Theorem~\ref{T:OmegaCBshape}, the intersection and inclusion branchings of $X$ have the following shapes, respectively
\begin{eqn}{equation}
\label{D:TwoCriticalBranchingsShapes}
\xymatrix @C=4em {
\strut 
\ar@{..>} [r] _-{u}
\ar@{..>}@/^7ex/ [rr] ^-{}="aw"
& \strut
\ar@{..>} [r] ^-{v}
\ar@{..>}@/_7ex/ [rr] _-{}="ub"
& \strut
\ar@{..>}[r] ="w"^-{w}
& \strut
\ar@1 "1,2";"aw"!<0pt,3pt> ^-{\alpha}
\ar@1 "1,3";"ub"!<0pt,7.5pt> ^-{\beta}
}
\qquad
\xymatrix @C=4em {
\strut 
\ar@{..>} [r] _-{u'}
\ar@{..>}@/^7ex/ [rrr] ^-{}="aw"
& \strut
\ar@{..>} [r] |-{v'}="jd"
\ar@{..>}@/_7ex/ [r] _-{}="ub"
& \strut
\ar@{..>}[r] ="w"_-{w'}
& \strut
\ar@1 "jd"!<0pt,7.5pt>;"aw"!<0pt,-7.5pt> _-{\alpha}
\ar@1 "jd"!<0pt,7.5pt>;"ub"!<0pt,-7.5pt> ^-{\beta} 
}
\end{eqn}
where $\alpha,\beta\in\Sigma(X)_1$. 
Here $\dashrightarrow$ and $\fl$ denote the $0$-cells and $1$-cells of $\Sigma(X)$, respectively, and $\overset{a}{\dashrightarrow} \overset{b}{\dashrightarrow}$ denotes the $0$-composition $ab$.

For the inclusion branchings in \eqref{D:TwoCriticalBranchingsShapes}, although $u', w' \in \Sigma_0^\ast$, they need not belong to $\Sigma_0^{\Ar_\Omega}$ in general.
\end{remark}

\begin{lemma}
\label{L:IntersectionShapes}
Let $u, v, w \in \Sigma_0^\ast$ as in \eqref{SSS:OperatedRewriting}. If both $uv$ and $vw$ belong to $\Sigma_0^{\Ar_\Omega}$, then $u, v, w$ belong to $\Sigma_0^{\Ar_\Omega}$.
\end{lemma}
\begin{proof}
If either $uv$ or $vw$ is $\epsilon$, the result follows trivially.
Let $u = x_1 \cdots x_m$, $v = y_1 \cdots y_n$, and $w = z_1 \cdots z_k$, with $x_i, y_i, z_i \in \Sigma_0$.
Since $vw \in \Sigma_0^{\Ar_\Omega}$, the PDA~$\Ar_\Omega$ can accept the monomial $y_1 \cdots y_n z_1 \cdots z_k$. This implies that $\Ar_\Omega$ can output $y_1 \cdots y_n$ while remaining in state $q_1$ in \eqref{EQ:PolyautomataForOmegaMonoids}. At this point, the top of the stack may contain either the symbol $i$ or $\$$.

Since $uv \in \Sigma_0^{\Ar_\Omega}$, $\Ar_\Omega$ also accepts the monomial $x_1 \cdots x_m y_1 \cdots y_n$. If the top of the stack contains the symbol $i$ after outputting $y_1 \cdots y_n$, it will still contain $i$ after outputting $x_1 \cdots x_m y_1 \cdots y_n$. This prevents~$\Ar_\Omega$ from transitioning to $q_2$ and halting, which contradicts $uv \in \Sigma_0^{\Ar_\Omega}$.

Therefore, when $\Ar_\Omega$ outputs $y_1 \cdots y_n$, the top of the stack must be $\$$, ensuring that $\Ar_\Omega$ accepts $y_1 \cdots y_n$, transitions to state $q_2$, and halts. Hence, $v \in \Sigma_0^{\Ar_\Omega}$.
Finally, if either $u$ or $w$ were not in~$\Sigma_0^{\Ar_\Omega}$, it would contradict the assumption that both $uv$ and $vw$ belong to $\Sigma_0^{\Ar_\Omega}$. Thus, we conclude that $u, v, w \in~\Sigma_0^{\Ar_\Omega}$.
\end{proof}

\begin{corollary}
\label{C:IntersectionShape}
For the intersection branchings in \eqref{D:TwoCriticalBranchingsShapes},  we have $u,v,w\in \Sigma_0^{\Ar_\Omega}$.
\end{corollary}

\subsubsection{Higher critical branchings}
\label{SSS:Higher-critical branching}
For $n \geq 2$, a \emph{critical $n$-branching of $X$} is a tuple $(f_1, \ldots, f_n)$ of $1$-cells $f_i$ with the same source, such that each pair $(f_i, f_j)$ is a critical branching of $X$ for all $i \neq j$.

\subsubsection{Example}
\label{E:InvolutiveAlgebra}
Consider the \LO-1-polygraph $X$ from Example~\ref{E:InvolutiveAlgebra1}. 
We denote by $\brca{u}^k$ the $k$-fold bracketing of~$u$. Then, for $n\geq 2$, the $n$-tuple
\[
\left(\beta_{\brca{u}^{n-1}}, \brca{\beta_{\brca{u}^{n-2}}}, \brca{\beta_{\brca{u}^{n-3}}}^2,\ldots, \brca{\beta_{u}}^{n-1}\right)
\]
is a critical $n$-branching with source $\brca{u}^{n+1}$ for every $0$-cell $u$.
In particular, we have the following critical $2$-branching $(\beta_{\brca{u}}, \brca{\beta_u})$ and the critical $3$-branching $(\beta_{\brca{u}^2}, \brca{\beta_{\brca{u}}}, \brca{\beta_u}^2)$
\[
\xymatrix@!R@!C@C=4em{
\brca{\brca{\brca{u}}}
\ar@/^5ex/ [r] ^-{\brca{\beta_u}}
\ar@/_5ex/ [r] _-{\beta_{\brca{u}}}
& \brca{u}
}
\qquad\raisebox{-0.15cm}{\text{and}}\qquad
\xymatrix@!R@!C@C=6em{
\brca{\brca{\brca{\brca{u}}}}
\ar@/^5ex/ [r] ^-{\brca{\brca{\beta_u}}}
\ar@// [r] |-{\brca{\beta_{\brca{u}}}}
\ar@/_5ex/ [r] _-{\beta_{\brca{\brca{u}}}}
& \brca{\brca{u}}\;.
}
\]
All critical branchings of $X$ are confluent. 
The termination of $X$ follows from the decrease in the number of operators under the application of the rule $\beta_u$. Consequently, $X$ is convergent.

\begin{remark}
\label{R:Reducedn-Branchings}
When the polygraph $X$ is reduced, all of its critical branchings are intersection branchings as in \eqref{D:TwoCriticalBranchingsShapes}.
Indeed, the inclusion branchings in \eqref{D:TwoCriticalBranchingsShapes} imply that there exist two rewriting steps $\alpha$ and $u'\beta w'$ with source $u'v'w'$.
Therefore, for any critical $n$-branching $(f_1, \dots, f_n)$ of $X$, each pair $(f_i, f_j)$ is an intersection branching of $X$ for all $i \neq j$. For instance, we illustrate critical $3$-branchings of $X$ as follows
\[
\xymatrix @C=3em {
\strut 
\ar@{..>} [r] _-{u_1}
\ar@{..>}@/^7ex/ [rr] ^-{}="a"
& \strut
\ar@{..>} [r] ^{u_2}
\ar@{..>}@/_7ex/ [rrr] _-{}="b"
& \strut
\ar@{..>}[r] ^-{u_3}="w"
& \strut
\ar@{..>}[r] ^-{u_4}="m"
\ar@{..>}@/^7ex/ [rr] ^-{}="c"
& \strut
\ar@{..>}[r] _-{u_5}="s"
& \strut
\ar@1 "1,2";"a"!<0pt,3pt> ^-{}
\ar@1 "w"!<0pt,7.5pt>;"b"!<0pt,-7.5pt> ^-{}
\ar@1 "1,5";"c"!<0pt,3pt> ^-{}
}
\quad
\xymatrix @C=3em {
\strut 
\ar@{..>} [r] _-{u_1}
\ar@{..>}@/^7ex/ [rrr] ^-{}="a"
& \strut
\ar@{..>} [r] _{u_2}="n"
\ar@{..>}@/_7ex/ [rrr] _-{}="b"
& \strut
\ar@{..>}[r] ^-{u_3}="w"
\ar@{..>}@/^7ex/ [rrr] ^-{}="c"
& \strut
\ar@{..>}[r] _-{u_4}="m"
& \strut
\ar@{..>}[r] _-{u_5}="s"
& \strut
\ar@1 "n"!<0pt,7.5pt>;"a"!<0pt,-7.5pt> ^-{}
\ar@1 "w"!<0pt,7.5pt>;"b"!<0pt,-7.5pt> ^-{}
\ar@1 "m"!<0pt,7.5pt>;"c"!<0pt,-7.5pt> ^-{}
}
\]
\end{remark}

\section{Examples of resolutions of operated algebras}
\label{S:ExampleOfResolutionOfOmegaAlgebras}

In this final section, we apply the above constructions to some classical free $\Omega$-algebras, including free Rota--Baxter algebras \cite{BokutChenQiu, KuruschGuo}, free differential algebras \cite{BokutChenQiu,LiGuo,
GuoSitZhang}, and free differential Rota--Baxter algebras \cite{GKeigher08,BokutChenQiu,LQQZ}.
 
\subsection{Polygraphic resolutions of free Rota--Baxter algebras}
\label{SS:RotaBaxterResolutions}

This subsection presents a polygraphic resolution of the free Rota--Baxter algebra $\RBA{\lambda}{X_0}$ on $X_0$.

\subsubsection{A presentation of $\RBA{\lambda}{X_0}$}
\label{SSS:PolygraphForRB}
The algebra $\RBA{\lambda}{X_0}$ is presented by the  \LO-1-polygraph $X^P$ with 
\[
X_1^P \: \coloneq \: \big\{ \alpha[u,v]: P(u)P(v)\rightarrow P(P(u)v)+P(uP(v))+\lambda P(uv)
\; |\; u,v\in X_0^\Omega\,\big\}.
\]  
We set $N=\mathbb{Z}^2$ and define a derivation $d : X_0^\Omega \to N$ by  
\[
d(u) = \bigg(\operatorname{deg}_\Omega(u), \sum_{P|u} \big(\operatorname{deg}_\Omega(u)-\operatorname{deg}_\Omega(P)\big) \bigg),
\]
for every $u\in X_0^\Omega$, where  $P|u$ denotes each occurrence of the operator $P$ in $u$, and $\operatorname{deg}_\Omega(P)$  counts the number of operators inside the operator $P$.
For instance, $d(P(1))=(1,1)$, $d (P(x)P(y) )=(2,4)$, and $d (P (P(x)y ) ) =(2,3)$ for $x,y\in X_0$.
For every $(m_1,m_2)\in N$ and $v\in X_0^\Omega$ with $\operatorname{deg}_\Omega(v)=n$, we~set
\[
v\cdot(m_1,m_2)=(m_1 ,m_2)\cdot v=(m_1,m_2+n m_1), 
\quad
P((m_1, m_2))
=(m_1+1, m_2+m_1+1).
\]
By definition, $d$ satisfies \eqref{E:Derivation}. By equipping $d(X_0^\Omega)$ with a monotone  lexicographic order, we ensure that $d(P(u)v)$, $d(uP(v))$ and $d(uv)$ are all less than $d(P(u)P(v))$ for every $u, v \in X_0^\Omega$, with $d(1) = (0,0)$ as the minimal element.
Thus,  $X^P$ is terminating.
There are three families of critical branchings in~$X^P$:
\begin{enumerate}
\item $ (P(u)\alpha[v,w], \alpha[u,v]P(w) )$ with the source $P(u)P(v)P(w)$, 
\item $ (P ( \left.q\right |_{\alpha[u,v]} )P(w), \alpha [ \left.q\right |_{P(u)P(v)}, w ] )$ with the source  $P ( \left.q\right |_{P(u)P(v)} )P(w)$, 
\item $(P(u)P( \left.q\right |_{\alpha[v,w]}), \alpha [u,  \left.q\right |_{P(v)P(w)} ])$  with the source $P(u)P (\left.q\right |_{P(v)P(w)} )$;
\end{enumerate}
these are all confluent by a straightforward computation, similar to that in Example~\ref{E:ConvergentOfDifferentialAlgebra}. Hence, $X^P$ is convergent.

\subsubsection{Normal forms}
\label{R:PDA2}
We now construct the set $\Nf(X^P)$. Define the \emph{alternating product} of objects $U$ and $V$ with the operator~$P$ (see also \cite{KuruschGuo}) as
\[
\Lambda_P(U, V) \:\coloneq \:\left(\bigcup_{r \geq 0}(U P(V))^r U\right) \cup \left(\bigcup_{r \geq 1}(U P(V))^r\right) \cup \left(\bigcup_{r \geq 0}(P(V) U)^r P(V)\right) \cup \left(\bigcup_{r \geq 1}(P(V) U)^r\right).
\]
We introduce the following notation
\[
 \Phi_0 \:\coloneq\: X_0^\ast \setminus \{1\}, 
\quad
 \Phi_1 \:\coloneq\: \Lambda_P(\Phi_0, X_0^\ast),
\quad
\text{ and } \quad\Phi_n \:\coloneq\:  \Lambda_P(\Phi_0, \Phi_{n-1}),
\;\;\text{for $n \geq 2$},
\]
and define the set
\[
\Phi_P \:\coloneq\: \left(\bigcup_{n \geq 0} \Phi_n\right) \cup \{1\}.
\]
Thus, $\Nf(X^P)=\Phi_P$,
since there are no $\Omega$-monomials of the form $\left.q\right|_{P(u)P(v)}$ in $\Nf(X^P)$, for all~$u, v \in X_0^\Omega$ and $q\in X_0^\Omega[\square]$.

In order to make the set of normal forms explicit, as in Example~\ref{D:DiffAutomaton}, we construct a PDA $\Ar^{\scriptscriptstyle P}$
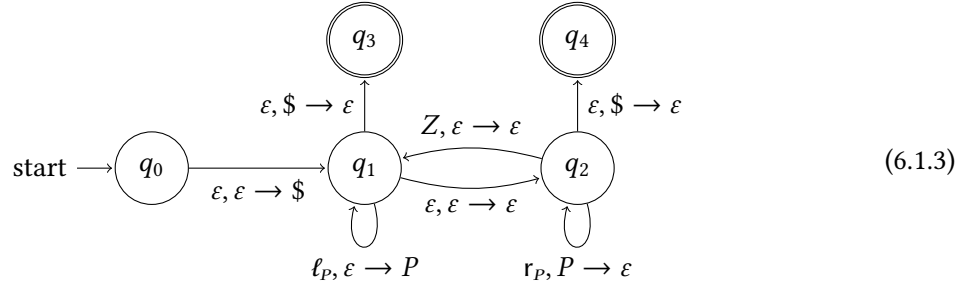
\begin{eqn}{equation}
\label{D:PDAOfRB}
\raisebox{-1.6cm}{
\begin{tikzpicture}[shorten >=1pt, node distance=2.8cm, on grid, auto]
   \node[state, initial] (q_0)   {$q_0$}; 
   \node[state] (q_1) [right=of q_0] {$q_1$}; 
   \node[state] (q_2) [right=of q_1] {$q_2$};
   \node[state,accepting] (q_3) [above=1.7cm of q_1] {$q_3$};
   \node[state,accepting] (q_4) [above=1.7cm of q_2] {$q_4$};

    \path[->] 
    (q_0) edge[below]              node {$\epsilon, \epsilon \fl \$ $} (q_1)
    (q_1) edge [loop below] node {$\mathsf{\ell}_{\scriptscriptstyle P}
, \epsilon \fl P $} ()
          edge[bend right=15] node[below] {$\epsilon, \epsilon \fl \epsilon $} (q_2)
    (q_2) edge [loop below] node {$\mathsf{r}_{\scriptscriptstyle P}
, P \fl \epsilon $} ()
          edge[bend right=15] node[above] {$X_0, \epsilon \fl \epsilon $} (q_1);

    \path[->] 
    (q_1) edge[left]              node {$\epsilon, \$ \fl \epsilon $} (q_3);

    \path[->] 
    (q_2)  edge[right]              node {$\epsilon, \$ \fl \epsilon $} (q_4);
\end{tikzpicture}
}
\end{eqn}
which accepts all monomials in $\Phi_P$.  We obtain $\Ar^{\scriptscriptstyle P}$ by modifying $\Ar_\Omega$ to exclude monomials containing the subword $\mathsf{r}_{\scriptscriptstyle P} \mathsf{\ell}_{\scriptscriptstyle P}$, since $\Phi_P = X_0^\Omega \setminus \{\left.q\right|_{P(u)P(v)}\}$ for every $u, v \in X_0^\Omega$.

The polygraph $X^P$ is not reduced, as it contains two families of inclusion branchings with sources $P(q|_{P(u)P(v)})P(w)$ and $P(u)P(q|_{P(v)P(w)})$ for all $u, v, w \in X_0^\Omega$.

\subsubsection{A reduced presentation} 
\label{SSS:LemmaForReducedRB}
We write $\rep{w}=\Nf(w,X^P)$, for every $w\in X_0^\Omega$, and construct a reduced \LO-1-polygraph $\widetilde{X}^P$, which is Tietze equivalent to $X^P$, with 
\[
\widetilde{X_1}^P\:\coloneq\: \{ \alpha[u,v]: P(u)P(v)\rightarrow P(\rep{P(u)v})+P(\rep{uP(v)})+\lambda P(\rep{uv})  \mid u,v \in \Phi_P\}.
\]
It follows that $\Nf(X^P) = \Nf(\widetilde{X}^P)=\Phi_P$.
 Indeed, it suffices to prove that  $Q(X^P)=Q(\widetilde{X}^P)$, where
\[
Q(X^P)\:\coloneq\:\{\left.q\right|_{P(u)P(v)}|u,v\in X_0^\Omega\} \quad \text{ and } \quad
Q(\widetilde{X}^P)\:\coloneq\:\{\left.q\right|_{P(u)P(v)}|u,v\in\Phi_P\}.
\]
The inclusion $Q(\widetilde{X}^P) \subset Q(X^P)$ is straightforward.
For every $\left.q_0\right|_{P(u_0)P(v_0)}$ in $Q(X^P)$, if $u_0,v_0\in\Phi_P$, then it belongs to $Q(\widetilde{X}^P)$. Suppose $u_0 \notin \Phi_P$. Then there exists a decomposition $u_0 = \left.q_1\right|_{P(u_1)P(v_1)}$. Repeating this process, we eventually obtain $\left.q_0\right|_{P(u_0)P(v_0)}=\left.q_k\right|_{P(u_k)P(v_k)}$ for $u_k,v_k\in\Phi_P$. Thus we conclude that~$Q(X^P) \subset Q(\widetilde{X}^P)$.

\begin{theorem}
\label{T:ResolutionRB} 
The algebra $\RBA{\lambda}{X_0}$ has the polygraphic resolution  $\mathrm{Sq}(\widetilde{X}^P)$, where
\begin{enumerate}
\item $\mathrm{Sq}_0(\widetilde{X}^P) \coloneq X_0$,
\item $\mathrm{Sq}_n(\widetilde{X}^P) \coloneq  \{\, P(u_1)|P(u_2)|\cdots|P(u_{n+1}) \;\big|\; u_i\in \Phi_P\,\}$, for every $n\geq1$.
\end{enumerate}
\end{theorem}
\begin{proof}
The proof follows from the fact that $\widetilde{X}^P$ is reduced and convergent, along with \eqref{T:Main conclusion}.
\end{proof}

\subsubsection{Low-dimensional generators of $\mathrm{Sq}(\widetilde{X}^P)$}
\label{SSS:InterpretationForRBResolution}
The following diagrams associate the $1$-generator $P(u_1) | P(u_2)$, the 2-generators $P(u_1) | P(u_2) | P(u_3)$, and the $3$-generators $P(u_1) | P(u_2) | P(u_3) | P(u_4)$ with higher critical branchings
\[
\xymatrix @C=3.5em {
\strut 
\ar@{..>} [r] _-{P(u_1)}
\ar@{..>}@/^7ex/ [rr] ^-{m}="a"
& \strut
\ar@{..>} [r] _{P(u_2)}
& \strut
\ar@1 "1,2";"a"!<0pt,7.5pt> |-{P(u_1)|P(u_2)}
}
\quad
\xymatrix @C=3.7em {
\strut 
\ar@{..>} [r] _-{P(u_1)}
\ar@{..>}@/^7ex/ [rr] ^-{m}="aw"
& \strut
\ar@{..>} [r] ^{P(u_2)}
\ar@{..>}@/_7ex/ [rr] _-{n}="ub"
& \strut
\ar@{..>}[r] ^-{P(u_3)}="w"
& \strut
\ar@1 "1,2";"aw"!<0pt,3pt> |-{P(u_1)|P(u_2)}
\ar@1 "1,3";"ub"!<0pt,7.5pt> |-{P(u_2)|P(u_3)}
}
\quad
\xymatrix @C=4em {
\strut 
\ar@{..>} [r] _-{P(u_1)}
\ar@{..>}@/^7ex/ [rr] ^-{m}="a"
& \strut
\ar@{..>} [r] ^{P(u_2)}
\ar@{..>}@/_7ex/ [rr] _-{n}="b"
& \strut
\ar@{..>}[r] ^-{P(u_3)}="w"
\ar@{..>}@/^7ex/ [rr] ^-{o}="c"
& \strut
\ar@{..>}[r] _-{P(u_4)}="s"
& \strut
\ar@1 "1,2";"a"!<0pt,3pt> |-{P(u_1)|P(u_2)}
\ar@1 "1,3";"b"!<0pt,7.5pt> |-{P(u_2)|P(u_3)}
\ar@1 "1,4";"c"!<0pt,3pt> |-{P(u_3)|P(u_4)}
}
\]
respectively, where $u_1,u_2,u_3,u_4\in \Phi_P$ and 
\[
\begin{aligned}
m&=P(\rep{P(u_1)u_2})+P(\rep{u_1P(u_2)})+\lambda P(\rep{u_1u_2}),\\
n&=P(\rep{P(u_2)u_3})+P(\rep{u_2P(u_3)})+\lambda P(\rep{u_2u_3}),\\
o&=P(\rep{P(u_3)u_4})+P(\rep{u_3P(u_4)})+\lambda P(\rep{u_3u_4}).
\end{aligned}
\]
We also illustrate their shapes in low dimensions.
For the $1$-generators $P(u_1) | P(u_2)$,
\[
P(u_1)|P(u_2):P(u_1)P(u_2)\fl \rep{P(u_1)P(u_2)}.
\]
By Theorem~\ref{T:Main conclusion}, the 2-generators $P(u_1) | P(u_2) | P(u_3)$ and the 3-generators $P(u_1) | P(u_2) | P(u_3) | P(u_4)$ have the following shapes
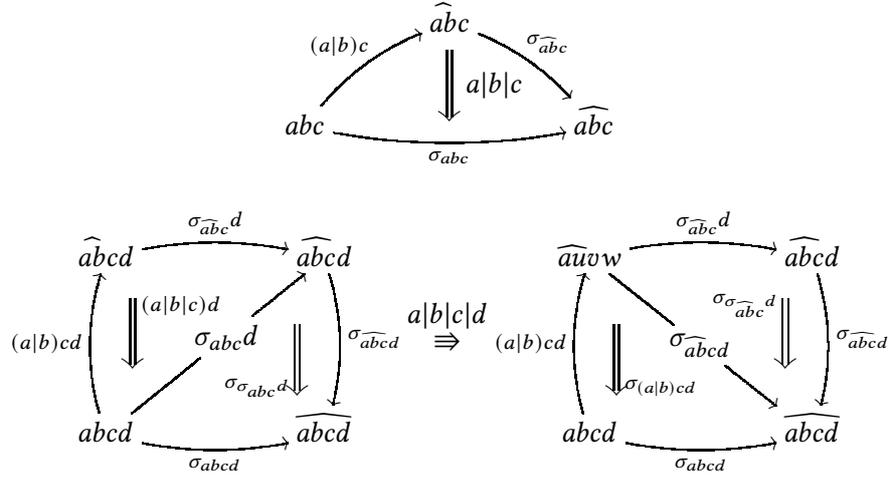
\[
\xymatrix @C=3em@R=2em {
& \rep{ab}c
\ar@/^/ [dr] ^-{\sigma_{\rep{ab}c}}
\ar@2 []!<0pt,-10pt>;[d]!<0pt,2.5pt> ^-*+{a|b|c}
\\
abc
\ar@/^/ [ur] ^-{(a|b)c}
\ar@/_/ [rr] _-{\sigma_{abc}}
&& {\rep{abc}}
}
\]
\[
\vcenter{\xymatrix @C=5em @R=4.5em {
{\rep{ab}cd}
	\ar@/^/ [r] ^-{\sigma_{\rep{ab}c}d}
	\ar@2 []!<10pt,-15pt>;[d]!<10pt,25pt> ^(0.1){(a| b| c)d}
& {\rep{abc}d}
	\ar@/^/ [d] ^-{\sigma_{\rep{abc}d}}
	\ar@2 []!<-10pt,-25pt>;[d]!<-10pt,15pt> _(0.9){\sigma_{\sigma_{abc}d}}
\\
abcd
	\ar@/^/ [u] ^-{(a| b)cd}
	\ar@//  [ur] |-*+{\sigma_{abc}d}
	\ar@/_/ [r] _-{\sigma_{abcd}}
& {\rep{abcd}}
}}
\!\!
\overset{\displaystyle a| b| c| d}{\Rrightarrow}
\!\!
\vcenter{\xymatrix @C=5em @R=4.5em {
{\rep{au}vw}
	\ar@/^/ [r] ^-{\sigma_{\rep{ab}c}d}
	\ar@// [dr] |-*+{\sigma_{\rep{ab}cd}}
	\ar@2 []!<10pt,-25pt>;[d]!<10pt,15pt> ^(0.9){\sigma_{(a| b)cd}}
& {\rep{abc}d}
	\ar@/^/ [d] ^-{\sigma_{\rep{abc}d}}
	\ar@2 []!<-10pt,-15pt>;[d]!<-10pt,25pt> _(0.1){\sigma_{\sigma_{\rep{ab}c}d}}
\\
abcd
	\ar@/^/ [u] ^-{(a| b)cd}
	\ar@/_/ [r] _-{\sigma_{abcd}}
& {\rep{abcd}}
}}
\]
where $a,b,c$ and $d$ correspond to $P(u_1)$, $P(u_2)$, $P(u_3)$, and $P(u_4)$, respectively.

\subsection{Polygraphic resolutions of free differential algebras}
\label{SS:ResolutionsOfFreeDifferentialAlgebras}
This subsection presents two polygraphic resolutions of the free differential algebra $\DA{\lambda}{X_0}$ on $X_0$.

\subsubsection{A reduced presentation} 
\label{SSS:ReducedPolygraphForDA}
The algebra $\DA{\lambda}{X_0}$ is presented by the convergent polygraph $X^D$, as defined in \eqref{E:FreeDifferentialAlgebra}. 
We construct a reduced \LO-1-polygraph $\widetilde{X}^D$ with
\begin{equation*}
\begin{aligned}
\widetilde{X_1}^D\:\coloneq\: \big\{ \alpha[u_1,u_2,\ldots,u_n]&:D(u_1u_2\cdots u_n)\rightarrow \sum_{\substack{1\leq i_1<\cdots< i_k\leq n \\ 1\leq k\leq n}} \lambda^{k-1}D_{i_1,\ldots, i_k}\left(u_1,\ldots ,u_n\right), \\
\varphi&:D(1)\fl0 \mid u_1,\ldots,u_n \in D^\theta(X_0)\setminus \{1\}, n\geq 2 \big\},
\end{aligned}
\end{equation*} 
where $D_{i_1,\ldots, i_k}\left(u_1,\ldots, u_n\right)\coloneq u_1\cdots D(u_{i_1})\cdots D(u_{i_k}) \cdots u_n$.

\begin{lemma} 
\label{L:TwoTietzeEquivalentPolyForDA}
The polygraphs $\widetilde{X}^D$ and $X^D$ are Tietze equivalent.
\end{lemma}
\begin{proof}
The polygraph $\widetilde{X}^D$ is  convergent since it contains no critical branchings. We prove this lemma in two steps.
First, we show that  $\Nf(\widetilde{X}^D)=\Nf(X^D)$. It suffices to prove that  $Q(X^{D})=Q(\widetilde{X}^D)$, where
\[
Q(X^D)\:\coloneq\:\{\left.q\right|_{D(uv)}|u,v\in X_0^\Omega\setminus \{1\}\} \quad\text{ and }\quad Q(\widetilde{X}^D)\:\coloneq\:\{\left.q\right|_{D(u_1 \cdots u_n)}|u_i\in D^\theta(X_0)\setminus \{1\},n\geq2\}.
\]
The inclusion $Q(\widetilde{X}^D)\subset Q(X^D)$ is straightforward.
For every $\left.q\right|_{D(u_1\cdots u_n)} \in Q(X^D)$ with $\operatorname{bre}(u_i) = 1$, if all $u_i\in D^\theta(X_0)\setminus \{1\}$, then it belongs to $Q(\widetilde{X}^D)$. If not, suppose $u_1 = D(v_1\cdots v_m) \in Q(X^D)$ with $\operatorname{bre}(v_i) = 1$, and repeat this process for $D(v_1\cdots v_m)$. Eventually, we obtain $\left.q\right|_{D(u_1\cdots u_n)} = \left.q'\right|_{D(w_1 \cdots w_k)}$, where $w_1, \dots, w_k \in D^\theta(X_0)\setminus \{1\}$, which shows that $ Q(X^D)\subset Q(\widetilde{X}^D)$. 

Next, we prove that $\Nf(w,X^D)=\Nf(w,\widetilde{X}^D)$, for every $w\in X_0^\Omega$. 
We first consider the case $w=D(u)$, with $u\in X_0^\Omega$. There exists a decomposition $D(u) = D(s_1 \cdots s_n)$, with $\operatorname{bre}(s_i) = 1$, and we proceed by induction on $\max(\operatorname{dep}(s_i))$. For the case $\max(\operatorname{dep}(s_i)) = 0$, meaning that all $s_i \in X_0$,
\[
D(u)=D(s_1\cdots s_n)\fl \sum_{\substack{1\leq i_1<\cdots< i_k\leq n \\ 1\leq k\leq n}} \lambda^{k-1}D_{i_1,\ldots, i_k}(s_1,\ldots ,s_n)
,
\]
in $X^D$.
It follows that $\Nf(D(u), X^D) = \Nf(D(u), \widetilde{X}^D)$.
Now, we assume that it holds for $\max(\operatorname{dep}(s_i))\leq m$. If $\max(\operatorname{dep}(s_i))=m+1$, then $s_j = D(v_j)$ with $\operatorname{dep}(v_j) = m$ for some $1\leq v_j\leq n$. By the induction hypothesis, $\Nf(v_j,X^D)=\Nf(v_j,\widetilde{X}^D)$ holds, implying that the normal forms of $D(u)$ are equal in both $X^D$ and $\widetilde{X}^D$.
So, for every $w = w_1 w_2 \cdots w_n \in X_0^\Omega$ with $\operatorname{bre}(w_i) = 1$, we have $\Nf(w_i, X^D) = \Nf(w_i, \widetilde{X}^D)$ for every $i$, and thus $\Nf(w, X^D) = \Nf(w, \widetilde{X}^D)$.
Hence, $\widetilde{X}^D$ and $X^D$ are Tietze equivalent.
\end{proof}

\begin{theorem}
\label{T:FirstResolutionOfDA}
The algebra $\DA{\lambda}{X_0}$ has the polygraphic resolution  $\mathrm{Sq}(\widetilde{X}^D)$, where
\begin{enumerate}
\item $\mathrm{Sq}_0(\widetilde{X}^D) \coloneq  X_0$,
\item $\mathrm{Sq}_1(\widetilde{X}^D) \coloneq  \{\, \epsilon|D(1), \epsilon | D(uv) \;\big|\; u,v\in (D^\theta(X_0))^\ast\setminus \{1\}\,\}$,
\item $\mathrm{Sq}_n(\widetilde{X}^D)$ is empty for every $n \geq 2$.
\end{enumerate}
\end{theorem}

If $\lambda \neq 0$, we provide another polygraphic resolution of $\DA{\lambda}{X_0}$, similar to $\mathrm{Sq}(\widetilde{X}^P)$.

\subsubsection{A reduced presentation for $\lambda \neq 0$}
When $\lambda\neq0$, we give another presentation $Y^{D}\coloneq\{X_0,Y_1^{D}\}$ of $\DA{\lambda}{X_0}$ with
\[
Y_1^{D}\:\coloneq\: \big\{ \alpha[u,v]: D(u)D(v)\fl \lambda^{-1}\big(D(uv)- D(u)v-u D(v)\big),\; \varphi:D(1)\fl 0 \mid u,v \in X_0^\Omega\setminus \{1\}\big\}.
\]
The termination of $Y^D$ follows from the decrease in the number of operators when the rules $\alpha[u,v]$ and~$\varphi$ are applied.
This polygraph is also confluent, as it has five families of critical branchings:
\begin{enumerate}
\item $ (D(u)\alpha[v,w], \alpha[u,v]D(w) )$, with the source $D(u)D(v)D(w)$, 
\item $ (D ( \left.q \right |_{\alpha[u,v]} )D(w), \alpha [ \left.q \right |_{D(u)D(v)}, w ] )$, with the source  $D ( \left.q \right |_{D(u)D(v)} )D(w)$, 
\item $ (D(u)D ( \left.q \right |_{\alpha[v,w]} ), \alpha [u,  \left.q \right |_{D(v)D(w)} ] )$,  with the source $D(u)D ( \left.q \right |_{D(v)D(w)} )$,
\item  $ (D ( \left.q \right |_{\varphi} )D(w), \alpha [ \left.q \right |_{D(1)}, w ] )$, with the source  $D ( \left.q \right |_{D(1)} )D(w)$,
\item $ (D(u)D ( \left.q \right |_{\varphi} ), \alpha [u,  \left.q \right |_{D(1)} ] )$,  with the source $D(u)D ( \left.q \right |_{D(1)} )$;
\end{enumerate}
all can be shown to be confluent by straightforward computation. Define the alternating product of $U$ and $V$ with operator $D$ as
\[
\Lambda_D(U, V) \:\coloneq\: \left(\bigcup_{r \geq 0}(U D(V))^r U\right) \cup \left(\bigcup_{r \geq 1}(U D(V))^r\right) \cup \left(\bigcup_{r \geq 0}(D(V) U)^r D(V)\right) \cup \left(\bigcup_{r \geq 1}(D(V) U)^r\right).
\]
We introduce the notation $\Phi_0 \coloneq X_0^\ast \setminus \{1\}$ and $\Phi_n \coloneq \Lambda_D(\Phi_0, \Phi_{n-1})$ for $n \geq 1$, and define the set
\[
\Phi_D \:\coloneq\: \left(\bigcup_{n \geq 0} \Phi_n\right) \cup \{1\}.
\] 
Thus, $\Nf(Y^{D})=\Phi_D$.
Note that the construction of $\Phi_D$ differs from that of $\Phi_P$ in \eqref{SSS:PolygraphForRB}, as~$\Phi_1 = \Lambda_D(\Phi_0, X_0^\ast \setminus \{1\})$, which implies $D(1) \notin \Phi_D$.
Next, we write $\rep{w}=\Nf(w,Y^D)$ for every $w\in X_0^\Omega$ and present a reduced convergent polygraph $\widetilde{Y}^D$, which is Tietze equivalent to $Y^D$, with
\[
\widetilde{Y}_1^D \:\coloneq\: \big\{ \alpha[u,v]: D(u)D(v)\rightarrow \lambda^{-1}\big(\rep{D(u)v}- \rep{uD(v)}-D(\rep{uv})\big),\;\varphi:D(1)\fl 0 \mid u, v \in \Phi_D\setminus \{1\} \big\}.
\]
By the same argument as for $\Nf(X^P)$ and $\Nf(\widetilde{X}^P)$ in \eqref{SSS:LemmaForReducedRB}, we have $\Nf(Y^{D}) = \Nf(\widetilde{Y}^D)=\Phi_D$.

\begin{theorem}
When $\lambda\neq0$, the algebra $\DA{\lambda}{X_0}$ has the polygraphic resolution  $\mathrm{Sq}(\widetilde{Y}^D)$, where 
\begin{enumerate}
\item $\mathrm{Sq}_0(\widetilde{Y}^D) \coloneq X_0$,
\item $\mathrm{Sq}_1(\widetilde{Y}^D) \coloneq  \{\, \epsilon|D(1), D(u_1)|D(u_2)\; \big| \; u_i\in \Phi_D\setminus \{1\}\,\}$,
\item $\mathrm{Sq}_n(\widetilde{Y}^D) \coloneq  \{\, D(u_1)|D(u_2)|\cdots|D(u_{n+1})\; \big| \; u_i\in \Phi_D\setminus~\{1\}\,\}$, for every $n \geq 2$.
\end{enumerate}
\end{theorem}

\subsection{Polygraphic resolutions of free differential Rota--Baxter algebras}
\label{SS:PolygraphicResolutionsDifferentialRotaBaxterAlgebras}

In this subsection, we consider the case with multiple operators and construct a polygraphic resolution of the free differential Rota--Baxter algebra $\DRBA{\lambda}{X_0}$ on $X_0$, assuming that $\lambda \neq 0$.

\subsubsection{A presentation of $\DRBA{\lambda}{X_0}$}
\label{SSS:DRBPresentation}
The algebra $\DRBA{\lambda}{X_0}$ is presented by
the \LO-1-polygraph $X$ with
\begin{equation*}
\begin{aligned}
X_1 \:\coloneq\: \{  \alpha[u,v]&:P(u)P(v) \rightarrow P(P(u)v) + P(uP(v)) + \lambda P(uv), \\
\beta[w_1,w_2]&: D(w_1) D(w_2)\rightarrow \lambda^{-1} \big(D(w_1w_2)- D(w_1)w_2- w_1 D(w_2)\big), \\
 \gamma[u]&:D(P(u)) \rightarrow u, \\
\varphi&:D(1)\rightarrow 0\mid u,v,w_1,w_2 \in X_0^\Omega \text{ and } w_1,w_2\neq1 \}.
\end{aligned}
\end{equation*} 
We set $N=\mathbb{Z}^2$ and define a derivation $d : X_0^\Omega \longrightarrow N$ as in~\eqref{SSS:PolygraphForRB}, by setting
\[
d(u) = \bigg( \operatorname{deg}_\Omega(u), \sum_{P|u} \big(\operatorname{deg}_\Omega(u) - \operatorname{deg}_\Omega(P)\big) 
+ \sum_{D|u} \big(\operatorname{deg}_\Omega(u) - \operatorname{deg}_\Omega(D)\big) \bigg),
\]
for every $u \in X_0^\Omega$.
Moreover, we define a bimodule structure by setting 
\begin{align*}
v\cdot(m_1,m_2)=(m_1 ,m_2)\cdot v &=(m_1,m_2+n m_1),\\
P((m_1, m_2))
=(m_1+1, m_2+m_1+1), &
\quad
D((m_1, m_2))
=(m_1+1, m_2+m_1+1),
\end{align*}
for all $(m_1,m_2)\in N$ and $v\in X_0^\Omega$ with $\operatorname{deg}_\Omega(v)=n$.
It follows that $d$ satisfies \eqref{E:Derivation}. By equipping $d(X_0^\Omega)$ with a monotone  lexicographic order, where $d(1) = (0,0)$ is the minimal element, we ensure that $X$ is terminating.

The critical branchings of $X$ are all confluent (see the Appendix), except for the following two cases
\[
\xymatrix @R=0.5em {
& uD(v) \quad\quad\quad\quad \quad\quad\quad\quad\quad\quad\quad\quad\quad\quad\quad\quad
\\
D(P(u))D(v)
\ar@/^3ex/ [ur] ^-{\gamma[u] D(v)}
\ar@/_3ex/ [dr] _-{\beta[P(u),v]}
\\
& \lambda^{-1} D(P(u)v)-\lambda^{-1} D(P(u))v-\lambda^{-1} P(u)D(v)
}
\]
\[
\xymatrix @R=0.5em {
& D(v)w \quad\quad\quad\quad \quad\quad\quad\quad\quad\quad\quad\quad\quad\quad\quad\quad
\\
D(v)D(P(w))
\ar@/^3ex/ [ur] ^-{D(v)\gamma[w]}
\ar@/_3ex/ [dr] _-{\beta[v,P(w)]}
\\
& \lambda^{-1} D(vP(w))-\lambda^{-1} D(v)P(w)-\lambda^{-1} vD(P(w)).
}
\]
where $u,v,w\in X_0^\Omega$ and $v\neq1$.
By the completion procedure in \eqref{SSS:CompletionProcedure} and the derivation $d$ above, we complete the polygraph $X$ into a convergent one, denoted $X^{PD}$, where
\[
\begin{aligned}
X_1^{PD} \:\coloneq\: \{  \alpha[u,v]&:P(u)P(v) \rightarrow P(P(u)v) + P(uP(v)) + \lambda P(uv), \\
\beta[w_1,w_2]&: D(w_1) D(w_2)\rightarrow \lambda^{-1} \big(D(w_1w_2)- D(w_1) w_2- w_1 D(w_2)\big), \\
 \gamma[u]&:D(P(u)) \rightarrow u, \\
 \delta_1[u,w_2]&:P(u)D(w_2) \rightarrow P(D(u)w_2) + uw_2 + \lambda uD(w_2),\\
 \delta_2[w_1,v]&:D(w_1)P(v) \rightarrow D(w_1P(v)) + w_1v + \lambda D(w_1)v,\\
\varphi&:D(1)\rightarrow 0\mid u,v,w_1,w_2 \in X_0^\Omega \text{ and } w_1,w_2\neq1 \}.
\end{aligned}
\]
We list all critical branchings of $X^{PD}$ in the Appendix. All are confluent.

\subsubsection{Normal forms}
\label{R:PDA3}
Let us denote $L^{ij}(u) \coloneq P^i\left(D^j(u)\right)$, for every $u \in X_0^\Omega$. 
We define $P^0(1) \coloneq 1$ and 
\[
\begin{aligned} 
\Lambda_{PD}(U, V) \: \coloneq \: & \left(\bigcup_{i, j, k\geq0;  r \geq 0}P^k(1)\left(U L^{ij}\left(V\right)\right)^r U P^k(1)\right) \cup\left(\bigcup_{ i, j, k \geq 0; r \geq 1}P^k(1)\left(U L^{ij}\left(V\right)\right)^r\right) \\ & \cup\left(\bigcup_{i, j, k \geq 0; r \geq 1}\left(L^{ij}\left(V\right) U\right)^r L^{ij}\left(V\right)\right) \cup\left(\bigcup_{i, j, k\geq 0; r \geq 1 }\left(L^{ij}\left(V\right) U\right)^r P^k(1)\right)
\end{aligned}
\]
We introduce the notation $X_0^+\coloneq X_0^\ast \setminus \{1\}$, and we set
\[
\Phi_0\:\coloneq\:\bigcup_{k\geq 0; r \geq 0}\left(X_0^+ P^k(1)\right)^r X_0^+
\]
Inductively, we define $\Phi_1 \coloneq \Lambda_{PD}(\Phi_0, X_0^+)$ and
 $\Phi_n \coloneq \Lambda_{PD}(\Phi_0, \Phi_{n-1})$ for $n > 1$. Finally, we  set
\[
\Phi\:\coloneq\:\bigcup_{n \geq 0} \Phi_n, \quad
\Phi_{PD}\:\coloneq\: L^{ij}(\Phi)\cup P^k(1)
\]
for all $i,j,k\geq 0$.
Thus, $\Nf(X^{PD}) = \Phi_{PD}$.
Here, the construction of $\Lambda_{PD}(U, V)$ differs from that of $\Lambda_{P}(U, V)$ and $\Lambda_{D}(U, V)$, as it is designed to exclude $\Omega$-monomials of the form $q|_{D(P(u))}$ from $\Phi_{PD}$.

In order to make the set of normal forms explicit, as in \eqref{D:PDAOfDiff} and \eqref{D:PDAOfRB}, 
we construct a PDA $\Ar^{\scriptscriptstyle PD}$ that accepts the monomials in~$\Phi_{PD}$
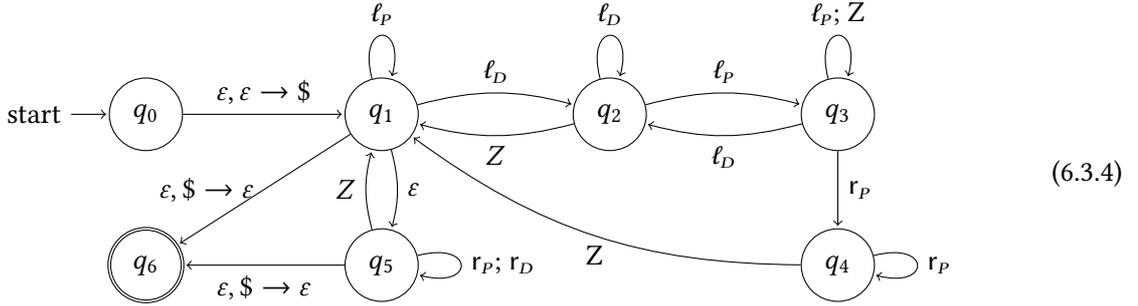
\begin{eqn}{equation}
\label{D:PDAOfDRB}
\raisebox{-1.7cm}{
\begin{tikzpicture}[shorten >=1pt, node distance=3cm, on grid, auto]
   % States
   \node[state, initial] (q_0) {$q_0$}; 
   \node[state] (q_1) [right=3.1cm of q_0] {$q_1$}; 
   \node[state] (q_2) [right=of q_1] {$q_2$}; 
   \node[state] (q_3) [right=of q_2] {$q_3$};
   \node[state] (q_4) [below=2cm of q_3] {$q_4$};
   \node[state] (q_5) [below=2cm of q_1] {$q_5$};
   \node[state,accepting] (q_6) [left=3.1cm of q_5] {$q_6$};

   % Loops
   \path[->]
   (q_2) edge[loop above] node {$\mathsf{\ell}_{\scriptscriptstyle D}
$} ()
   (q_3) edge[loop above] node {$\mathsf{\ell}_{\scriptscriptstyle P}
$; $X_0$} ()
   (q_4) edge[loop right] node {$\mathsf{r}_{\scriptscriptstyle P}
$} ()
   (q_5) edge[loop right] node[] {$\mathsf{r}_{\scriptscriptstyle P}
$; $\mathsf{r}_{\scriptscriptstyle D}
$} ()
   (q_1) edge[loop above] node {$\mathsf{\ell}_{\scriptscriptstyle P}
$} ();

   % Transitions
   \path[->]
   (q_0) edge[left] node[above] {$\epsilon, \epsilon \fl \$$} (q_1)
   (q_1) edge[bend left=15] node[above] {$\mathsf{\ell}_{\scriptscriptstyle D}
$} (q_2)
   (q_1) edge[bend left=15] node[] {$\epsilon$} (q_5)
   (q_5) edge[bend left=15] node[] {$X_0$} (q_1)
   (q_2) edge[bend left=15] node[above] {$\mathsf{\ell}_{\scriptscriptstyle P}
$} (q_3)
   (q_2) edge[bend left=15] node[] {$X_0$} (q_1)
   (q_3) edge[bend left=15] node[below] {$\mathsf{\ell}_{\scriptscriptstyle D}
$} (q_2)
   (q_3) edge[left] node[right] {$\mathsf{r}_{\scriptscriptstyle P}
$} (q_4)
   (q_4) edge[bend left=18] node[below] {$X_0$} (q_1)
   (q_5) edge[left] node[below] {$\epsilon, \$ \fl \epsilon $} (q_6)
   (q_1) edge[left] node[] {$\epsilon, \$ \fl \epsilon $} (q_6);
\end{tikzpicture}\raisebox{0.51cm}{\;\;.}
}
\end{eqn}
For simplicity, we display only the input symbols in certain instructions while omitting the corresponding stack operations
\[
\epsilon, \epsilon \fl \epsilon;\quad X_0, \epsilon \fl \epsilon;\quad \mathsf{\ell}_{\scriptscriptstyle D}
, \epsilon \fl D; \quad\mathsf{\ell}_{\scriptscriptstyle P}
, \epsilon \fl P;\quad \mathsf{r}_{\scriptscriptstyle P}
, P \fl \epsilon;\quad \mathsf{r}_{\scriptscriptstyle D}
, D \fl \epsilon.
\]
From the construction of $X_1^{PD}$, it suffices to modify $\Ar_\Omega$ to exclude monomials containing the subwords $\mathsf{r}_{\scriptscriptstyle P}
\mathsf{\ell}_{\scriptscriptstyle P}
, \mathsf{r}_{\scriptscriptstyle D}
\mathsf{\ell}_{\scriptscriptstyle D}
, \mathsf{r}_{\scriptscriptstyle P}
\mathsf{\ell}_{\scriptscriptstyle D}
, \mathsf{r}_{\scriptscriptstyle D}
\mathsf{\ell}_{\scriptscriptstyle P}
, \mathsf{\ell}_{\scriptscriptstyle D}
\mathsf{r}_{\scriptscriptstyle D}
$, and $\mathsf{\ell}_{\scriptscriptstyle D}
\mathsf{\ell}_{\scriptscriptstyle P}
 u\mathsf{r}_{\scriptscriptstyle P}
\mathsf{r}_{\scriptscriptstyle D}
$, where $u\in\{\mathsf{\ell}_{\scriptscriptstyle P}
,\mathsf{r}_{\scriptscriptstyle P}
,\mathsf{\ell}_{\scriptscriptstyle D}
,\mathsf{r}_{\scriptscriptstyle D}
,X_0\}^\ast$.

\subsubsection{A reduced presentation}
\label{SSS:ReducedDR}
We write $\rep{w}=\Nf(w,X^{PD})$ for every $w\in X_0^\Omega$ and construct a reduced presentation $\widetilde{X}^{PD} $, which is Tietze equivalent to $X^{PD}$, with
\[
\begin{aligned}
\widetilde{X}_1^{PD} \:\coloneq\: \{  \alpha[u,v]&:P(u)P(v) \rightarrow P(\rep{P(u)v}) + P(\rep{uP(v)}) + \lambda P(\rep{uv}), \\
\beta[u,v]&: D(w_1) D(w_2)\rightarrow \lambda^{-1} \big(\rep{D(w_1w_2)}- \rep{D(w_1) w_2}- \rep{w_1 D(w_2)}\big), \\
\gamma[u]&:D(P(u)) \rightarrow u, \\
 \delta_1[u,w_2]&:P(u)D(w_2) \rightarrow P(\rep{D(u)w_2}) + \rep{uw_2} + \lambda \rep{uD(w_2)},\\
 \delta_2[w_1,v]&:D(w_1)P(v) \rightarrow \rep{D(w_1P(v))} + \rep{w_1v} + \lambda \rep{D(w_1)v},\\
\varphi&:D(1)\rightarrow 0\mid u,v,w_1,w_2 \in \Phi_{PD} \text{ and } w_1,w_2\neq1 \}.
\end{aligned}
\]
It follows that $\Nf(X^{PD}) = \Nf(\widetilde{X}^{PD})=\Phi_{PD}$.

\begin{theorem}
When $\lambda\neq0$, the algebra $\DRBA{\lambda}{X_0}$ has the polygraphic resolution  $\mathrm{Sq}(\widetilde{X}^{PD})$, where
\begin{enumerate}
\item $\mathrm{Sq}_0(\widetilde{X}^{PD}) \coloneq  X_0$,
\item $\mathrm{Sq}_1(\widetilde{X}^{PD}) \coloneq  \{\,\epsilon|D(1),\; \epsilon|D(P(u_0)), \; R(u_1)|R(u_2) \;\big|\; R\in \{P,D\}, \; u_0, R(u_i)\in \Phi_{PD}\,\}$,
\item $\mathrm{Sq}_n(\widetilde{X}^{PD}) \coloneq  \{\,R(u_1)|R(u_2)|\cdots|R(u_{n+1})\; \big|\; R(u_i)\in\Phi_{PD}\,\}$, for every $n \geq 2$.
\end{enumerate}
\end{theorem}

\subsubsection*{Acknowledgements} 
This work was supported by the China Scholarship Council (CSC) under Grant No.~202406140026. 

\begin{small}
\renewcommand{\refname}{\Large\textsc{References}}
\bibliographystyle{plain}
\bibliography{biblioCURRENT}

\def\cprime{$'$}
\begin{thebibliography}{10}

\bibitem{Anick86}
David~J. Anick.
\newblock On the homology of associative algebras.
\newblock {\em Trans. Amer. Math. Soc.}, 296(2):641--659, 1986.

\bibitem{Polybook2025}
Dimitri Ara, Albert Burroni, Yves Guiraud, Philippe Malbos, Fran{\c{c}}ois
  M{\'e}tayer, and Samuel Mimram.
\newblock {\em Polygraphs: from rewriting to higher categories}, volume 495 of
  {\em Lond. Math. Soc. Lect. Note Ser.}
\newblock Cambridge: Cambridge University Press, 2025.

\bibitem{Atkinson1963}
Frederick~V. Atkinson.
\newblock Some aspects of {B}axter's functional equation.
\newblock {\em J. Math. Anal. Appl.}, 7:1--30, 1963.

\bibitem{Bai2007}
Chengming Bai.
\newblock A unified algebraic approach to the classical {Y}ang-{B}axter
  equation.
\newblock {\em J. Phys. A}, 40(36):11073--11082, 2007.

\bibitem{Batanin98}
M.~A. Batanin.
\newblock Computads for finitary monads on globular sets.
\newblock In {\em Higher category theory ({E}vanston, {IL}, 1997)}, volume 230
  of {\em Contemp. Math.}, pages 37--57. Amer. Math. Soc., Providence, RI,
  1998.

\bibitem{Baxter1960}
Glen Baxter.
\newblock An analytic problem whose solution follows from a simple algebraic
  identity.
\newblock {\em Pacific J. Math.}, 10:731--742, 1960.

\bibitem{BokutChenQiu}
Leonid~A. Bokut, Yuqun Chen, and Jianjun Qiu.
\newblock Gr\"obner--{S}hirshov bases for associative algebras with multiple
  operators and free {R}ota-{B}axter algebras.
\newblock {\em J. Pure Appl. Algebra}, 214(1):89--100, 2010.

\bibitem{Brown92}
Kenneth~S. Brown.
\newblock The geometry of rewriting systems: a proof of the
  {A}nick-{G}roves-{S}quier theorem.
\newblock In {\em Algorithms and classification in combinatorial group theory
  ({B}erkeley, {CA}, 1989)}, volume~23 of {\em Math. Sci. Res. Inst. Publ.},
  pages 137--163. Springer, New York, 1992.

\bibitem{Buchberger65}
Bruno Buchberger.
\newblock {\em Ein Algorithmus zum Auffinden der Basiselemente des
  Restklassenringes nach einem nulldimensionalen Polynomideal (An Algorithm for
  Finding the Basis Elements in the Residue Class Ring Modulo a Zero
  Dimensional Polynomial Ideal)}.
\newblock PhD thesis, Mathematical Institute, University of Innsbruck, Austria,
  1965.
\newblock English translation in J. of Symbolic Computation, Special Issue on
  Logic, Mathematics, and Computer Science: Interactions. Vol. 41, Number 3-4,
  Pages 475--511, 2006.

\bibitem{Buchberger06}
Bruno Buchberger.
\newblock An algorithm for finding the basis elements of the residue class ring
  of a zero dimensional polynomial ideal.
\newblock {\em J. Symbolic Comput.}, 41(3-4):475--511, 2006.
\newblock Translated from the 1965 German original by Michael P. Abramson.

\bibitem{Burroni93}
Albert Burroni.
\newblock Higher-dimensional word problems with applications to equational
  logic.
\newblock {\em Theoret. Comput. Sci.}, 115(1):43--62, 1993.
\newblock 4th Summer Conference on Category Theory and Computer Science (Paris,
  1991).

\bibitem{Cartier1972}
Pierre Cartier.
\newblock On the structure of free {B}axter algebras.
\newblock {\em Adv. Math.}, 9:253--265, 1972.

\bibitem{ChenGuoWangZhou}
Jun Chen, Li~Guo, Kai Wang, and Guodong Zhou.
\newblock Koszul duality, minimal model and {$L_{\infty }$}-structure for
  differential algebras with weight.
\newblock {\em Adv. Math.}, 437:Paper No. 109438, 41, 2024.

\bibitem{ChenChenLi}
Yuqun Chen, Yongshan Chen, and Yu~Li.
\newblock Composition-diamond lemma for differential algebras.
\newblock {\em Arab. J. Sci. Eng. Sect. A Sci.}, 34(2):135--145, 2009.

\bibitem{ConnesKreimer2000}
Alain Connes and Dirk Kreimer.
\newblock Renormalization in quantum field theory and the {R}iemann-{H}ilbert
  problem. {I}. {T}he {H}opf algebra structure of graphs and the main theorem.
\newblock {\em Comm. Math. Phys.}, 210(1):249--273, 2000.

\bibitem{KuruschGuo}
Kurusch Ebrahimi-Fard and Li~Guo.
\newblock Rota-{B}axter algebras and dendriform algebras.
\newblock {\em J. Pure Appl. Algebra}, 212(2):320--339, 2008.

\bibitem{GGR15}
Xing Gao, Li~Guo, and Markus Rosenkranz.
\newblock Free integro-differential algebras and {G}r\"obner--{S}hirshov bases.
\newblock {\em J. Algebra}, 442:354--396, 2015.

\bibitem{GaoGuoSit}
Xing Gao, Li~Guo, William~Y. Sit, and Shanghua Zheng.
\newblock {R}ota-{B}axter type operators, rewriting systems and
  {G}r\"obner--{S}hirshov bases.
\newblock arXiv:1412.8055, 2014.

\bibitem{Groves90}
John R.~J. Groves.
\newblock Rewriting systems and homology of groups.
\newblock In {\em Groups---{C}anberra 1989}, volume 1456 of {\em Lecture Notes
  in Math.}, pages 114--141. Springer, Berlin, 1990.

\bibitem{Guiraud06jpaa}
Yves Guiraud.
\newblock Termination orders for three-dimensional rewriting.
\newblock {\em J. Pure Appl. Algebra}, 207(2):341--371, 2006.

\bibitem{GuiraudHoffbeckMalbos19}
Yves Guiraud, Eric Hoffbeck, and Philippe Malbos.
\newblock Convergent presentations and polygraphic resolutions of associative
  algebras.
\newblock {\em Math. Z.}, 293(1-2):113--179, 2019.

\bibitem{GuiraudMalbos09}
Yves Guiraud and Philippe Malbos.
\newblock Higher-dimensional categories with finite derivation type.
\newblock {\em Theory Appl. Categ.}, 22:No. 18, 420--478, 2009.

\bibitem{GuiraudMalbos12advances}
Yves Guiraud and Philippe Malbos.
\newblock Higher-dimensional normalisation strategies for acyclicity.
\newblock {\em Adv. Math.}, 231(3-4):2294--2351, 2012.

\bibitem{GuiraudMalbos18}
Yves Guiraud and Philippe Malbos.
\newblock Polygraphs of finite derivation type.
\newblock {\em Math. Structures Comput. Sci.}, 28(2):155--201, 2018.

\bibitem{Guo09}
Li~Guo.
\newblock Operated semigroups, {M}otzkin paths and rooted trees.
\newblock {\em J. Algebraic Combin.}, 29(1):35--62, 2009.

\bibitem{Guo10}
Li~Guo.
\newblock Algebraic {B}irkhoff decomposition and its applications.
\newblock In {\em Automorphic forms and the {L}anglands program}, volume~9 of
  {\em Adv. Lect. Math. (ALM)}, pages 277--319. Int. Press, Somerville, MA,
  2010.

\bibitem{GGL22}
Li~Guo, Richard Gustavson, and Yunnan Li.
\newblock An algebraic study of {V}olterra integral equations and their
  operator linearity.
\newblock {\em J. Algebra}, 595:398--433, 2022.

\bibitem{GKeigher08}
Li~Guo and William Keigher.
\newblock On differential {R}ota-{B}axter algebras.
\newblock {\em J. Pure Appl. Algebra}, 212(3):522--540, 2008.

\bibitem{GLS21}
Li~Guo, Honglei Lang, and Yunhe Sheng.
\newblock Integration and geometrization of {R}ota-{B}axter {L}ie algebras.
\newblock {\em Adv. Math.}, 387:Paper No. 107834, 34, 2021.

\bibitem{GuoSitZhang}
Li~Guo, William~Y. Sit, and Ronghua Zhang.
\newblock Differential type operators and {G}r\"obner--{S}hirshov bases.
\newblock {\em J. Symbolic Comput.}, 52:97--123, 2013.

\bibitem{KnuthBendix70}
Donald Knuth and Peter Bendix.
\newblock Simple word problems in universal algebras.
\newblock In {\em Computational {P}roblems in {A}bstract {A}lgebra ({P}roc.
  {C}onf., {O}xford, 1967)}, pages 263--297. Pergamon, Oxford, 1970.

\bibitem{Kobayashi90}
Yuji Kobayashi.
\newblock Complete rewriting systems and homology of monoid algebras.
\newblock {\em J. Pure Appl. Algebra}, 65(3):263--275, 1990.

\bibitem{Kobayashi05}
Yuji Kobayashi.
\newblock Gr\"obner bases of associative algebras and the {H}ochschild
  cohomology.
\newblock {\em Trans. Amer. Math. Soc.}, 357(3):1095--1124 (electronic), 2005.

\bibitem{Kolchin1973}
E.~R. Kolchin.
\newblock {\em Differential algebra and algebraic groups}, volume~54 of {\em
  Pure Appl. Math., Academic Press}.
\newblock Academic Press, New York, NY, 1973.

\bibitem{LiGuo}
Yunnan Li and Li~Guo.
\newblock Construction of free differential algebras by extending
  {G}r\"obner--{S}hirshov bases.
\newblock {\em J. Symbolic Comput.}, 107:167--189, 2021.

\bibitem{LQQZ}
Zuan Liu, Zihao Qi, Yufei Qin, and Guodong Zhou.
\newblock Gr\"obner-{S}hirshov bases for free multi-operated algebras over
  algebras.
\newblock {\em J. Symbolic Comput.}, 133:Paper No. 102489, 29, 2026.

\bibitem{Magid1994}
Andy~R. Magid.
\newblock {\em Lectures on differential {G}alois theory}, volume~7 of {\em
  University Lecture Series}.
\newblock American Mathematical Society, Providence, RI, 1994.

\bibitem{MalbosRen23}
Philippe Malbos and Isaac Ren.
\newblock Shuffle polygraphic resolutions for operads.
\newblock {\em J. Lond. Math. Soc. (2)}, 107(1):61--122, 2023.

\bibitem{QQWZ21}
Zihao Qi, Yufei Qin, Kai Wang, and Guodong Zhou.
\newblock Free objects and {G}r\"obner--{S}hirshov bases in operated contexts.
\newblock {\em J. Algebra}, 584:89--124, 2021.

\bibitem{Ritt1932}
Joseph~Fels Ritt.
\newblock {\em Differential equations from the algebraic standpoint.},
  volume~14 of {\em Colloq. Publ., Am. Math. Soc.}
\newblock American Mathematical Society (AMS), Providence, RI, 1932.

\bibitem{Ritt1950}
Joseph~Fels Ritt.
\newblock {\em Differential algebra}, volume~33 of {\em Colloq. Publ., Am.
  Math. Soc.}
\newblock American Mathematical Society (AMS), Providence, RI, 1950.

\bibitem{Rota1969}
Gian-Carlo Rota.
\newblock Baxter algebras and combinatorial identities. {I}, {II}.
\newblock {\em Bull. Amer. Math. Soc.}, 75:325--329; 75 (1969), 330--334, 1969.

\bibitem{Shirshov}
Anatoly~I. Shirshov.
\newblock Some algorithmic problems for {$\varepsilon $}-algebras.
\newblock {\em Sibirsk. Mat. \v Z.}, 3:132--137, 1962.

\bibitem{SongWangZhang}
Chao Song, Kai Wang, and Yuanyuan Zhang.
\newblock Deformations and cohomology theory of {$\Omega$}-{R}ota-{B}axter
  algebras of arbitrary weight.
\newblock {\em J. Geom. Phys.}, 201:Paper No. 105217, 22, 2024.

\bibitem{Squier87}
Craig~C. Squier.
\newblock Word problems and a homological finiteness condition for monoids.
\newblock {\em J. Pure Appl. Algebra}, 49(1-2):201--217, 1987.

\bibitem{Street76}
Ross Street.
\newblock Limits indexed by category-valued {$2$}-functors.
\newblock {\em J. Pure Appl. Algebra}, 8(2):149--181, 1976.

\bibitem{Ufnarovski95}
Victor~A. Ufnarovskij.
\newblock Combinatorial and asymptotic methods in algebra.
\newblock In {\em Algebra, {VI}}, volume~57 of {\em Encyclopaedia Math. Sci.},
  pages 1--196. Springer, Berlin, 1995.

\bibitem{MariusMichael}
Marius van~der Put and Michael~F. Singer.
\newblock {\em Galois theory of linear differential equations}, volume 328 of
  {\em Grundlehren Math. Wiss.}
\newblock Berlin: Springer, 2003.

\bibitem{WangZhou24}
Kai Wang and Guodong Zhou.
\newblock The minimal model of {R}ota-{B}axter operad with arbitrary weight.
\newblock {\em Selecta Math. (N.S.)}, 30(5):Paper No. 99, 44, 2024.

\end{thebibliography}
\end{small}

\clearpage
\section*{Appendix}
The sources of all critical branchings of the polygraph $X^{PD}$ in \eqref{SSS:DRBPresentation} are listed below. Here, we denote by $\alpha \wedge \beta$ the set of sources for critical branchings of the form $(\alpha[u, v], \beta[u, v])$, with similar conventions for other notations. This enumeration is similar to that presented in {\cite[Thm.~3.7]{LQQZ}}, which studies the theory of GS bases for free differential Rota--Baxter algebras. For all $u,v,w\in X_0^\Omega$, $s,t,r\in X_0^\Omega\setminus \{1\}$, and $q \in X_0^\Omega[\square]$, we have the following:
\begin{itemize}
	\item [$\alpha \wedge \alpha$] \quad $P(u)P(v)P(w)$, \quad $P\left(\left.q\right|_{P(u)P(v)} \right)P\left(w\right)$, \quad  $P\left(u\right) P\left(\left.q\right|_{P(v)P(w)}\right)$,

	\item [$\alpha \wedge \beta$] \quad $P\left(\left.q\right|_{D(s)D(t)} \right)P\left(u\right)$, \quad $P\left(u\right) P\left(\left.q\right|_{D(s)D(t)}\right)$,

	\item [$\alpha \wedge \gamma$] \quad $P\left(\left.q\right|_{D(P(u))} \right)P\left(v\right)$, \quad $P\left(u\right) P\left(\left.q\right|_{D(P(v))}\right)$,

		\item [$\alpha \wedge \delta_1$] \quad $P(u)P(v)D(s)$,\quad $P\left(\left.q\right|_{P(u)D(s)} \right)P\left(v\right)$, \quad $P\left(u\right) P\left(\left.q\right|_{P(v)D(s)}\right)$,

		\item [$\alpha \wedge \delta_2$] \quad $P\left(\left.q\right|_{D(s)P(u)} \right)P\left(v\right)$, \quad $P\left(u\right) P\left(\left.q\right|_{D(s)P(v)}\right)$,

		\item [$\alpha \wedge \varphi$] \quad $P\left(\left.q\right|_{D(1)} \right)P\left(w\right)$, \quad $P\left(u\right) P\left(\left.q\right|_{D(1)}\right)$,

\item [$\beta \wedge \alpha$] \quad $D\left(\left.q\right|_{P(u)P(v)}\right)D\left(s\right)$,\quad $D\left(s\right)D\left(\left.q\right|_{P(u)P(v)}\right)$,
	
	\item [$\beta \wedge \beta$] \quad$D(s)D(t)D(r)$ , \quad $D\left(\left.q\right|_{D(s)D(t)}\right)D\left(r\right)$,\quad $D\left(s\right)D\left(\left.q\right|_{D(t)D(r)}\right)$,

		\item [$\beta \wedge \gamma$] \quad $D(P(u))D(s)$, \quad $D(s)D(P(u))$,\quad $D\left(\left.q\right|_{D(P(u))}\right)D\left(s\right)$,\quad $D\left(s\right)D\left(\left.q\right|_{D(P(u))}\right)$,

		\item [$\beta \wedge \delta_1$] \quad  $D\left(\left.q\right|_{P(u)D(s)}\right)D\left(t\right)$,\quad $D\left(s\right)D\left(\left.q\right|_{P(u)D(t)}\right)$,

		\item [$\beta \wedge \delta_2$] \quad $D(s)D(t)P(u)$,\quad  $D\left(\left.q\right|_{D(s)P(u)}\right)D\left(t\right)$,\quad $D\left(s\right)D\left(\left.q\right|_{D(t)P(u)}\right)$,

		\item [$\beta \wedge \varphi$] \quad  $D\left(\left.q\right|_{D(1)}\right)D\left(s\right)$,\quad $D\left(s\right)D\left(\left.q\right|_{D(1)}\right)$,

		\item [$\gamma \wedge \alpha$] \quad $D\left(P\left(\left.q\right|_{P(u)P(v)}\right)\right)$,
	
		\item [$\gamma \wedge \beta$] \quad $D\left(P\left(\left.q\right|_{D(s)D(t)}\right)\right)$,

		\item [$\gamma \wedge \gamma$] \quad  $D\left(P\left(\left.q\right|_{D(P(u))}\right)\right)$,

		\item [$\gamma \wedge \delta_1$] \quad $D\left(P\left(\left.q\right|_{P(u)D(s)}\right)\right)$,

		\item [$\gamma \wedge \delta_2$] \quad $D\left(P\left(\left.q\right|_{D(s)P(u)}\right)\right)$,

		\item [$\gamma \wedge \varphi$] \quad $D\left(P\left(\left.q\right|_{D(1)}\right)\right)$,

		\item [$\delta_1 \wedge \alpha$] \quad $P\left(\left.q\right|_{P(u)P(v)}\right)D\left(s\right)$,\quad $P\left(u\right)D\left(\left.q\right|_{P(v)P(w)}\right)$,

\item [$\delta_1 \wedge \beta$] \quad $P(u)D(s)D(t)$, \quad $P\left(\left.q\right|_{D(s)D(t)}\right)D\left(u\right)$,\quad $P\left(u\right)D\left(\left.q\right|_{D(s)D(t)}\right)$,

    \item [$\delta_1 \wedge \gamma$] \quad $P(u)D(P(v))$, \quad $P\left(\left.q\right|_{D(P(u))}\right)D\left(s\right)$,\quad $P\left(u\right)D\left(\left.q\right|_{D(P(v))}\right)$,

        \item [$\delta_1 \wedge \delta_1$] \quad $P\left(\left.q\right|_{P(u)D(s)}\right)D\left(t\right)$,\quad $P\left(u\right)D\left(\left.q\right|_{P(v)D(s)}\right)$,

	  \item [$\delta_1 \wedge \delta_2$] \quad $P(u)D(s)P(v)$, \quad $P\left(\left.q\right|_{D(s)P(u)}\right)D\left(t\right)$,\quad $P\left(u\right)D\left(\left.q\right|_{D(s)P(v)}\right)$,

	\item [$\delta_1 \wedge \varphi$] \quad $P\left(\left.q\right|_{D(1)}\right)D\left(s\right)$,\quad $P\left(u\right)D\left(\left.q\right|_{D(1)}\right)$,

		\item [$\delta_2 \wedge \alpha$] \quad $D(s)P(u)P(v)$, \quad $D\left(\left.q\right|_{P(u)P(v)}\right)P\left(w\right)$,\quad $D\left(s\right)P\left(\left.q\right|_{P(u)P(v)}\right)$,

		\item [$\delta_2 \wedge \beta$] \quad $D\left(\left.q\right|_{D(s)D(t)}\right)P\left(u\right)$,\quad $D\left(s\right)P\left(\left.q\right|_{D(t)D(r)}\right)$,

		\item [$\delta_2 \wedge \gamma$] \quad $D(P(u))P(v)$,\quad $D\left(\left.q\right|_{D(P(u))}\right)P\left(v\right)$,\quad $D\left(s\right)P\left(\left.q\right|_{D(P(u))}\right)$,

\item [$\delta_2 \wedge \delta_1$] \quad $D(s)P(u)D(t)$,\quad $D\left(\left.q\right|_{P(u)D(s)}\right)P\left(v\right)$,\quad $D\left(s\right)P\left(\left.q\right|_{P(u)D(t)}\right)$,

\item [$\delta_2 \wedge \delta_2$] \quad $D\left(\left.q\right|_{D(s)P(u)}\right)P\left(v\right)$,\quad $D\left(s\right)P\left(\left.q\right|_{D(t)P(u)}\right)$,

\item [$\delta_2 \wedge \varphi$] \quad $D\left(\left.q\right|_{D(1)}\right)P\left(u\right)$,\quad $D\left(s\right)P\left(\left.q\right|_{D(1)}\right)$.
 \end{itemize}

\quad

\vfill

\begin{footnotesize}

\bigskip
\auteur{Zuan Liu}{zliu@math.univ-lyon1.fr}
{Université Claude Bernard Lyon 1\\
CNRS, Institut Camille Jordan, UMR5208\\
43, boulevard du 11 novembre 1918\\
F-69622 Villeurbanne cedex, France}

\bigskip
\auteur{Philippe Malbos}{malbos@math.univ-lyon1.fr}
{Université Claude Bernard Lyon 1\\
CNRS, Institut Camille Jordan, UMR5208\\
43, boulevard du 11 novembre 1918\\
F-69622 Villeurbanne cedex, France}
\end{footnotesize}

\vspace{1.5cm}

\begin{small}---\;\;\today\;\;-\;\;\hhmm\;\;---\end{small}
\end{document}